\documentclass[11pt, a4paper, reqno]{amsart}
\usepackage{amsmath,amsfonts, amsthm,amssymb,amsfonts,amscd, graphicx, hyperref, enumerate,mathtools, mathrsfs, dsfont, upgreek, appendix,xcolor, xfrac}

\newtheorem{theorem}{Theorem}[section]
\newtheorem{corollary}[theorem]{Corollary}
\newtheorem{lemma}[theorem]{Lemma}
\newtheorem{definition}[theorem]{Definition}
\newtheorem{prop}[theorem]{Proposition}
\theoremstyle{remark}
\newtheorem{remark}[theorem]{Remark}
\usepackage[left=2.5cm, right=2.5 cm]{geometry}

\newcommand{\eps}{\varepsilon} 
\makeatletter

\hypersetup{
  colorlinks   = true, 
  urlcolor     = black, 
  linkcolor    = black, 
  citecolor    = black 
}

\mathtoolsset{showonlyrefs=true}

\newcommand{\norm}[1]{\lVert#1\rVert} 
\DeclareMathOperator*{\supp}{supp} 
\DeclareMathOperator{\sign}{sgn} 

\newcommand*{\R}{\ensuremath{\mathbb{R}}}

\newcommand*{\N}{\ensuremath{\mathbb{N}}}

\def\Xint#1{\mathchoice
{\XXint\displaystyle\textstyle{#1}}%
{\XXint\textstyle\scriptstyle{#1}}%
{\XXint\scriptstyle\scriptscriptstyle{#1}}%
{\XXint\scriptscriptstyle\scriptscriptstyle{#1}}%
\!\int}
\def\XXint#1#2#3{{\setbox0=\hbox{$#1{#2#3}{\int}$ }
\vcenter{\hbox{$#2#3$ }}\kern-.6\wd0}}

\def\dashint{\Xint-}

\newcommand{\mean}[1]{\,-\hskip-1.08em\int_{#1}}
\renewcommand{\mean}[1]{\dashint_{#1}}

\newcommand{\new}{}
\newcommand{\cor}{}

\AtBeginDocument{
\let\div\undefined\DeclareMathOperator{\div}{div} 
}
\AtBeginDocument{
}
\makeindex

\author[Colombo]{Maria Colombo}
\address{EPFL SB, Station 8, 
CH-1015 Lausanne, Switzerland
}
\email{maria.colombo@epfl.ch}

\author[Haffter]{Silja Haffter
}
\address{EPFL SB, Station 8, 
CH-1015 Lausanne, Switzerland
}
\email{silja.haffter@epfl.ch}

\begin{document}
\title[]{Estimate on the dimension of the singular set of the supercritical surface quasigeostrophic equation}
\begin{abstract}
We consider the SQG equation with dissipation given by a fractional Laplacian of order $\alpha<\frac{1}{2}$. We introduce a notion of suitable weak solution, which exists for every $L^2$ initial datum, and we prove that for such solution the singular set is contained in a compact set in spacetime of {\new Hausdorff dimension at most $\frac{1}{2\alpha} \left( \frac{1+\alpha}{\alpha} (1-2\alpha) + 2\right)$.} 
\end{abstract}
\maketitle

\tableofcontents

\section{Introduction}
For $ \alpha \in \big(0, \frac{1}{2}\big]$ we consider the following fractional drift-diffusion equation
\begin{equation}\label{eq:SQG}
\begin{cases}
\partial_t \theta + u \cdot \nabla \theta = - (-\Delta)^\alpha \theta \\
\div u =  0 \, ,
\end{cases}
\end{equation}
where $\theta: \R^2 \times [0, \infty) \rightarrow \R$ is an active scalar, $u: \R^2 \times [0, \infty) \rightarrow \R^2$ is the velocity field and $(-\Delta)^\alpha$ corresponds to the Fourier multiplier with symbol $\lvert \xi \rvert^{2\alpha}$. The system is usually complemented with the initial condition 
\begin{equation}\label{eq:Cauchy}
\theta (\ \cdot \ , 0) = \theta_0\, .
\end{equation}
We will be particularly interested in the surface quasigeostrophic (SQG) equation where the velocity field $u$ is determined from $\theta$ by the Riesz-transform $\mathcal{R}$ on $\R^2$. More precisely, we require
\begin{equation}\label{eq:u}
 u = \nabla^\perp (-\Delta)^{-\frac{1}{2}} \theta = \mathcal{R}^\perp \theta\,.
\end{equation}
There is a natural scaling invariance associated to the system: whenever $(\theta, u)$ solves \eqref{eq:SQG}, then so does the pair 
\begin{equation}\label{eq:rescaling}
\theta_r(x,t) := r^{2\alpha-1} \theta(rx, r^{2\alpha}t) \qquad u_r(x,t)= r^{2\alpha-1} u(rx, r^{2\alpha}t)\,.
\end{equation}

\subsection{Main result} Our main result shows that for every $L^2$ initial datum and every $\alpha \in \big[\frac{9}{20}, \frac{1}{2}\big),$ there exists an almost everywhere smooth solution of the SQG equation and, more precisely, it provides a bound on the 
box-counting and Hausdorff dimension of the closed set of its singular points. 
\begin{theorem}\label{thm:main}
Let $\alpha_0:=  \frac{1+\sqrt{33}}{16}$. For any $\alpha \in \big(\alpha_0, \frac 12\big)$ and any initial datum $\theta_0\in L^2(\R^2)$ there is a Leray--Hopf weak solution $(\theta, u)$ of \eqref{eq:SQG}--\eqref{eq:u} (see Definition \ref{def:LH}) and a relatively closed set ${\rm Sing}\, \theta  \subset \mathbb R^2 \times (0, \infty)$ such that
\begin{itemize}
\item $\theta \in C^{\infty}\big([\R^2 \times (0, \infty) ]\setminus {\rm Sing} \,\theta\big)$,
\item ${\rm Sing} \, \theta \,\cap \, [\mathbb R^2 \times [t, \infty)]$ is compact {\new with box-counting dimension at most $\frac{1}{2\alpha} \left( \frac{1+\alpha}{\alpha} (1-2\alpha) + 2\right)$} for any $t>0\,,$
\item {\new the Hausdorff dimension of ${\rm Sing}\, \theta$ does not exceed  $\frac{1}{2\alpha} \left( \frac{1+\alpha}{\alpha} (1-2\alpha) + 2\right).$}

\end{itemize}
\end{theorem}


\begin{remark} We will in fact prove a slightly stronger statement, namely that all suitable weak solutions $\theta$ of \eqref{eq:SQG}--\eqref{eq:u} on $\mathbb R^2 \times (0, \infty)$ (see Definition \ref{def:sws}) satisfy the estimate on the dimension of the spacetime singular set ${\rm Sing} \, \theta$; in particular, they are smooth almost everywhere in spacetime. Moreover, the set ${\rm Sing } \, \theta $ is compact as soon as the initial datum is regular enough to guarantee local smooth existence.
\end{remark}

The regularity issue for the equation \eqref{eq:SQG}--\eqref{eq:u} is fully understood only in the subcritical and critical regime, namely for $\alpha \geq \frac{1}{2}$. The critical case (without bounderies) is now well-understood thanks to Kiselev, Nazarov and Volberg \cite{KiselevNazarovVolberg2007} and Caffarelli and Vasseur \cite{CaffarelliVasseur2010} (see also \cite{ConstantinVicol2012}) and one even has a description of the long time behaviour of the system \cite{ConstantinTarfuleaVicol2015,ConstantinCotiZelatiVicol2016}. On bounded domains, the critical case has been well-studied in a series of works initiated by \cite{ConstantinIgnatova2016}.
{In the supercritical range $\alpha < \frac{1}{2}\,,$ the global regularity of Leray-Hopf weak solutions to the SQG equation is an open problem related to the problem of global existence of classical solutions: in fact, it is well-known that Leray-Hopf weak solutions coincide with classical solutions as long as the latter exist.}
Constantin and Wu \cite{ConstantinWu2008,ConstantinWu2009} obtained partial results by extending the program of \cite{CaffarelliVasseur2010} to the supercritical regime. In \cite{CaffarelliVasseur2010} the technique of De Giorgi for uniformly elliptic equations with measurable coefficients is adapted to prove the smoothness of Leray-Hopf weak solutions in three steps: the local boundedness of $L^2$ solutions, the H\"older continuity of $L^\infty$ solutions, and the smoothness of H\"older solutions. While the $L^\infty$-bound for Leray-Hopf weak solutions still works in the supercritical case \cite{ConstantinWu2008}, only conditional regularity results are kown regarding the second and third step of the scheme. For instance, H\"older solutions in $C^\delta$ are smooth for $\delta>1-2\alpha$, while for $\delta<1-2\alpha$ this is left open \cite{ConstantinWu2009}. On the negative side, \cite{BuckmasterShkollerVicol2019} established non-uniqueness of a class of (very) weak-solution for the system \eqref{eq:SQG}--\eqref{eq:u}, even for subcritical dissipations. {\new In this context Theorem \ref{thm:main} is, to our knowledge, the first  a.e. smoothness / partial regularity result.} 
\bigskip

{\new  
\subsection{An $\eps$-regularity theorem} The estimate on the dimension of the singular set in Theorem \ref{thm:main} follows from a simple covering argument and a so-called $\eps$-regularity result: in order to fix the main ideas, we present the latter in a simplified version in Theorem \ref{prop:ereg2} below. 
In what follows we denote by $\theta^*$ the Caffarelli-Silvestre extension of $\theta$ and by $\mathcal{M}$ the maximal function with respect to the space variable (see Sections \ref{sec:CS} and \ref{sec:maximal}); $K_q$, as defined in \eqref{eq:Kq}, is a constant depending on the local-in-time $L^\infty_t L^q_x(\R^2)$ estimate of $\theta$ (recalled in Section \ref{s:LH}).
\begin{theorem}\label{prop:ereg2} Let $\alpha \in \big[\alpha_0, \frac{1}{2}\big)$, $ q \geq 8$ and $p:=\frac{1+\alpha}{\alpha} + \frac{1}{q}\,.$ There exists a universal $\varepsilon= \varepsilon(\alpha)>0$ such that the following holds: Let $(\theta, u)$ be a suitable weak solution of \eqref{eq:SQG}--\eqref{eq:u} on  $\R^2 \times (0,T)\,$ (see Definition \ref{def:sws}) satisfying 
\begin{align}\label{eq:blabla}
\frac{\norm{\theta}_{L^\infty(\R^2 \times [t-r^{2\alpha}, t+r^{2\alpha}])}^{p-2}}{r^{p(1-2\alpha)+2}} \Bigg( &\int_{\mathcal C^*(x,t;r)} y^b\lvert \overline \nabla \theta^* \rvert^2  \, dz\, ds \, dy +\int_{\mathcal C(x,t;r)} \mathcal{M}\left((D_{\alpha,2}\, \theta)^2\right) \, dz \, ds \Bigg) \leq \varepsilon\,,
\end{align}
where $\mathcal C(x,t;r):=B_{K_q r^{2\alpha-2/q}}(x) \times (t-r^{2\alpha}, t+r^{2\alpha})\,,$ $\mathcal C^*(x,t;r):=[0,r)\times \mathcal C(x,t;r)$ and
\begin{equation}
\label{eqn:Dalpha2}
(D_{\alpha, 2} \, \theta) (z, s):=\left( \int_{\R^2} \frac{\lvert \theta(z, s)-\theta(z',s) \rvert^2}{\lvert z-z' \rvert^{2+2\alpha}} \, dz' \right)^{\frac{1}{2}}\,.
\end{equation} Then $\theta$ is smooth on $B_{r/8}(x) \times (t-r^{2\alpha}/16,t + r^{2\alpha}/16)\,.$
\end{theorem}
The integral quantities present in \eqref{eq:blabla} are two non-equivalent localized versions of the dissipative part of the energy, i.e. the $L^2((0,T), W^{\alpha,2}(\R^2))$-norm of  $\theta$, and are globally controlled through the latter. At this point, the careful reader will object that Theorem \ref{prop:ereg2} cannot be used in a covering argument since the maximal function is not bounded on $L^1$. This issue represents a mere technical difficulty though: it is resolved by introducing a suitable variant of the sharp maximal function which leads to the more involved $\varepsilon$-regularity criterion of Corollary \ref{prop:fixedscale}.}
{\new Theorem \ref{prop:ereg2} is a consequence of the $\varepsilon$-regularity Theorem \ref{prop:ereg1} (which holds for every $\alpha \in (\frac 14, \frac 12)$) whose smallness requirement features an $L^p$-based excess quantity and can be met at some small scale by requiring \eqref{eq:blabla}. Theorem \ref{prop:ereg1} on the other hand is obtained} via an excess decay result and a linearization argument, in analogy with \cite{Lin1998} for the classical Navier-Stokes equations and with \cite{CDLM17} for the hyperdissipative Navier-Stokes equations. Nevertheless there are some novelties in our approach with respect to the corresponding results for Navier-Stokes: 
\begin{itemize}
\item 
Our $\varepsilon$-regularity result relies on the crucial observation (previously used in \cite{CaffarelliVasseur2010,ConstantinWu2009}) that the equation \eqref{eq:SQG} is invariant under a change of variables which sets the space average of $u$ to zero. 
Indeed, the scaling \eqref{eq:rescaling}, in contrast to the analogous situation for the Navier-Stokes equations, does not guarantee any control on the average of $u$ on $B_r$ in terms of the average of the rescaled solution $u_r$ on $B_1$ as $r \to 0$. The lack of control on the averages introduces a challenge to iterate the excess decay, since at each step we need to correct for this change of variable, in a similar spirit to  \cite{ConstantinWu2009}.
\item  As a second ingredient we introduce a new notion of suitable weak solution which enables us to perform energy estimates of nonlinear type controlling a potentially large power of $\theta$. {\new Such nonlinear energy estimates exploit the boundedness of Leray-Hopf weak solutions in an essential way and are not available for the Navier-Stokes equations. The freedom of choosing a suitable nonlinear power on the other hand is crucial in the context of the SQG equation:} Indeed, the classical (local) energy controls naturally $\theta \in L^\infty ((0,T), L^2(\R^2)) \cap L^2 ((0,T), W^{\alpha,2}(\R^2) )$ and hence, by interpolation, $\theta \in L^{2(1+\alpha)}(\R^2 \times (0, T)) $. Yet, since $2(1+\alpha) <3$ for $\alpha< \frac{1}{2}$, this is not enough to conclude a strong enough Caccioppoli-type inequality which accounts for the cubic nonlinearity in the local energy.
\item { On one side Theorem~\ref{prop:ereg2} may be seen as an analogue of Scheffer's result \cite{Scheffer77} for Navier-Stokes, providing  $\varepsilon$-regularity criterion at a fixed scale. On the other side, in order to give an estimate on the dimension of the singular set, the smallness \eqref{eq:blabla} must be required in terms of differential quantities of $\theta$, as it happens in the more refined result by Caffarelli, Kohn and Nirenberg for Navier-Stokes \cite{CaffarelliKohnNirenberg1982}. In the context of the SQG equation, the easier Corollary~\ref{thm:Scheffer} below may be seen as the full analogue of Scheffer's result. Although Corollary~\ref{thm:Scheffer} still establishes the compactness of the singular set, it does,  in contrast to Navier-Stokes, not yield any estimate on the dimension of the singular set.}

\end{itemize}
%

Using the ``continuity" of the aforementioned $\varepsilon$-regularity Theorem \ref{prop:ereg1} under strong convergence in $L^p$, an immediate consequence is the stability of the singular set in the fractional order $\alpha \in (\frac 14, \frac 12]$ which in particular recovers the following result of \cite{CotiZelatiVicol2016}.
\begin{corollary}[Gobal regularity for slightly supercritical SQG]\label{cor:stability} Let $\theta_0 \in H^2(\R^2)$ with $\norm{\theta_0}_{H^2}\leq R\, .$ Then there exists $\varepsilon= \varepsilon(R)>0$ such that \eqref{eq:SQG}--\eqref{eq:u} has a unique smooth solution $\theta \in L^\infty_{loc}([0, \infty), H^2(\R^2)) \cap L^2_{loc}((0, \infty), H^{2+ \alpha}(\R^2))$ for all fractional orders $\alpha \in \left[\frac{1}{2}-\varepsilon, \frac{1}{2}\right]\,.$
\end{corollary}
\begin{remark} The corollary could be set in any $H^{1+ \delta}(\R^2)$ for $\delta>0$ which is subcritical for orders $\alpha$ close to $\frac 12$ and there admits a (quantified) short-time existence of smooth solutions. In \cite{CotiZelatiVicol2016} the assumption  $\norm{\theta_0}_{H^2}\leq R$ is replaced by the scaling invariant assumption $\norm{\theta_0}^\alpha_{L^2}\norm{\theta_0}^{1-\alpha}_{\dot H^2}\leq R\,.$ The latter statement can be reduced to ours by applying a first rescaling which renormalizes the $L^2$-norm of the initial datum to $1$.  
\end{remark}

Moreover, by the decay of the $L^\infty$-norm of solutions (see Theorem \ref{thm:LerayConstWu} below), {the $\varepsilon$-regularity criterion} is verified for large times and we recover the eventual regularization of suitable weak solutions from $L^2$-initial data for $\alpha \in (\frac{1}{4}, \frac{1}{2})$ previously established for Leray--Hopf solutions in \cite{Silvestre2010} for $\alpha$ close to $ \frac 12$ and in \cite{Dabkowski2011,Kiselev2011} for any $\alpha \in (0,\frac 12)$. 


%
\bigskip

{\new 
\subsection{A conjecture on the optimal dimension estimate}Theorem~\ref{thm:main} leaves open the question of whether or not the estimate on the dimension of the singular set, as well as the range of $\alpha$ for which it is valid, is optimal.
We believe that a natural conjecture for an optimal estimate of the dimension of the spacetime singular set is 
\begin{equation}
\label{eqn:conj1}
\dim_\mathcal{P}({\rm Sing} \,\theta) \leq \frac{1}{2\alpha}(4-4\alpha) \,,
\end{equation}
and
\begin{equation}
\label{eqn:conj}
\dim_\mathcal{H}({\rm Sing} \,\theta) \leq \frac{1}{2\alpha}\big(4-4\alpha\big)\,.
\end{equation}
 In \eqref{eqn:conj1}, $\mathcal P$ is 
 the parabolic Hausdorff measure that is, for $\alpha < \frac{1}{2}\,,$ the Hausdorff measure resulting from restricting the class $\mathcal{F}$ of admissible covering sets to the spacetime cylinders $\tilde Q_r(x,t)= B_{r^{1/(2\alpha)}}(x) \times (t-r, t]\,.$ 
 We refer for instance to \cite{CDLM17} for its construction for $\alpha> \frac{1}{2}$. The cylinders $\tilde Q_r(x,t)$ are the natural choice for $\alpha<\frac12$ because their diameter is less than $4r$, at difference from the classical parabolic cylinders $B_r(x) \times (t-r^{2\alpha}, t]$ whose diameter is of the order of $r^{2\alpha}$.

The conjecture \eqref{eqn:conj1} is based on a dimensional analysis of the equation
: We may assign a ``dimension'' to any function $f(\theta)$ of $\theta$ via the exponent $\beta$ of the rescaling factor $1/r^\beta$ which makes the spacetime integral of $f(\theta)$ on $\tilde Q_r$ dimensionless, i.e. scaling-invariant with respect to \eqref{eq:rescaling}. The number appearing on the right-hand side of \eqref{eqn:conj1} corresponds then to 
 dimension of the energy, whose dissipative part is \emph{the} globally controlled quantity in the form of a spacetime integral which scales best. This would correspond to the result of Caffarelli, Kohn and Nirenberg \cite{CaffarelliKohnNirenberg1982} ) for the Navier-Stokes system  (see \cite{TangYu2015,CDLM17, KwonOzanski2020} for fractional dissipations of order $\alpha \in [\frac{3}{4}, \frac{5}{4})$)
who proved  that suitable weak solutions of the latter are smooth outside a closed set of dimension $1.$ In fact, for the Navier-Stokes system this bound on the dimension of the singular set is what the scaling of the equations and boundedness of the energy suggest.
Notice that the right-hand side of both \eqref{eqn:conj1} and \eqref{eqn:conj} does not converge to $0$ as $\alpha \to \frac 12$: this is due to the fact that the quantity that dictates the scaling-criticality of the equation, namely the $L^\infty$-norm of $\theta$, is not of integral type and hence cannot be used in a covering argument of the type that we do in the proof of Theorem~\ref{thm:main}. In turn, this covering argument finds his pivotal quantity in the dissipative part of the energy, which has a worse scaling than the $L^\infty$-norm of $\theta$.}

{\new In the proof of Theorem \ref{thm:main}
, it is natural to consider the classical Hausdorff measure, since the tilting effect of the change of variables, which sets the space average of $u$ to zero, forces us to work on balls in spacetime (rather than parabolic cylinders, see Section \ref{sec:fixedscale} and in particular Step 3 of the proof of Corollary \ref{prop:fixedscale}). This effect of the change of variables constitutes a serious obstacle for any parabolic Hausdorff dimension estimate. However, our estimate is nonoptimal: to obtain the optimal estimate, one should replace $\frac{1+\alpha}{\alpha}$ by $2$ in the estimate of the dimension of the singular set in Theorem~\ref{thm:main}; however, the integrability exponent $\frac{1+\alpha}{\alpha}$ represents the least possible exponent for which we are able to use a ``nonlinear" localized energy inequality in an excess decay argument (cf. Lemma \ref{lem:compactness}). An analogous difficulty appears for the ipodissipative Navier-Stokes equations for low fractional orders $\alpha < \frac{3}{4}$ where the Caccioppoli-type inequality as described before fails to be strong enough to control the cubic nonlinearity and indeed no estimate of the dimension of the singular set is known.
}

\subsection{Structure of the paper} The paper is structured as follows. After recalling some technical preliminaries in Section \ref{sec:preliminaries}, we discuss in Section \ref{s:energy} the global and local energy inequalities of the SQG equation and we define the notion of suitable weak solutions. The key compactness property of the latter is proven in Section \ref{sec:compactness} and leads to an excess decay result established in Section \ref{sec:excessdecay}. The iteration of the excess decay on all scales is performed in Section \ref{sec:iter} and requires to introduce a change of variables which sets to $0$ the average of the velocity $u$ on suitable balls. This excess decay yields the basis for  several $\varepsilon$-regularity results, in particular Theorem \ref{prop:ereg2}, which are deduced in Section \ref{sec:ereg}. The proof of Theorem \ref{thm:main} is given in Section \ref{sec:dimension}. In Section \ref{sec:stability}, we discuss the stability of the singular set with respect to variations of the fractional order of dissipation.

\section{Preliminaries}\label{sec:preliminaries}
\subsection{Notation} We use the following notation for space(time) averages of functions or vector fields $f$ defined on $\R^2 \times [0, \infty)$: For bounded sets $E \subseteq \R^2 \times [0, \infty)$ and $F \subseteq \R^2\,,$ we define
\begin{equation}
(f)_E := \mean{E} f(x,t) \, dx \,dt \qquad \text{ and } \qquad [f(t)]_F :=\mean{F} f(x,t) \, dx\,.
\end{equation}
We introduce the spacetime cylinder adapted to the parabolic scaling \eqref{eq:rescaling} of the equation
$$Q_r(x, t):=B_r(x) \times (t-r^{2\alpha}, t]\,.$$
In the upper half-space $\R^3_+$ we define $B_r^*(x):=B_r(x) \times [0, r)$ and we define the extended cylinder
$$Q_r^*(x, t) := B_r^*(x) \times (t-r^{2\alpha}, t]\,. $$ We will omit the center of the cylinders whenever $(x, t)=(0,0) \,.$ 
Moreover, we use the following convention to describe spacetime H\"older spaces: For $\alpha, \beta \in (0,1)$ and $Q \subset \R^2 \times \R$ we denote by $C^{\alpha, \beta}(Q)$ the functions which are $\alpha$- and $\beta$-H\"older continuous in space and time respectively, namely such that the following semi-norm is finite
$$\| \theta \|_{C^{\alpha, \beta}(Q)} = \sup\bigg\{\frac{\lvert \theta(x,t)-\theta(y,s) \rvert}{\lvert x-y \rvert^\alpha + \lvert s-t \rvert^\beta }:{ (x,t), (y,s) \in Q \mbox{ with }(x, t) \neq (y, s)} \bigg\}.$$ 
Whenever $\alpha=\beta\,,$ we denote the above space just by $C^\alpha(Q)\,.$ Furthermore, we will also work with spatial Sobolev spaces of fractional order: For $\Omega \subseteq \R^n$, $s \in (0,1)$ and $1 \leq p < \infty \,,$ we denote by
\begin{equation}
W^{s, p}(\Omega):=\left \{ f \in L^p(\Omega): \frac{\lvert f(x)-f(y) \rvert}{\lvert x-y \rvert^{\frac{n}{p} +s} } \in L^p(\Omega \times \Omega)\right\}  \,.
\end{equation} 
Correspondingly, we define for $f \in W^{s,p}(\Omega)$ the Gagliardo semi-norm by
\begin{equation}
 [f]_{W^{s,p}(\Omega)}:= \left(\int_{\Omega} \int_{\Omega} \frac{\lvert f(x)-f(y) \rvert^p}{\lvert x-y \rvert^{n +sp} } \, dx \, dy \right)^\frac{1}{p}\,.
\end{equation}
In the special $p=2$, we will sometimes denote $W^{\alpha,2}$ by $H^\alpha$ and  we recall that for $\Omega=\R^n$ the Gagliardo semi-norm coïncides, up to a universal constant, with the semi-norms \eqref{e:en_of_ext-CS}. Finally, we will consider the Bochner spaces $L^q((0, T), X)$ for $1 \leq q \leq \infty$ and for some Banach space $X$ (here: $X=L^p(\R^2)$ or $X=W^{\alpha, 2}(\R^2)$). Whenever we work on a parabolic cube $Q_r(x,t)$, we will use the short-hand notation
$$L^2 W^{\alpha, 2}(Q_r):= L^q((t-r^{2\alpha}, t), W^{\alpha,2}(B_r(x)))\,.$$

\subsection{Singular points}We call a point $(x,t) \in \R^2 \times (0, \infty)$ a regular point of a Leray-Hopf weak solution $\theta$ of \eqref{eq:SQG}--\eqref{eq:u}  (see Definition \ref{def:LH}) if there exists a neighbourhood of $(x,t)$ where $\theta$ is smooth.
We denote by ${\rm Reg} \, \theta$ the open set of regular points in spacetime. Correspondingly, we define the spacetime singular set ${\rm Sing} \, \theta:=[ \R^2 \times (0, \infty)] \setminus {\rm Reg} \, \theta \, .$

\subsection{Riesz-transform} We recall that the Riesz-transforms admit a singular integral representation. Indeed, for $f:  \R^2 \to [0, \infty)$ and $i=1,2$
\begin{equation}
\mathcal{R}_if (x)= c \, {\rm p.v.} \int_{\R^2} \frac{x_i-z_i}{\lvert x-z \rvert^3} f(z) \, dz \, .
\end{equation}
By Calderon-Zygmund they are bounded operators on $L^p$ for $1<p< \infty$ and from $L^\infty$ to $ BMO \, .$

\subsection{Caffarelli-Silvestre extension}\label{sec:CS} We recall the following extension problem. We use the notation $\overline \nabla$, $\overline \Delta$ for differential operators defined on the upper half-space $\mathbb R^{n+1}_+$.
\begin{theorem}[Caffarelli--Silvestre \cite{CaffarelliSilvestre2007}]\label{thm:CF}
	Let $\theta\in H^{\alpha} (\R^n)$ with $\alpha\in (0,1)$ and set $b:=1-2\alpha$. Then there is a unique ``extension'' $\theta^*$ of $\theta$ in the weighted space $H^1(\R^{n+1}_+,y^b)$ which satisfies
	\begin{equation}\label{harm}
	\overline{\Delta}_b \theta^*(x,y):=\overline \Delta \theta^*+\frac{b}{y}\partial_y \theta^* = \frac{1}{y^b} \overline{{\rm div}}\, \big(y^b \overline\nabla \theta^*\big)=0
	\end{equation}
	and the boundary condition
	\begin{equation}\label{eqn:b-c-caffarelli}
	\theta^*(x,0)=\theta(x)\, .
	\end{equation}
Moreover, there exists a constant $c_{n, \alpha}$, depending only on $n$ and $\alpha$, with the following properties:
\begin{itemize}
\item[(a)] The fractional Laplacian $(-\Delta)^{\alpha} w$ is given by the formula 
\begin{equation}
	\label{eqn:frac-lap-est-CS}
		(-\Delta)^{\alpha} \theta (x)=c_{n,\alpha}\lim_{y\to 0}y^b\partial_y \theta^* (x,y)\,.
	\end{equation}
\item[(b)] The following energy identity holds
\begin{equation}\label{e:en_of_ext-CS}
\int_{\R^n}|(-\Delta)^{\frac{\alpha}{2}} \theta|^2\,dx=\int_{\R^n}|\xi|^{2\alpha}|\widehat{\theta}(\xi)|^2\,d\xi=c_{n,\alpha}\int_{\R^{n+1}_+}y^b|\overline\nabla \theta^*|^2\,dx\,dy\,.
\end{equation}
\item[(c)] The following inequality holds for every extension $\eta \in H^1 (\R^{n+1}_+, y^b)$ of $\theta$:
\begin{equation}\label{e:minimum_ext_caff}
\int_{\R^{n+1}_+}y^b|\overline \nabla \theta^*|^2\,dx\,dy \leq \int_{\R^{n+1}_+}y^b|\overline \nabla \eta|^2\,dx\,dy\,.
\end{equation}\
\end{itemize}
\end{theorem}

{\new \subsection{Poincar\'e inequalities}
Let $\alpha \in (0, 1)$, $1\leq p < \frac{n}{\alpha}$ and  ${p^*} := \frac{pn}{n-p \alpha }$. There exists a universal constant $C=C(\alpha, n, p)$ such that for every $f \in W^{\alpha, p}(\R^n)$, $q \in [p, p^*]\,,$ $x \in \R^n$ and  $r>0$
\begin{equation}\label{eq:PoincareStd}
\left(\int_{B_{r}(x)} \lvert f (z) - [f]_{B_{r}(x)} \rvert^q \, dz \right)^\frac{1}{q} \leq C r^{\alpha - n (\frac 1p - \frac 1q)} [f]_{W^{\alpha, p}(B_{r}(x))} \, .
\end{equation}
We will also need a weighted Poincar\'e inequality in the spirit of the classical work \cite{FabesKenigSerapioni1982} for $\alpha=1$ (where on the other side much more general weights 
are admissible). 
Let $\omega\in C^\infty_{c}(\R^n)$ be a radial, non-increasing weight such that $\omega \equiv 1$ on $\overline{B_{r/2}}(x)$, $\omega \equiv 0$ outside $B_{r}(x)\,$ and $|\nabla \omega |\leq \frac C r$ pointwise. We introduce the weighted average
\begin{equation}
[f]_{\omega, B_{r}(x)}:= \bigg( \int_{B_{r}(x)} \omega(z) \, dz \bigg)^{-1} \int_{B_{r}(x)} f(z) \omega (z) \, dz \, .
\end{equation}
The following weighted Poincar\'e inequality is  classical for $\alpha=1$ (see \cite[Lemma 6.12]{Lieberman1996}) and it is established for $q=p$ in \cite[Proposition 4]{DydaKassmann2013}: Their proof extends to the other endpoint $q= p^*\,$ and hence to the range $q \in [p,  p^*]$ by interpolation.
\begin{lemma} Under the above assumptions, we have the weighted Poincar\'e inequality
 \begin{equation}\label{eq:PoincareWgt}
\bigg(\int_{B_{r}(x)} \lvert f(z) - [f]_{\omega, B_{r}(x)} \rvert^q \omega(z) \, dz \bigg)^\frac{1}{q} \leq \tilde C  r^{\alpha - n (\frac 1p - \frac 1q)} [f]_{W^{\alpha, p}(B_{r}(x))}
\end{equation}
where $\tilde C=2^{2-\alpha +n/p } C \, .$
\end{lemma}
In the case $p=2$, we can rewrite the right-hand side of \eqref{eq:PoincareStd} and \eqref{eq:PoincareWgt} in terms of the extension as follows.
}
\begin{lemma}\label{lem:Sobolev} Let $n \geq2$, $\alpha \in (0,1)\,$, $0<r<s\,,$ {\cor$f\in  W^{\alpha, 2}(\R^n)$ and $g \in C^1(\R)$}. 
Then there exists $C=C(n, \alpha)$ such that 
\begin{equation}\label{est:boh}
[g \circ f]_{W^{\alpha, 2}(B_r)} \leq C \bigg( \int_{B_s^\ast} y^b \lvert \overline{\nabla} [g(f^\ast)] \rvert^2 \, dx \, dy +(s-r)^{-2} \int_{B_s^\ast\setminus B_r^\ast} y^b( g(f^\ast) )^2 \, dx \, dy  \bigg)^\frac{1}{2}
\end{equation}
and for any $2 \leq q \leq \frac{2n}{n-2\alpha}$
\begin{equation} \label{eq: localsobolevineq} 
\norm{g \circ f}_{L^q(B_r)} \leq C \bigg(\int_{B_s^\ast} y^b \lvert \overline{\nabla}  [g(f^\ast)]  \rvert^2\, dx \, dy +(s-r)^{-2} \int_{B_s^\ast\setminus B_r^\ast} y^b(  g(f^\ast) )^2 \, dx \, dy + \int_{B_r}  f^2 \,dx\bigg)^\frac{1}{2} \, .
\end{equation}
{\new In particular, for any $x \in \R^n\,$ 
\begin{equation}
\label{eq:GagliardoToExt}
[f]_{W^{\alpha,2}(B_r(x))} \leq C \bigg (\int_0^{\frac 43 r} \int_{B_{\frac 43 r}(x)} y^b \lvert \overline \nabla f^* \rvert^2 (z,y) \, dz \, dy \bigg)^\frac{1}{2} \, .
\end{equation}}
\end{lemma}
\begin{proof}
Let $n \geq 2$, $\alpha\in (0,1)$, $0<r<s$ and $g \in C^1(\R)$. By approximation we may assume that  $f$ is Schwartz. Fix a smooth cut-off function $\varphi \in C^\infty_c(\mathbb{R}^{n+1}_+)$ such that $0 \leq \varphi \leq 1$, $\varphi \equiv 1$ on $B_r^\ast$, $\supp \varphi \subseteq B_s^\ast$ and $\lvert \overline{\nabla} \varphi \rvert \leq C (s-r)^{-1}\,.$
For $\alpha \in (0,1)$ we use the minimizing property of the extension to write
\begin{align}\label{eq: Gagliardonorm with ext}
[g \circ f]_{W^{\alpha,2}(B_r)}^2 &= \int_{B_r} \int_{B_r} \frac{\lvert g(f(x))- g(f(z)) \rvert^2}{\lvert x-z\rvert^{n+2\alpha}} \, dx \, dz
\\& \leq \int_{\mathbb{R}^n}\int_{\mathbb{R}^n} \frac{\lvert ((g\circ f) \varphi|_{y=0}) (x) - ((g\circ f)\varphi|_{y=0})(z) \rvert^2}{\lvert x-z \rvert^{n+2\alpha}} \,dx \, dz \nonumber \\ 
&= c_{n,\alpha} \int_{\mathbb{R}^{n+1}_+} y^b \lvert \overline{\nabla}((g \circ f)\varphi|_{y=0})^\ast \rvert^2 \, dx \, dy \leq  c_{n,\alpha} \int_{\mathbb{R}^{n+1}_+} y^b \lvert \overline{\nabla}(g(f^\ast)\varphi) \rvert^2 \, dx \, dy  \nonumber \\
&\leq 2c_{n,\alpha} \int_{\mathbb{R}^{n+1}_+} y^b \lvert \overline{\nabla} [g(f^\ast )]\rvert^2 \varphi^2 \, dx \, dy+ 2c_{n,\alpha} \int_{\mathbb{R}^{n+1}_+} y^b( g(f^\ast) )^2 \lvert \overline{\nabla} \varphi \rvert^2 \, dx \, dy\nonumber \
\end{align}
and thus \eqref{est:boh} follows.
The estimate \eqref{eq: localsobolevineq} follows from \eqref{est:boh} via Sobolev embedding and interpolation. {\new
 As for \eqref{eq:GagliardoToExt}\,, we may assume $x=0$ and denoting by $c$ the weighted average of $f^*$ with respect to the weight $y^b$ on $B_{4r/3} \times [0, 4r/3]$, we have by the weighted Poincar\'e inequality of \cite{FabesKenigSerapioni1982} (with the Muckenhoupt weight $\omega(x,y)= y^b\in A_2$ on $\R^3_+$)
\begin{equation}\label{est:finiamola}
r^{-2} \int_0^{\frac{4}{3}r}\int_{B_{\frac{4}{3}r}} y^b \lvert f^* - c \rvert^2 \, dx \, dy \lesssim \int_0^{\frac{4}{3}r}\int_{B_{\frac{4}{3}r}} y^b \lvert \overline{\nabla} f^* \rvert^2 \, dx \, dy \,.
\end{equation}
The estimate \eqref{eq:GagliardoToExt} follows then from \eqref{est:boh} (applied to $f-c$ and $s=4r/3$) and \eqref{est:finiamola}.
}
\end{proof}
{\new
\subsection{(Sharp) maximal function}\label{sec:maximal}
For a function $f: \R^2 \times [0, \infty) \to \R\,$ we introduce the maximal function (in space)
\begin{equation}
\mathcal{M}f(x,t) := \sup_{r >0} \, \mean{B_r(x)} \lvert f(z, t) \rvert \, dz
\end{equation} 
as well as the sharp fractional maximal function (in space)
\begin{equation}\label{eq:sharpfractionalmax}
f^\#_{\alpha} (x,t):= \sup_{r>0} r^{-\alpha} \mean{B_r(x)} \lvert f(z,t) - [f(t)]_{B_r(x)} \rvert \, dz \,
\end{equation}
and for $q> \sqrt 2$ the following variant
\begin{equation}\label{defn:fsharpq}
f^\#_{\alpha, q} (x,t):= \sup_{r>0} r^{-2\alpha (1- 1/q^2)} \mean{B_r(x)} \lvert f(z,t) - [f(t)]_{B_r(x)} \rvert^{2(1-1/q^2)} \, dz \,.
\end{equation}
In order to use a spacetime integral of $\mathcal{\theta}^\#_{\alpha, q}$ in a covering argument, we need to know that it is globally controlled; to guarantee the latter, we are forced to choose $q< \infty$.
\begin{lemma}\label{lem:sharpmax} Let $\alpha \in (0,1)\,,$ $f \in W^{\alpha, 2}(\R^n) \,$ and $ q \in (\sqrt 2, \infty) \,.$ Then there exists a constant $C=C(n,q) \geq 1$ (which is uniformly bounded for $q$ bounded away from $\infty$) such that
\begin{equation}
\norm{f^\#_{\alpha}}_{L^2(\R^n)}^2 + \norm{f^{\#}_{\alpha, q} }_{L^{1+1/(q^2-1)}(\R^n)}^{1+1/(q^2-1)} \leq C [f]_{W^{\alpha,2}(\R^n)}^2 \, .
\end{equation}
\end{lemma}
For $\alpha=1\,,$ the equivalent of Lemma \ref{lem:sharpmax} is a simple consequence of the Poincar\'e inequality and the maximal function estimate. Indeed, by Poincar\'e we have almost everywhere the pointwise estimate
$$f^\#_1(x) \lesssim \mathcal{M}(\lvert \nabla f \rvert)(x) \qquad f^\#_{1,q}(x) \lesssim  \mathcal{M}\big(\lvert \nabla f \rvert^{2(1-1/q^2)}\big)(x)$$ 
for $f \in W^{1,1}_{loc}$ and  $f \in W^{1,2 (1-1/q^2) }_{loc}$ respectively. Integrating in $x$ and using the boundedness of the maximal function on $L^2$ and $L^{1+1/(q^2 -1)}\,,$ we obtain the equivalent of Lemma \ref{lem:sharpmax}.
\begin{proof} We give the proof for $f^\#_{\alpha,q}$
. We estimate the quantity in the supremum in \eqref{defn:fsharpq} 
\begin{align}\label{eq:blu}
 \mean{B_r(x)} \bigg\lvert \, \mean{B_r(x)} \frac{f(z) - f(y)}{r^{\alpha}} \, dy \bigg\rvert^{2(1-1/q^2)} \, dz
&\leq \mean{B_r(x)} \bigg(\, \mean{B_r(x)} \frac{|f(z) - f(y)|^2}{r^{2\alpha}} \, dy \bigg)^{1-1/q^2} \, dz
 \nonumber\\
&\leq  C \mean{B_r(x)} \bigg( \, \int_{B_r(x)} \frac{|f(z) - f(y) |^2}{\lvert z-y \rvert^{n+2\alpha}}\, dy \bigg)^{1-1/q^2} \, dz
\nonumber\\
&\leq  C \, \mean{B_r(x)} \left( D_{\alpha, 2} \, f(z) \right)^{2(1-1/q^2)} \, dz\,,
\end{align}
where $D_{\alpha, 2} \, f$ is the $n$-dimensional version of \eqref{eqn:Dalpha2}, i.e. for $z \in \R^n$
\begin{equation}
(D_{\alpha, 2} \, f) (z):=\left( \int_{\R^n} \frac{\lvert f(z)-f(z') \rvert^2}{\lvert z-z' \rvert^{n+2\alpha}} \, dz' \right)^{\frac{1}{2}}\,.
\end{equation}
By taking the supremum over $r>0\,,$ we deduce from \eqref{eq:blu} that for almost every $x$
\begin{equation}\label{eq:pointwiseestforsharpmaximalfunction}
f^\#_{\alpha, q} (x) \leq C \mathcal{M}\big((D_{\alpha,2} \, f)^{2(1-1/q^2)}\big)(x)
\end{equation}
and hence by the maximal function estimate on $L^{1 + {1}/{(q^2-1)}}$
\begin{equation*}
\norm{f^\#_{\alpha, q} }^{1+1/(q^2-1)}_{L^{1 + {1}/{(q^2-1)}}(\R^n)} \leq C \norm{(D_{\alpha,2} \, f)^{2(1-1/q^2)}}^{1+1/(q^2-1)}_{L^{1 + {1}/{(q^2-1)}}(\R^n)} = C \norm{D_{\alpha,2}\, f}^2_{L^{2}(\R^n)} = C[f]^2_{W^{\alpha,2}(\R^n)} \, .
\end{equation*}
\end{proof}
}
\section{The local energy inequalities}\label{s:energy}
\subsection{Leray-Hopf weak solutions}\label{s:LH} We recall the notion of Leray--Hopf weak solutions.
\begin{definition}\label{def:LH} Let $\theta_0\in L^2 (\R^2)$. A pair $(\theta,u)$ is a Leray--Hopf weak solution of \eqref{eq:SQG}--\eqref{eq:Cauchy}  on $\mathbb R^2\times (0,T)$ if:
\begin{itemize}
\item[(a)] $\theta\in L^\infty ((0,T), L^2 (\R^2)) \cap L^2 ((0,T), W^{\alpha,2} (\R^2))$;
\item[(b)] $\theta$ solves \eqref{eq:SQG}--\eqref{eq:Cauchy} in the sense of distributions, namely $\div u= 0$ and 
\begin{equation}
\label{eqn:weak}
\int \Big(\partial_t \varphi  \theta + u \theta \cdot \nabla \phi - (-\Delta)^\alpha \varphi  \theta \Big)\, dx\, dt 
= - \int \theta_0 (x) \cdot \varphi (0,x)\, dx
\end{equation}
for any $\varphi\in C^\infty_c (\mathbb R^2\times \mathbb R)$.
\item[(c)] The following inequalities hold for every $t\in (0,T)$ and for almost every $s\in (0,T)$ and every $t \in (s,T)$ respectively:
\begin{align}
\frac{1}{2} \int \theta^2 (x,t)\, dx + \int_0^t \int |(-\Delta)^{\frac{\alpha}{2}} \theta|^2 (x,\tau)\, dx\, d\tau &\leq \frac{1}{2} \int |\theta_0|^2 (x)\, dx
\label{e:g_energy_1}\\
\frac{1}{2} \int |\theta|^2 (x,t)\, dx + \int_s^t \int |(-\Delta)^{\frac{\alpha}{2}} \theta|^2 (x,\tau)\, dx\, d\tau &\leq \frac{1}{2} \int |\theta|^2 (x,s)\, dx
\label{e:g_energy_2} 
\end{align}
\end{itemize}
Correspondingly, we say that $\theta$ is a Leray--Hopf weak solution of \eqref{eq:SQG}--\eqref{eq:u} if additionally \eqref{eq:u} holds. 
\end{definition}
Observe that 
from the weak formulation \eqref{eqn:weak} it follows that for all $\varphi \in C^\infty_c(\R^3_+ \times [0, T))$
\begin{align}\label{eq:weakform3}
\int &\theta(x,t) \varphi(x,0,t) \, dx - \int \theta(x,s) \varphi(x,0,s) \, dx 
\nonumber \\&= \int_s^t \int \Big( \theta \partial_\tau \varphi|_{y=0} + (u\theta) \cdot \nabla \varphi|_{y=0} - \theta (-\Delta)^\alpha \varphi|_{y=0} \Big) \, dx \, d\tau
\nonumber \\&= \int_s^t \int \Big( \theta \partial_\tau \varphi|_{y=0} + (u\theta) \cdot \nabla \varphi|_{y=0} \Big) \, dx \, d \tau  
- c_{\alpha} \int_s^t \int_{\R^3_+} y^b  \overline{\nabla} \theta^* \cdot \overline{\nabla} \varphi \, dx \, dy \, d\tau\,
\end{align}
for $s=0$ and almost every $t\in (0,T)$ and for almost every $0<s<t<T\,$ (with $c_\alpha$ given by Theorem \ref{thm:CF}). 
Indeed, the equality between the last term of the second line and the last term of \eqref{eq:weakform3} holds for every $\theta \in L^2((0,T),W^{\alpha,2}(\R^2))$; in the smooth case this equality is a consequence of Theorem \ref{thm:CF} which one recovers for general $\theta$ through regularization. 

We recall that any Leray-Hopf weak solution is actually in $L^\infty$ for $t>0$. 
\begin{theorem}[\cite{ConstantinWu2009} Theorem 2.1]\label{thm:LerayConstWu} Let $\theta_0 \in L^2(\R^2)$ and let $(\theta, u)$ be a Leray-Hopf weak solution of \eqref{eq:SQG}--\eqref{eq:Cauchy}. Then there exist a universal constant, independent on $u$, such that for any $t>0\,$
\begin{equation}\label{eq:thetaLinfty}
\sup_{x \in \R^2} \lvert \theta(x,t) \rvert \leq C\norm{\theta_0}_{L^2} t^{-\frac{1}{2\alpha}}\, .
\end{equation}
In the particular case \eqref{eq:u}, where $u= \mathcal{R}^\perp \theta$, we obtain as a consequence that for any $t>0$
\begin{equation}\label{eq:uBMO}
\norm{u(\cdot, t)}_{{BMO}(\R^2)} \leq C \norm{\theta_0}_{L^2}t^{-\frac{1}{2\alpha}}\, .
\end{equation}
\end{theorem}
\begin{remark} In \cite{ConstantinWu2009} Theorem \ref{thm:LerayConstWu} is proven for Leray--Hopf weak solutions of the coupled system \eqref{eq:SQG}--\eqref{eq:u}. However, in the proof of \eqref{eq:thetaLinfty} only the energy inequality on level sets together with the assumption $\div u=0$ is used; the structure \eqref{eq:u} is only used to deduce \eqref{eq:uBMO} from \eqref{eq:thetaLinfty}.
\end{remark}

\subsection{Suitable weak solutions} We are now ready to give our definition of suitable weak solution. Both this notion and the one of Leray-Hopf solution are given without requiring the coupling \eqref{eq:u}, since, in the proof of Theorem~\ref{prop:ereg2}, we will need to work on a larger class of equations, where $u$ is obtained from $\theta$ by means of the Riesz transform and a temporal translation. 
{
\begin{definition}\label{def:sws}
A Leray--Hopf weak solution $(\theta,u)$ of \eqref{eq:SQG}--\eqref{eq:Cauchy} 
 on $\R^2 \times (0,T)$ is a suitable weak solution if the following two inequalities 
hold for almost every $t\in (0,T)$, all nonnegative test functions\footnote{That is, the function $\varphi$ vanishes when 
$|x| + y + |t|$ is large enough and if $t$ is sufficiently close to $0$, but it can be nonzero on some regions of $\{(x,y,t): y=0\}$.} $\varphi \in C^{\infty}_c ( \R^3_+ \times (0,T))$ with $\partial_y \varphi (\cdot,0,\cdot)= 0$ in $\R^2 \times (0,T)$, for all {\cor$q \geq 2$}  and every linear transformation of the form $\eta:= (\theta-M)/L$ with scalar $L>0$ and shift $M \in \R$:
	\begin{align}\label{eqn:suit-weak-tested2}
\int_{\R^2} \varphi (x,0,t) &\eta^2 (x,t) \, dx +2 c_{\alpha}\int_0^t \int_{\R^3_+} y^b|\overline{ \nabla} \eta^* |^2 \varphi \, dx \, dy \, ds\\
	\leq&\;
	\int_0^t \int_{\R^2}( \eta^2 \partial_t \varphi|_{y=0} + u \eta^2 \cdot \nabla \varphi|_{y=0})\, dx \, ds \nonumber + c_{\alpha} \int_0^t\int_{\R^3_+} y^b  (\eta^*)^2  \overline\Delta_b \varphi \, dx \, dy \, ds \,,
	\end{align}
	\begin{align}\label{eqn:suit-weak-tested}
\int_{\R^2} \varphi (x,0,t) &\lvert \eta\rvert^q(x,t) \, dx +4 \big(1- \tfrac{1}{q}\big)  c_{\alpha}\int_0^t \int_{\R^3_+} y^b|\overline{ \nabla} \lvert \eta^* \rvert^\frac{q}{2} |^2 \varphi \, dx \, dy \, ds\\
	\leq&\;
	\int_0^t \int_{\R^2}( \lvert \eta\rvert^q \partial_t \varphi|_{y=0} + u \lvert \eta \rvert^q \cdot \nabla \varphi|_{y=0})\, dx \, ds \nonumber + c_{\alpha} \int_0^t\int_{\R^3_+} y^b \lvert \eta^*\rvert^q \overline\Delta_b \varphi \, dx \, dy \, ds \,,
	\end{align}
where the constant $c_{\alpha}$ depends only on $\alpha$ and comes from Theorem~\ref{thm:CF}. 

Correspondingly, we say that $\theta$ is a suitable weak solution of \eqref{eq:SQG}--\eqref{eq:u} if additionally \eqref{eq:u} holds. 
\end{definition}
}
\begin{remark} In the classical notion of suitable weak solutions for the (hyperdissipative) Navier-Stokes equations, the local energy inequality \eqref{eqn:suit-weak-tested2} is asked to hold only for $\theta$ and not for every linear transformation $\eta:=(\theta-M)/L$. However, it can be proved (see for instance \cite{CDLM17}) that the class of suitable weak solutions is stable under this transformation. Here on the other hand, since we use a ``nonlinear" energy inequality \eqref{eqn:suit-weak-tested}, it is no longer obvious that the class of suitable weak solutions is stable under linear transformations; hence we require it already in the definition.
The class of suitable weak solutions contains smooth solutions (see Section \ref{subsec:energyineqsmooth}) and is non-empty (see Section \ref{subsec:existence}) for any $L^2$ initial datum.
\end{remark}

\subsection{Local energy equality for smooth solutions}\label{subsec:energyineqsmooth} It is not difficult to see that \eqref{eqn:suit-weak-tested2} and \eqref{eqn:suit-weak-tested} hold with an equality for every {\em smooth} solution of \eqref{eq:SQG}--\eqref{eq:Cauchy}. Indeed, let $f\in C^2(\R)$. We multiply \eqref{eq:SQG} by $ f'(\theta) \varphi|_{y=0}$
 and integrate in space to obtain for $t\in [0,T]$
\begin{align}
\int_{\R^2} f(\theta) (x,t) \varphi|_{y=0}(x, t) \, dx - \int_0^t \int_{\R^2} [f(\theta) \partial_t \varphi|_{y=0} &+ u f(\theta) \cdot \nabla \varphi|_{y=0}] \, dx \, ds \\
&= -2 \int_0^t \int_{\R^2} (-\Delta)^\alpha \theta  f'(\theta)\varphi|_{y=0} \, dx \, ds \,. 
\end{align}
By means of the divergence theorem, we compute for fixed time $t$
\begin{align}
\int_{\R^2} (-\Delta)^\alpha &\theta(x,t) f'(\theta)(x,t) \varphi|_{y=0}(x,t) \, dx \\
&= c_\alpha \lim_{y \to 0^+} \int_{\R^2} y^b \partial_y \theta^*(x,y,t) \left(f'(\theta^*) \varphi\right)(x,y,t) \, dx \\
&= c_\alpha \int_{\R^3_+} \overline{\div} ( y^b \overline{\nabla} \theta^* f'(\theta^*)\varphi) \, dx \, dy \\
&= 
 c_\alpha \int_{\R^3_+} y^b \lvert \overline{\nabla}\theta^* \rvert^2 f''(\theta^*) \varphi \, dx \, dy - {c_\alpha}\int_{\R^3_+} y^b f(\theta^*)\overline{\Delta}_b \varphi \, dx \, dy \,,
\end{align}
where we integrated by parts in the third equality and used that the boundary terms vanish due to the hypothesis $\partial_y \varphi(\cdot, 0, \cdot) = 0\, .$ We obtain that for $f\in C^2(\R)$ 
\begin{align}
\int_{\R^2} &f(\theta)(x,t) \varphi|_{y=0} (x,t) \, dx + c_\alpha \int_0^t \int_{\R^3_+} y^b
 \lvert \overline{\nabla}\theta^* \rvert^2 f''(\theta^*) 
  \varphi \, dx \, dy \, ds \\
&= \int_0^t \int_{\R^2} \left[f(\theta) \partial_t \varphi|_{y=0} + u f(\theta) \cdot \nabla \varphi_{y=0} \right] \, dx \, ds + c_\alpha \int_0^t \int_{\R^3_+} y^b f(\theta^*) \overline{\Delta}_b \varphi \, dx \, dy \, ds \, .
\end{align}
Observe that if $f$ is moreover convex and nonnegative, both the left- and the right-hand side of the above equality have a sign.
In particular, we obtain \eqref{eqn:suit-weak-tested2} with an equality when choosing $f(x)=\big(\frac{x-M}{L}\big)^2$ (since $f''\equiv 2L^{-2}$) and  \eqref{eqn:suit-weak-tested} when choosing $f(x)= \big\lvert \frac{x-M}{L} \big\rvert^{q}$ for $q \geq 2 \, .$ 
\subsection{Existence of suitable weak solutions}\label{subsec:existence} For any $\alpha \in (0, \frac{1}{2})$ the existence of suitable weak solutions can be established from any initial datum $\theta_0 \in L^2(\R^2)$ by adding a vanishing viscosity term $\epsilon \Delta \theta$ on the right-hand side and letting $\epsilon \to 0\, .$ The key argument is a classical Aubin-Lions type compactness argument that we sketch in Appendix~\ref{a:existencesws}.  

\begin{theorem}\label{thm:existencesuitable}
For any $\theta_0\in L^2 (\R^2)$ there is a suitable weak solution of \eqref{eq:SQG}--\eqref{eq:u} on $\R^2 \times (0, \infty)$. 
\end{theorem}

\subsection{Compactness}\label{sec:compactness} We establish the compactness of a sequence of suitable weak solutions with vanishing excess. Let $(\theta, u)$ be a solution of \eqref{eq:SQG}--\eqref{eq:Cauchy} on $\R^2 \times (0, T)\,.$ For $r>0$ and 
$Q_r(x,t) \subseteq \mathbb{R}^2 \times (0, T)$, we define the excess as 
\begin{align}
E(\theta, u; x, t, r)&:= E^S(\theta; x,t,r) +
E^V(u; x, t,r)+ 
E^{NL} (\theta; x, t,r)
 \end{align}
where
\begin{align}
 E^S(\theta; x,t,r) &:= \bigg(\,\mean{Q_r(x,t)} \lvert \theta(z,s)- (\theta)_{Q_r(x,t)} \rvert^p \, dz \, ds \bigg)^\frac{1}{p} \\
 E^V(u; x, t,r) &:= \bigg(\, \mean{Q_r(x,t)} \lvert u(z,s)- \left[u(s)\right]_{B_r(x)} \rvert^p \, dz \, ds \bigg)^\frac{1}{p} 
 \\
 E^{NL}(\theta; x, t, r)&:= \bigg(\,\mean{t-r^{2\alpha}}^t \sup_{R \geq \frac{r}{4}} \left(\frac{r}{R} \right)^{\sigma p} \bigg(\, \mean{B_R(x)} \lvert \theta(z,s) - {\new [\theta(s)]_{B_r(x)} }\rvert^{ \frac{3}{2}} \, dz \bigg)^\frac{2p}{3}  \, ds \bigg)^\frac{1}{p}
 \, ,
\end{align}
for $p \in (3, \infty)$ and $\sigma \in (0, 2\alpha)$ yet to be chosen. Observe that both parameters serve as (hidden) parameter for now and will be chosen in the very end to close the main $\varepsilon$-regularity Theorem (see also Remark \ref{rem:parameter}). Whenever $(x,t)=(0,0)$, we will denote the excess simply by $E(\theta, u;r)\, .$

\begin{remark}[Rescaling of the excess] \label{rem:rescalingexcess} The excess behaves nicely under the natural rescaling \eqref{eq:rescaling}. Indeed, for $r>0$ we have
\begin{equation}
E(\theta, u;x,t,r)= r^{1-2\alpha} E(\theta_r, u_r ;x, t, 1)\,.
\end{equation}
\end{remark}

\begin{lemma}[Compactness]\label{lem:compactness} Let $\alpha \in (0, \frac{1}{2}]\,,$ $\sigma \in (0, 2\alpha)$ and {$p> \frac{1+\alpha}{\alpha}\, .$} Let $(\theta_k, u_k)$ be a sequence of suitable weak solutions of \eqref{eq:SQG}--\eqref{eq:Cauchy} on $\R^2 \times [-1,0]\,$ with 
\begin{itemize}
\item $\lim_{k \to \infty }E(\theta_k, u_k ;1) =0\,$
\item and $[u_k(s)]_{B_1} =0$ for all $s \in [-1, 0]\, .$
\end{itemize}
Set $E_k:=E(\theta_k, u_k; 1)$ and define $\eta_k:=(\theta_k - (\theta_k)_{Q_1})/E_k\, .$  

Then there exists $\eta \in L^{3/2}_{loc}(\R^2 \times [-1, 0])$ such that, up to subsequences, $\eta_k \rightharpoonup \eta$ weakly in $L^{3/2}_{loc}(\R^2 \times [-1,0])\,.$  Moreover, $$\eta_k \to \eta \text{ strongly in }L^p(Q_{3/4})\,$$  and $\eta$ solves $$\partial_t \eta + (-\Delta)^\alpha \eta =0 \text{ on } Q_{3/4}\,$$ with $E^S(\eta;1) +E^{NL}(\eta;1) \leq 1\, .$ 
\end{lemma}
We will need the following auxiliary Lemma.

\begin{lemma}[Tail estimate] \label{lem:tail} Let $\alpha \in (0, \frac{1}{2}]\,,$ $\sigma \in (0, 2 \alpha)$ and $1<p<\infty\,$. Then there exists a universal constant $C=C(\alpha, \sigma, p)\geq 1$ such that for every $\theta \in L^p(B_2)$ with 
\begin{equation}
\sup_{R \geq 1} \frac{1}{R^{\sigma p}} \bigg(\, \mean{B_R} \lvert \theta(x) \rvert \, dx \bigg)^p < + \infty \, ,
\end{equation}
we have the estimate 
\begin{equation}
\int_{B_1^*} y^b \lvert \theta^*(x,y) \rvert^p \, dx \, dy \leq C \Bigg( \int_{B_2} \lvert \theta(x) \rvert^p \, dx +\sup_{R \geq 1} \frac{1}{R^{\sigma p}} \bigg( \,\mean{B_R} \lvert \theta(x) \rvert \, dx \bigg)^p\Bigg) \,.
\end{equation}
\end{lemma}
\begin{proof}
We set $\theta_1:= \theta \mathds{1}_{B_2}$ and $\theta_{i+1} := \theta (\mathds{1}_{B_{2^{i+1}}}- \mathds{1}_{B_{2^i}})$ for $i\geq1\, $. Recall that for $y>0$ the extension is given by $\theta^*(x,y)= (P(\cdot, y) \ast \theta)(x)$ for $P(x,y)= y^{2\alpha}/(\lvert x \rvert^2 + y^2)^{1+\alpha}\, .$ We estimate
\begin{equation}
\int_{B_1^* } y^b \lvert \theta^* \rvert^p\, dx \, dy \lesssim \int_{B_1^*} y^b \lvert \theta_1^* \rvert^p \, dx \, dy+ \int_0^1 y^b \Big( \sum_{i>1} \norm{\theta_i^*(\cdot, y)}_{L^\infty(B_1)} \Big)^p \, dy \, .
\end{equation}
The first term is estimated using Young and the fact that $\norm{P(\cdot, y)}_{L^1(\mathbb{R}^2)}=\norm{P(\cdot, 1)}_{L^1(\mathbb{R}^2)} = C_{\alpha}$ is a universal constant (see for instance the appendix of \cite{CDLM17}). Indeed, 
\begin{align}
\int_{B_1^*} y^b \lvert \theta_1^* \rvert^p \, dx \, dy= \int_0^1 y^b \norm{\theta_1 \ast P(\cdot, y)}_{L^p(B_1)}^p \, dy &\leq \int_0^1 y^b \norm{P(\cdot, y)}_{L^1(\R^2)}^p \norm{\theta_1}_{L^p(\R^2)}^p \, dy  \leq C_{\alpha}^p \int_{B_2} \lvert \theta \rvert^p \, dx \, .
\end{align}
For $i\geq1$, we estimate, using the fact that for $x \in B_1$ and $z \in B_{2^{i+1}} \setminus B_{2^i}$ we have $\lvert x-z \rvert \geq 2^{i-1}$ and thus $P(x-z,y) \leq P(2^{i-1}, y)$ uniformly in $z$, 
\begin{equation}
\norm{\theta_{i+1}^*(\cdot, y)}_{L^\infty(B_1)} \leq P(2^{i-1}, y) \int_{B_{2^{i+1}} \setminus B_{2^i}} \lvert \theta \rvert \, dx \leq \int_{B_{2^{i+1}} \setminus B_{2^i}} \frac{y^{2\alpha} \lvert \theta \rvert}{2^{2(1+\alpha)(i-1)}} \, dx \, ,
\end{equation}
so that 
\begin{align*}
\sum_{i>1} \norm{\theta_i^* (\cdot, y)}_{L^\infty(B_1)} &\leq \sum_{i\geq 1} \frac{y^{2\alpha}}{2^{(2\alpha-\sigma)(i-1)}} \int_{B_{2^{i+1}} \setminus B_{2^i}} \frac{\lvert \theta \rvert}{2^{(2+\sigma)(i-1)}} \, dx \\
&\leq C y^{2\alpha} \Big( \sum_{i\geq 1} \frac{1}{2^{(2\alpha-\sigma)(i-1)}} \Big) \sup_{R \geq 1} \frac{1}{R^\sigma} \mean{B_R} \lvert \theta \rvert \,dx \, .
\end{align*}
We obtain the claim by raising the previous inequality to the power $p$.
\end{proof}

\begin{proof}[Proof of Lemma \ref{lem:compactness}] Observe that $u_k \to 0$ in $L^p(Q_1)$ and thus, we may assume that $$\sup_{k \geq 1} \, \norm{u_k}_{L^p(Q_1)} \leq 1\, .$$ Moreover, by construction the pair $(\eta_k, u_k)$ is a distributional solution to $$\partial_t \eta_k + u_k \cdot \nabla \eta_k + (-\Delta)^\alpha \eta_k =0\,$$
with $E(\eta_k, \frac{u_k}{E_k}; 1) = 1\, .$

\textit{Step 1: We prove the uniform boundedness of $\eta_k$ in $\big(L^\infty L^2 \cap L^2 W^{\alpha,2} \cap L^{(p-1)(1+\alpha)}\big)(Q_{3/4})\,.$} \\
Fix a test function $\varphi \in C^\infty_c(\R^3_+ \times (0, \infty))$ such that $0 \leq \varphi \leq 1$, $\supp \varphi \subset Q_{7/8}^*$ and $\varphi \equiv 1$ on $Q_{27/32}^*\,.$ Moreover, we assume that $\varphi$ is constant in $y$ for small $y$, that is $\partial_y \varphi =0$ for $\{y < \frac{1}{2}\}\,.$ From the local energy inequality \eqref{eqn:suit-weak-tested2} we deduce that for $t \in [-(13/16)^{2\alpha}, 0]$
\begin{align}
\int_{B_{27/32}} &\eta_k^2(x,t) \, dx + 2 c_\alpha \int_{-\left(\frac{13}{16}\right)^{2\alpha}}^t \int_{B_{27/32}^*} y^b {\new \lvert \overline \nabla} \eta_k^*\rvert^2(x,y,s) \, dx \, dy \, ds \\
&\leq \int_{B_{7/8}}  \eta_k^2(x,t) \varphi(x,0,t) \, dx + 2 c_\alpha \int_{-\left(\frac{7}{8}\right)^{2\alpha}}^t \int_{B_{7/8}^*} y^b {\new \lvert \overline \nabla} \eta_k^*\rvert^2(x,y,s) \varphi(x,y,s) \, dx \, dy \, ds\\
&\leq \int_{-\left(\frac{7}{8}\right)^{2\alpha}}^t \int_{B_{7/8}} ( \eta_k^2 \partial_t \varphi|_{y=0}+ u_k \eta_k^2 \cdot \nabla \varphi|_{y=0})\, dx \, ds+ c_\alpha \int_{-\left(\frac{7}{8}\right)^{2\alpha}}^t \int_{B_{7/8}^*} y^b (\eta_k^*)^2 \overline{\Delta}_b \varphi \,dx \, dy \, ds \\
&\lesssim  \int_{Q_{7/8}}( \eta_k^2+ \lvert \eta_k \rvert^\frac{2p}{p-1}+ \lvert u_k \rvert^p)  \, dx \, ds  +   \int_{Q_{7/8}^*} y^b ( \eta_k^*)^2 \, dx \, dy \, ds \,.
\end{align}
Using Lemma \ref{lem:Sobolev}, the previous inequality and Lemma \ref{lem:tail}, we deduce that for $t \in [-(13/16)^{2\alpha}, 0]$
\begin{align}
\int_{B_{13/16}}  &\eta_k^2(x,t) \, dx + \int_{-\left(\frac{13}{16}\right)^{2\alpha}}^t [\eta_k(s)]_{W^{\alpha, 2}(B_{13/16})}^2 \, ds\\
&\lesssim \int_{B_{27/32}}  \eta_k^2(x,t) \, dx + 2 c_\alpha \int_{-\left(\frac{13}{16}\right)^{2\alpha}}^t \int_{B_{27/32}^*} y^b\big( \lvert \overline{\nabla}\eta_k^* \rvert^2 +(\eta_k^*)^2\big) \, dx \, dy \, ds \\
&\lesssim \int_{Q_1}( \eta_k^2+ \lvert \eta_k \rvert^\frac{2p}{p-1}+ \lvert u_k \rvert^p )\, dx \, ds + \int_{-1}^0 \sup_{R \geq 1} \frac{1}{R^{2\sigma}} \bigg(\, \mean{B_R} \lvert \eta_k(x,s) - [\eta_k(s)]_{B_1} \rvert \, dx \bigg)^2\, ds\\
&\lesssim 1 \,,
\end{align}
where we used in the last inequality that $\frac{2p}{p-1} \leq p \,$ together with the fact $E(\eta_k, \frac{u_k}{E_k};1) =1 \,.$ Taking the supremum over $t \in [-(13/16)^{2\alpha}, 0]\,,$ we deduce that uniformly in $k \geq 1$
\begin{align}\label{eq:uniformbound1}
\sup_{t \in [-(13/16)^{2\alpha}, 0]} \int_{B_{13/16}}  \eta_k^2(x,t) \, dx &+ \int_{-\left(\frac{13}{16}\right)^{2\alpha}}^0 [\eta_k(s)]_{W^{\alpha,2}(B_{13/16})}^2 \, ds \leq C \,.
\end{align}
We now consider $\psi_k := \lvert \eta_k \rvert^\frac{p-1}{2} \,$. 
Using Lemma~\ref{lem:Sobolev} applied with $r=\frac{13}{16}$, $s= \frac{27}{32}$, $g(x)= |x|^{\frac{p-1}{2}}$, the local energy inequality \eqref{eqn:suit-weak-tested} for $\psi_k$ and proceeding as before, using also that $p>3$, we thus have for any $t \in [-(13/16)^{2\alpha}, 0]$ that
\begin{align}
&\int_{B_{13/16}} \psi_k^2(x,t) \, dx + \int_{-\left(\frac{13}{16}\right)^{2\alpha}}^t [\psi_k(s)]_{W^{\alpha,2}(B_{13/16})} \, ds
\\
&\lesssim 
\int_{B_{27/32}}  \lvert \eta_k \rvert^{p-1}(x,t) \, dx + \int_{-\left(\frac{27}{32}\right)^{2\alpha}}^t \int_{B_{27/32}^*} y^b { \lvert \overline \nabla} \lvert \eta_k^* \rvert^\frac{p-1}{2} \rvert^2 \, dx \, dy \, ds
 +\int_{Q_{27/32}^*} y^b \rvert \eta_k^*\lvert^{p-1} \, dx \, dy \, ds
\\
&\lesssim  \int_{Q_{7/8}}\big( \lvert \eta_k \rvert^{p-1}+ \lvert \eta_k \rvert^p+ \lvert u_k \rvert^p \big)  \, dx \, ds  +\int_{Q_{7/8}^*} y^b \rvert \eta_k^*\lvert^{p-1} \, dx \, dy \, ds \\
&\lesssim \int_{Q_1}\big( \lvert \eta_k \rvert^{p-1}+ \lvert \eta_k \rvert^p+ \lvert u_k \rvert^p \big) \, dx \, ds +  \int_{-1}^0 \sup_{R \geq 1} \frac{1}{R^{(p-1)\sigma}} \bigg(\,\mean{B_R} \lvert \eta_k(x,s) - [\eta_k(s)] \rvert \, dx \bigg)^{p-1}\, ds \\
&\lesssim 1 \, .
\end{align}
Taking the supremum over $t \in [-(13/16)^{2\alpha}, 0]$, we obtain as before the uniform-in-$k$ bound
\begin{align}\label{eq:uniformbound2}
\sup_{t \in [-(13/16)^{2\alpha}, 0]} \int_{B_{13/16}}  \psi_k^2(x,t) \, dx &+ \int_{-\left(\frac{13}{16}\right)^{2\alpha}}^0 [\psi_k(s)]_{W^{\alpha,2}(B_{13/16})}^2 \, ds \leq C \,.
\end{align}
From Sobolev embedding $\dot W^{\alpha, 2} \hookrightarrow L^\frac{2}{1-\alpha}\,,$ we obtain by interpolation that $\norm{\psi_k}_{L^{2(1+\alpha)}(Q_{13/16})} \leq C \,$ uniformly in $k\geq 1$ and hence in particular
$\sup_{k \geq 1} \norm{\eta_k}_{L^{(p-1)(1+\alpha)}(Q_{13/16})} \leq C \, .$

\textit{Step 2: We use an Aubin-Lions type compactness argument to deduce strong convergence of $\eta_k$ in $L^q(Q_{3/4})$ for every $1 \leq q< (p-1)(1+\alpha)\,.$ Since $(p-1)(1+\alpha) >p$ by hypothesis, we deduce in particular that $\eta_ k \to \eta$ strongly in $L^p(Q_{3/4})\,.$ 
 }\\
We may assume $q \in [2, (p-1)(1+\alpha))$. Since the excess uniformly bounds the $L^{3/2}_{loc}$-norm of $\eta_k$, there exists by Banach-Alaoglu a limit $\eta \in L^{3/2}_{loc}(\R^2 \times [-1,0])$ such that $\eta_k \rightharpoonup \eta$ weakly in $L^{3/2}_{loc}(\R^2\times [-1,0])$, up to extracting a subsequence.
By the uniform boundedness established in Step 1, we may assume, up to extracting a further subsequence, that $\eta_k \rightharpoonup \eta$ weakly in $L^q(Q_{13/16})$. We now claim that the latter convergence is in fact strong on the slightly smaller cube $Q_{3/4}$. Indeed, fix $\varepsilon>0$ and a family $\{\phi_\delta \}_{\delta >0}$ of mollifiers in the space variable. For $k, j\geq 1$ we estimate
\begin{equation}
\norm{\eta_k-\eta_j}_{L^q(Q_{3/4})} \leq \norm{\eta_k-\eta_k \ast \phi_\delta}_{L^q(Q_{3/4})} + \norm{\eta_j-\eta_j \ast \phi_\delta}_{L^q(Q_{3/4})} + \norm{(\eta_k-\eta_j) \ast \phi_\delta}_{L^q(Q_{3/4})} \, .
\end{equation}
We claim that the first two contributions converge to $0$ as $\delta\to 0$, uniformly in $k$ and $j$. 
Indeed, we compute for $\delta$ small enough by H\"older and the uniform boundedness of $\eta_k$ in $L^2 W^{\alpha,2}(Q_{13/16})$ 
\begin{align}\label{eqn:tocopy}
\norm{\eta_k-\eta_k &\ast \phi_\delta}_{L^2(Q_{3/4})}^2 = \int_{-\left(\frac{3}{4}\right)^{2\alpha}}^0 \int_{B_{3/4}} \bigg \lvert \int (\eta_k(x)-\eta_k(y)) \phi_\delta(x-y) \, dy \bigg \rvert^2 \, dx \, dt\nonumber\\
&\leq  \int_{-\left(\frac{3}{4}\right)^{2\alpha}}^0 \int_{B_{3/4}} \bigg( \int \frac{\lvert \eta_k(x)-\eta_k(y) \rvert^2}{\lvert x-y \rvert^{2+2\alpha}} \mathds{1}_{\lvert x-y \rvert \leq \delta} \, dy\bigg) \bigg( \int \phi_\delta^2(x-y) \lvert x-y \rvert^{2+2\alpha} \, dy \bigg) \, dx \,dt\nonumber \\
&\leq {\new \pi} \delta^{2\alpha} \norm{\phi}_{L^\infty}^2  \int_{-\left(\frac{3}{4}\right)^{2\alpha}}^0 \int_{B_{{3/4}+\delta}} \int_{B_{{3/4}+\delta}} \frac{\lvert \eta_k(x)-\eta_k(y) \rvert^2}{\lvert x-y \rvert^{2+2\alpha}}  \,dy\, dx \, dt \nonumber \\
&\leq {\new \pi} \delta^{2\alpha} \norm{\phi}_{L^\infty}^2   \int_{-\left(\frac{13}{16}\right)^{2\alpha}}^0[\eta_k(t)]_{W^{\alpha,2}(B_{13/16})}^2 \, dt \leq C \delta^{2\alpha} \, ,
\end{align}
where $C$ does not depend on $k\geq 1$ by Step 1. Since $\eta_k$ is uniformly bounded in $L^{(1+\alpha)(p-1)}(Q_{13/16})$, we have by interpolation for some $\vartheta\in (0, 1]$ and $C \geq 1$ that
\begin{equation}
\norm{\eta_k-\eta_k \ast \phi_\delta}_{L^q(Q_{3/4})} \leq C \delta^{\alpha \vartheta } \, .
\end{equation}
We now fix $\delta$ small enough, independently of $k$, such that this contribution does not exceed $\frac{\varepsilon}{3}\,.$ As for the third term, we consider for fixed $\delta>0$ small, the family of curves $\{t \mapsto \eta_k \ast \phi_\delta\}_{k\geq 1}$.  From the equation, we have the identity
\begin{equation}
\partial_t( \eta_k \ast \phi_\delta) = - (u_k \cdot \nabla \eta_k) \ast \phi_\delta  - (-\Delta)^\alpha \eta_k \ast \phi_\delta \, .
\end{equation}
Observe that $u_k \cdot \nabla \eta_k = \div( u_k \eta_k)$ so that
\begin{equation}
\norm{(u_k \cdot \nabla \eta_k) \ast \phi_\delta }_{L^\frac{pq}{p+q}([-(3/4)^{2\alpha}, 0],W^{1, \infty}(B_{3/4}))} \leq \norm{u_k}_{L^p(Q_{3/4})} \norm{\eta_k}_{L^q(Q_{3/4})} \norm{\phi_\delta}_{W^{2, \frac{pq}{pq-p-q}}(B_{3/4})} \, .
\end{equation}
As for the last term, we have that for $x\in B_{3/4}$
\begin{align*}
\lvert (-\Delta)^\alpha \phi_\delta (x) \rvert &=c_\alpha \left \lvert  \int \frac{\phi_\delta(x)-\phi_\delta(y)}{\lvert x-y \rvert^{2+2\alpha}} \, dy \right \rvert \lesssim \norm{\phi_\delta}_{C^2} \int_{B_1} \frac{\, dy}{\lvert y \rvert^{2\alpha}} + \norm{\phi_\delta }_{L^\infty} \int_{B_1^c} \frac{\,dy}{\lvert x-y \rvert^{2+2\alpha}} \\
&\leq  \frac{C(\delta)}{1+\lvert x \rvert^{2+2\alpha}} \, .
\end{align*}
Analogously,  $\lvert (-\Delta)^\alpha  \nabla \phi_\delta (x) \rvert \leq \frac{C(\delta)}{1 + \lvert x \rvert^{2+2\alpha}}\, .$
We estimate the convolution on dyadic balls for fixed time. We set $\eta_{k, i}:= \eta_k (\mathds{1}_{B_{2^{i+1}}} - \mathds{1}_{B_{2^i}})$ for $i\geq 0$ and estimate
\begin{equation}
\norm{(-\Delta)^\alpha \phi_\delta \ast \eta_k}_{W^{1, \infty}(B_{3/4})} \leq  C(\delta) \bigg( \int_{B_1} \lvert \eta_k \rvert + \sum_{i \geq 0} \norm{(-\Delta)^\alpha \phi_\delta \ast \eta_{k, i}}_{W^{1,\infty}(B_{3/4})} \bigg) \ \, .
\end{equation}
For $i \geq 0$ we observe that for $x \in B_{3/4} $ and $z \in B_{2^{i+1}} \setminus B_{2^i}$ we have $\lvert x - z \rvert \geq 2^{i-2}\,,$ so that 
\begin{align}
\sum_{i \geq 0} \norm{(-\Delta)^\alpha \phi_\delta \ast \eta_{k, i}}_{W^{1,\infty}(B_{3/4})} &\leq C(\delta)\sum_{i\geq 0} \frac{1}{2^{(i-2)(2+2\alpha)}} \int_{B_{2^{i+1}}\setminus B_{2^i}} \lvert \eta_k \rvert \, dy \\ 
&\leq C(\delta) \sum_{i\geq 0} \frac{1}{2^{(i-2)(2\alpha-\sigma)}} \int_{B_{2^{i+1}}\setminus B_{2^i}} \frac{\lvert \eta_k \rvert}{2^{(2+\sigma)(i-2)}}  \, dy\\
 &\leq C(\delta) 	\bigg(  \sup_{R \geq 1} \frac{1}{R^\sigma} \mean{B_R} \lvert \eta_k - [\eta_k]_1 \rvert \, dy + \int_{B_1} \lvert \eta_k\rvert \, dy \bigg) \, .
\end{align}
We conclude by integrating in time that 
\begin{equation}
\norm{(-\Delta)^\alpha \phi_\delta \ast \eta_k}_{L^p([-(3/4)^{2\alpha},0], W^{1, \infty}(B_{3/4}))} \leq C(\delta) \left( E^S(\eta_k;1) + E^{NL}(\eta_k;1) \right) \,.
\end{equation}
Summarizing, we have shown that 
\begin{equation}
\norm{\partial_t (\eta_k \ast \phi_\delta)}_{L^\frac{pq}{p+q}([-(3/4)^{2\alpha}, 0], W^{1,\infty}(B_{3/4}))} \leq C(\delta)
\end{equation}
uniformly in $k\geq 1$. Hence the family of curves $\{t \mapsto \eta_k \ast \phi_\delta\}_{k\geq 1}$ is an equicontinuous sequence with values in a bounded subset of $W^{1, \infty}(B_{3/4})$. By Arzela-Ascoli there exists a uniformly convergent subsequence (which we don't relabel), and in particular, there exists $N=N(\delta)>0$ such that for any $k, j \geq N$ we have 
\begin{equation}
\norm{(\eta_k-\eta_j) \ast \phi_\delta}_{L^q(Q_{3/4})} \leq \frac{\varepsilon}{3} \, ,
\end{equation}
hence $\norm{\eta_k-\eta_j}_{L^q(Q_{3/4})} \leq \varepsilon$ for all $k, j \geq N$ which proves the claim.

\textit{Step 3: Conclusion.} \\
By Step 2 we can pass to the limit in the equation in $Q_{3/4}$ and deduce that $\eta \in L^p(Q_{3/4})$ is a distributional solution of $\partial_t \eta + (-\Delta)^\alpha \eta =0$ in $Q_{3/4} \, .$ Moreover, by weak lower semicontinuity $E^S(\eta;1) + E^{NL}(\eta;1) \leq 1 \, .$
\end{proof}

\section{Decay of the excess}\label{sec:excessdecay}
In this section, we prove the self-improving property of the excess, namely that if the excess is small at any given $Q_r\,,$ there exists a small, fixed scale $\mu_0 \in (0, \frac{1}{2})$, independent of $r$, at which the excess decays between $Q_r$ and $Q_{\mu_0 r}$ - provided that the velocity field has zero average on $B_r$. This requirement is crucial to guarantee the decay of the excess related to the non-local part of the velocity (see $E^V(v_k^2;\mu)$ in the proof of Proposition \ref{prop:excessdecay}). More generally, one could prove this excess decay at scale $\mu_0$ under the weaker assumption that the average of the rescaled velocity $u_r$ on $B_1$ is bounded uniformly in $r$ for $r\in (0,1)$. However, 
since all $L^p$-norms are supercritical with respect to the scaling \eqref{eq:rescaling} of the equation, we will not be able to guarantee such an assumption. In this section, we will also for the first time make use of the structure of the velocity field \eqref{eq:u}. Similar arguments should apply for velocity fields determined from $\theta$ by other singular integral operators.

\begin{prop}[Excess decay]\label{prop:excessdecay} Let $\alpha \in (0, \frac{1}{2})$, $\sigma \in (0, 2\alpha)$ and {\new $p > \max \left\{\frac{1+\alpha}{\alpha}, \frac{2\alpha}{\sigma} \right\}$}. For any $c>0$ and any $\gamma \in (0, \sigma -\frac{2\alpha}{p}) $ there exist universal $\varepsilon_0= \varepsilon_0(\alpha, \sigma, p, c, \gamma) \in (0, \frac{1}{2})$ and $\mu_0= \mu_0(\alpha, \sigma, p, c, \gamma)\in \left(0, \frac{1}{2}\right)$ such that the following holds: Let $Q_r(x,t) \subseteq \R^2 \times (0, \infty)$ 
and let $(\theta, u)$ be a suitable weak solution to \eqref{eq:SQG}--\eqref{eq:Cauchy}. We assume that the velocity field satisfies $\left[u(s)\right]_{B_r(x)} = 0$ for all $s \in [t-r^{2\alpha}, t]$ and is obtained from $\theta$ by 
\begin{equation}\label{eq:uwithf}
u(y,s) = \mathcal{R}^\perp \theta (y,s) + f(s) \,
\end{equation}
for some $f \in L^1([t-r^{2\alpha}, t])$. Then, if $E(\theta, u; x, t,r) \leq r^{1-2\alpha} \varepsilon_0\,,$ the excess decays at scale $\mu_0$, that is
\begin{equation}\label{eq:decay}
E(\theta, u; x,t, \mu_0 r) \leq c {\mu_0}^\gamma E(\theta, u; x,t,r) \, .
\end{equation}
\end{prop}
\begin{remark} If in \eqref{eq:uwithf} $f=0$ we recover simply the SQG equation. We will need the freedom to subtract a function of time $f$  from the velocity field $u$ in order to satisfy the zero-average assumption (see Lemma \ref{lem:changeofvar}).
\end{remark}

We will need the following auxiliary Lemma. 

\begin{lemma}\label{lem:boundfirstder} Assume $\theta\in L^2_{loc}(\R^2)$ with $\supp \theta \subseteq \left(\overline{B}_{3/8}\right)^c$ and that for some $\sigma \in (0,1)$ we have
\begin{equation}
\sup_{R\geq 1} \frac{1}{R^\sigma} \mean{B_R} \lvert \theta (x) \rvert \, dx < + \infty \,.
\end{equation}
Then $\mathcal{R}^\perp \theta \in C^\infty(B_{3/8})$ and there exists a universal $C=C(\sigma)>0$  such that
\begin{equation}
\norm{D \mathcal{R}^\perp \theta}_{L^\infty(B_{1/4})} \leq C \sup_{R\geq 1} \frac{1}{R^\sigma} \mean{B_R} \lvert \theta(x) \rvert \,dx  \,.
\end{equation}
\end{lemma}
\begin{proof}[Proof of Lemma \ref{lem:boundfirstder}] Observe that from $(-\Delta)^\frac{1}{2} \mathcal{R}^\perp \theta =0$ on $\overline{B}_{3/8}$, we infer that $\mathcal{R}^\perp \theta  \in C^\infty(B_{3/8}) \, .$ Moreover for $i,j=1,2$ and $x \in B_{1/4}$, we notice that the integral representation is no longer singular and we can compute by integration by parts 
\begin{align}
\partial_j \mathcal{R}^i\theta(x)= \int_{\lvert z \rvert \geq \frac{3}{8}} \frac{x_i-z_i}{\lvert x-z \rvert} \partial_j \theta(z) \, dz=- \int_{\lvert z \rvert \geq \frac{3}{8}} \partial_j \left( \frac{x_i-z_i}{\lvert x-z \rvert^3} \right) \theta(z) \, dz  \,,
\end{align}
where we used that the boundary terms at $\{ \lvert z \rvert = \frac{3}{8} \}$ and at infinity vanish.
Observe that for $x \in B_{1/4}$ and $z \in B_{2^{i+1}} \setminus B_{2^i}$ we have $\lvert x-z \rvert \geq 2^i -\frac{1}{4} \geq 2^{i-1}$ for $i \geq -1 \, .$ Thus 
\begin{align}
\lvert \partial_j \mathcal{R}^i \theta(x) \rvert &\leq \int_{\frac{3}{8} \leq \lvert z \rvert \leq \frac 12} \Big| \partial_j \Big( \frac{x_i-z_i}{\lvert x-z \rvert^3} \Big)   \theta (z) \Big| \, dz +\sum_{i \geq -1} \int_{B_{2^{i+1}} \setminus B_{2^i}} \Big| \partial_j \Big( \frac{x_i-z_i}{\lvert x-z \rvert^3} \Big)  \theta(z)\Big| \, dz \\
&\leq C \bigg( \int_{B_1} \lvert \theta (z) \rvert \, dz +  \sum_{i \geq -1} \int_{B_{2^{i+1}} \setminus B_{2^i}}  \frac{ \lvert \theta (z)\rvert}{2^{3(i-1)}} \, dz  \bigg) \\
&\leq C \Big( \sum_{i \geq -1} \frac{1}{2^{(i-1)(1-\sigma)}} \Big) \bigg( \sup_{R\geq 1} \frac{1}{R^\sigma} \mean{B_R} \lvert \theta (z) \rvert  \, dz   \bigg) \,.
\end{align}
\end{proof}
\begin{proof}[Proof of Proposition \ref{prop:excessdecay}]
By translation and scaling invariance, we may assume w.l.o.g. $(x,t)=(0, 0)$ and $r=1\,.$ We argue by contradiction. Then there exists a sequence $(\theta_k, u_k)$ of suitable weak solutions to \eqref{eq:SQG}--\eqref{eq:Cauchy} such that
\begin{itemize}
\item $E(\theta_k, u_k; \mu ) > c \mu^\gamma E(\theta_k, u_k;1) $ for all $\mu \in (0, \frac{1}{2})\,,$
\item $\lim_{k \to \infty} E(\theta_k, u_k; 1) = 0\,,$
\item $[u_k(s)]_{B_1}= 0$ for all $s\in [-1,0]$ and for all $k \geq 1\,$
\item $u_k(y,s)= \mathcal{R}^\perp \theta(y,s) + f_k(s)\,$ for some $f_k \in L^1([-1, 0])\,.$
\end{itemize}
We set $E_k:=E(\theta_k, u_k; 1)$ and $M_k := (\theta_k)_{Q_1}$. We will consider the rescaled and shifted sequence
\begin{equation}
\eta_k := \frac{\theta_k - M_k}{E_k} \,
\quad  \text{and } \quad v_k := \frac{u_k }{E_k} \, .
\end{equation}
By construction, $(\eta_k)_{Q_1}=0$ and $E(\eta_k, v_k;1)=1\, .$ In particular, we have for all $\mu \in (0, \frac{1}{2})$ that
\begin{equation}\label{eq:tocontradict}
E(\eta_k, v_k; \mu )  >  c \mu^{\gamma}\,.
\end{equation}
We will now take the limit $k \to \infty$ and argue that \eqref{eq:tocontradict} contradicts the excess decay dictated by the linear limit equation. Indeed, by Lemma \ref{lem:compactness}, the sequence $\eta_k$ converges weakly to $\eta$ in $L^{3/2}_{loc}(\R^2 \times [-1,0])$ and moreover, $\eta_k \to \eta$ strongly in $L^p(Q_{3/4})\,.$ Hence we have for $\mu \in (0, \frac{1}{2})$
\begin{equation}
E^S(\eta; \mu) = \lim_{k \to \infty} E^S(\eta_k; \mu) \,.
\end{equation}
We also know from Lemma \ref{lem:compactness} that $\eta \in L^p(Q_{3/4})$ solves the fractional heat equation $\partial_t \eta +(-\Delta)^\alpha \eta=0$ on $Q_{3/4}\, $ with $E^S(\eta;1)+ E^{NL}(\eta;1) \leq 1\, .$ In particular, we deduce from Lemma \ref{lem:linearizedeq} that $\eta$ is smooth (in space) on $Q_{1/2}$ and that $\eta \in C^{ 1-1/p}(Q_{1/2}) $ with the estimate
\begin{align}\label{eq:holderconteta}
\norm{\eta}_{L^\infty([-(1/2)^{2\alpha}, 0], C^1(B_{1/2}))} + \norm{\eta}_{C^{1-\frac{1}{p}}(Q_{1/2})} \leq \bar C( E^S(\eta; 1) + E^{NL}(\eta;1))\leq  \bar C \,.
\end{align}
In particular, we infer that for $\mu \in (0, \frac{1}{2})$
\begin{equation}\label{eq:excessS}
 \lim_{k \to \infty} E^S(\eta_k; \mu) =E^S(\eta; \mu) \leq \bar C \mu^{2\alpha(1-\frac{1}{p})} \,.
\end{equation}
Let us now consider the non-local part of the excess. We split
\begin{align}
E^{NL}(\eta_k;\mu) &\leq \bigg(\, \mean{-\mu^{2\alpha}}^0 \sup_{\frac{\mu}{4} \leq R < \frac{1}{4}} \left(\frac{\mu}{R}\right)^{\sigma p} \bigg(\,\mean{B_R} \lvert \eta_k(x,t) -[\eta_k(t)]_{B_\mu} \rvert^\frac{3}{2} \, dx \bigg)^\frac{2p}{3}\, dt \bigg)^\frac{1}{p}\\
&+\bigg( \,\mean{-\mu^{2\alpha}}^0 \sup_{R \geq \frac{1}{4}} \left(\frac{\mu}{R}\right)^{\sigma p} \bigg(\,\mean{B_R} \lvert \eta_k(x,t) -[\eta_k(t)]_{B_\mu} \rvert^\frac{3}{2} \, dx \bigg)^\frac{2p}{3} \, dt \bigg)^\frac{1}{p}\, .
\end{align}
We estimate the second term by adding and subtracting $[\eta_k(t)]_{B_1}$ for fixed time $t$. In the sequel $C'=C'(\alpha, \sigma)$ will denote a universal constant which may change line by line. Using that $E^{NL}(\eta_k;1)+E^S(\eta_k;1) \leq 1$ for all $k \geq 1$, we obtain that 
\begin{align}
\bigg(\, \mean{-\mu^{2\alpha}}^0 &\sup_{R \geq \frac{1}{4}} \left(\, \frac{\mu}{R}\right)^{\sigma p} \bigg(\,\mean{B_R} \lvert \eta_k(x,t) -[\eta_k(t)]_{B_\mu} \rvert^\frac{3}{2} \, dx \bigg)^\frac{2p}{3} \, dt \bigg)^\frac{1}{p} \\
&\leq \mu^{\sigma - \frac{2\alpha}{p}}  E(\eta_k;1) +  (4\mu)^\sigma \bigg( \,\mean{-\mu^{2\alpha}}^0 \lvert [\eta_k(t)]_{B_1}- [\eta_k(t)]_{B_\mu} \rvert^p \, dt \bigg)^\frac{1}{p}
\\
&\leq \mu^{\sigma - \frac{2\alpha}{p}} + (4\mu)^\sigma  \Bigg(\bigg(\, \mean{-\mu^{2\alpha}}^0 \lvert [\eta_k(t)]_{B_1}- [\eta_k(t) ]_{B_{1/2}} \rvert^p \, dt \bigg)^\frac{1}{p} +  \bigg(\, \mean{-\mu^{2\alpha}}^0 \lvert [\eta_k(t)]_{B_{1/2}}- [\eta_k(t)]_{B_\mu} \rvert^p \, dt \bigg)^\frac{1}{p}\Bigg) \\
&\leq \mu^{\sigma - \frac{2\alpha}{p}}+ C' \mu^{\sigma - \frac{2\alpha}{p}} E^S(\eta_k;1) + (4\mu)^\sigma \bigg(\, \mean{-\mu^{2\alpha}}^0 \lvert [\eta_k(t)]_{B_{1/2}}- [\eta_k(t)]_{B_\mu} \rvert^p \, dt \bigg)^\frac{1}{p} \\
&\leq C' \mu^{\sigma-\frac{2\alpha}{p}} +(4\mu)^\sigma  \bigg(\, \mean{-\mu^{2\alpha}}^0 \lvert [\eta_k(t)]_{B_{1/2}}- [\eta_k(t)]_{B_\mu} \rvert^p \, dt \bigg)^\frac{1}{p}
\, .
\end{align}
We infer, using the strong converge of $\eta_k \rightarrow \eta$ in $L^p(Q_{3/4})$, that 
\begin{align}
\liminf_{k \to \infty} E^{NL}(\eta_k;\mu) \leq C' \mu^{\sigma-\frac{2\alpha}{p}}   &+\,  (4\mu)^\sigma  \bigg(\, \mean{-\mu^{2\alpha}}^0 \lvert [\eta(t)]_{B_{1/2}}- [\eta(t)]_{B_\mu} \rvert^p \, dt \bigg)^\frac{1}{p}\\
&+\bigg(\, \mean{-\mu^{2\alpha}}^0 \sup_{\frac{\mu}{4} \leq R < \frac{1}{4}} \left(\frac{\mu}{R}\right)^{\sigma p} \bigg(\, \mean{B_R} \lvert \eta(x,t) -[\eta(t)]_{B_\mu} \rvert^\frac{3}{2}\, dx \bigg)^\frac{2p}{3} \, dt\bigg)^\frac{1}{p} \, .
\end{align}
Using \eqref{eq:holderconteta} again, we obtain that $ \lvert [\eta(t)]_{B_\mu}-[ \eta(t)]_{B_{1/2}} \rvert  \leq \bar C$ uniformly in time as well as
\begin{align}
\Bigg(\, \mean{-\mu^{2\alpha}}^0 \sup_{\frac{\mu}{4} \leq R < \frac{1}{4}} \left(\frac{\mu}{R}\right)^{\sigma p} \bigg(\,\mean{B_R} \lvert \eta(x,t) -[\eta(t)]_{B_\mu} \rvert^\frac{3}{2} \bigg)^\frac{2p}{3} \Bigg)^\frac{1}{p} 
\leq \bar C \bigg( \,\mean{-\mu^{2\alpha}}^0 \sup_{\frac{\mu}{4} \leq R < \frac{1}{4}} \left(\frac{\mu}{R}\right)^{\sigma p} R^p  \bigg)^\frac{1}{p} \leq \bar C \mu^\sigma \, .
\end{align}
We conclude that
\begin{equation}\label{eq:excessNL}
\liminf_{k \to \infty} E^{NL}(\eta_k;\mu) \leq  C' \mu^{\sigma-\frac{2\alpha}{p}}  \,.
\end{equation}
Finally, let us consider the part of the excess which is related to the velocity $v_k\,.$ We observe that, using the structure of the velocity \eqref{eq:uwithf},
\begin{equation}
E^V(v_k;\mu) = \bigg(\, \mean{Q_\mu} \big| \mathcal{R}^\perp\big(E_k^{-1} \theta_k\big)(x,t)- \big[\mathcal{R}^\perp\big(E_k^{-1} \theta_k\big)(t)\big]_{B_\mu} \big|^p \, dx \, dt \bigg)^\frac{1}{p} \,.
\end{equation}
We write $\mathcal{R}^\perp(E_k^{-1} \theta_k)= v_k^1+ v_k^2$ where we introduce $$v_k^1:=\mathcal{R}^\perp \left(\eta_k \chi \right)$$
for some cut-off $\chi$ between $B_{3/8}$ and $B_{1/2}\,$. Correspondingly, we write
\begin{equation}
E^V(v_k; \mu) \leq \bigg(\,\mean{Q_\mu} \lvert v_k^1 - [v_k^1(t)]_{B_\mu} \rvert^p \, dx \,dt  \bigg)^\frac{1}{p} + \bigg(\,\mean{Q_\mu} \lvert v_k^2 - [v_k^2(t)]_{B_\mu} \rvert^p \, dx \, dt\bigg)^\frac{1}{p} \, .
\end{equation}
By Calderon-Zygmund estimates, we infer that $v_k^1 \rightarrow \mathcal{R}^\perp (\eta \chi)=:v^1$ strongly in $L^p(Q_{3/4})$. {Moreover by Schauder estimates \cite[Proposition 2.8]{Silvestre2007}, we have for fixed time 
$$[v^1(t)]_{C^{1-\frac{1}{p}}(\R^2)} \leq C' \norm{\eta(t)\chi}_{C^{1-\frac 1p}(\R^2)} \leq C' \norm{\eta(t)}_{C^{1-\frac 1p}(Q_\frac{1}{2})} \leq C' \bar C$$ 
uniformly in $t\in [-(1/2)^{2\alpha}, 0]$ by \eqref{eq:holderconteta}.} We conclude
that 
\begin{equation}\label{eq:excessvk1}
\lim_{k \to \infty} \bigg(\,\mean{Q_\mu} \lvert v_k^1 - [v_k^1(t)]_{B_\mu} \rvert^p \, dx \, dt \bigg)^\frac{1}{p} = \bigg(\, \mean{Q_\mu} \lvert v^1-[v^1(t)]_{B_\mu} \rvert^p \, dx \, dt \bigg)^\frac{1}{p}\leq C' \mu^{1-\frac 1 p} \, .
\end{equation}
We now come to the excess related to $v_k^2\,.$ By construction, 
$$v_k^2 = \mathcal{R}^\perp \bigg(\frac{\theta_k}{E_k} - \eta_k \chi\bigg) = \mathcal{R}^\perp \bigg( \eta_k (1-\chi) + \frac{M_k}{E_k} \bigg)\,.$$ Correspondingly, we define $w_{k,\rho}^1:=\mathcal{R}^\perp( \eta_k(1-\chi) \chi_\rho)$ and $w_{k,\rho}^2 := \mathcal{R}^\perp \left( \frac{M_k}{E_k} \chi_\rho \right)$ for some radially symmetric cut-off $\chi_\rho$ between $B_\rho$ and $B_{\rho+1}\,.$ By Calderon-Zygmund, we have for fixed time $t$ that $w_{k,\rho}^1(t) + w_{k,\rho}^2(t) \rightarrow v_k^2(t)$ as $\rho\to \infty$ strongly in $L^p(\R^2)\, .$ In particular, 
\begin{align}
\bigg(\,\mean{Q_\mu} | v_k^2 &- [v_k^2(t)]_{B_\mu} |^p \, dx \, dt \bigg)^\frac{1}{p} = \lim_{\rho \to \infty}\bigg(\,\mean{Q_\mu} \lvert w_{k,\rho}^1 + w_{k, \rho}^2 - [(w_{k,\rho}^1 + w_{k, \rho}^2)(t)]_{B_\mu} \rvert^p \, dx \, dt \bigg)^\frac{1}{p}  \\
&\leq {\new \limsup_{\rho \to \infty}}\bigg(\,\mean{Q_\mu} \lvert w_{k,\rho}^1 - [w_{k,\rho}^1(t) ]_{B_\mu} \rvert^p \, dx \, dt \bigg)^\frac{1}{p} + {\new \limsup_{\rho \to \infty}}\bigg(\,\mean{B_\mu} \lvert w_{k, \rho}^2 - [w_{k, \rho}^2(t)]_{B_\mu} \rvert^p \, dx \, dt \bigg)^\frac{1}{p}  \, .
\end{align} 
Let us consider first $w_{k,\rho}^1$. We apply Lemma \ref{lem:boundfirstder} to $w_{k,\rho}^1$ to deduce that for fixed time $t$
\begin{align}
[w_{k,\rho}^1(t)]_{\rm Lip(B_{1/4})} &\leq C' \sup_{R \geq 1} \frac{1}{R^\sigma} \mean{B_R} \lvert \eta_k \rvert(x,t) \, dx \\
&\leq C' \bigg( \sup_{R \geq 1} \frac{1}{R^\sigma} \mean{B_R} \lvert \eta_k(x,t) - [\eta_k(t)]_{B_1} \rvert \, dx + \int_{B_1} \lvert \eta_k\rvert (x,t) \, dx \bigg)
\, .
\end{align}
Integrating in time, we infer that uniformly in $\rho \geq 1$
\begin{equation}
\norm{w_{k,\rho}^1}_{L^p([-1, 0], \rm Lip(B_{1/4}))} \leq C' \left( E^S(\eta_k;1) +E^{NL}(\eta_k;1)\right) \, .
\end{equation}
We deduce that for $\mu \in (0, \frac{1}{4})$ we have
\begin{equation}\label{eq:excessvk2}
\lim_{\rho \to \infty}\bigg(\, \mean{Q_\mu} \lvert w_{k, \rho}^1(x,t) - [w_{k,\rho}^1(t)]_{B_\mu} \rvert^p \, dx \, dt \bigg)^\frac{1}{p} \leq C' \mu \, .
\end{equation}
We now come to the contribution of $w_{k, \rho}^2$. Observe that $(-\Delta)^{1/2} \, w_{k,\rho}^2 = 0$ in $B_1$ for $\rho\geq 1$ such that $w_{k,\rho}^2$ is smooth in the inside of $B_1$. Recall moreover, that we have the integral representation (in the principal value sense)
\begin{equation}
\big(w_{k, \rho}^2\big)^\perp (x) =-\,  c\, \frac{M_k}{E_k}\int \frac{x-y}{\lvert x-y \rvert^3} \chi_\rho(y) \, dy \,
\end{equation}
so that, by spherical symmetry of $\chi_\rho$, we immediately infer $w_{k,\rho}^2(0)=0 \, .$ Moreover, for $x \in B_{1/2}$ we have
\begin{align}
\lvert w_{k, \rho}^2 (x) \rvert 
&=c \frac{\lvert M_k \rvert}{E_k} \bigg \lvert \int_{\rho < \vert x-y \rvert < \rho+1} \frac{y}{\lvert y \rvert^3} \chi_\rho(x-y) \, dy \bigg \rvert \leq c \frac{\lvert M_k \rvert}{E_k}  \pi ((\rho+1)^2 - \rho^2) \left(\frac{\rho}{2}\right)^{-2}=C' \frac{\lvert M_k \rvert}{E_k} \rho^{-1} \,.
\end{align}
Thus for fixed $k\geq 1,$ we have $$\lim_{\rho \to 0} \norm{w_{k,\rho}^2}_{L^\infty(B_{1/2})} = 0\,,$$ so that excess associated to $v_{k,2}$ is controlled by \eqref{eq:excessvk2}.
Collecting the terms \eqref{eq:excessS} and \eqref{eq:excessNL}
and taking the limes inferior $k \to \infty$ in \eqref{eq:tocontradict} we have obtained, for a universal constant $C'=C'(\alpha, \sigma)>0\,,$ that 
\begin{equation}
c{\mu}^\gamma \leq C' \mu^{\sigma- \frac{2\alpha}{p}}= {\mu}^\gamma  C'\mu^{\sigma - \frac{2\alpha}{p}-\gamma} \,.
\end{equation} 
for all $\mu \in (0, \frac{1}{4}) \, .$ We reach the desired contradiction for 
\begin{equation*}
\mu \leq \left(\frac{c}{C' } \right)^\frac{1}{\sigma - \frac{2\alpha}{p}-\gamma} \, . 
\qedhere
\end{equation*}
\end{proof}
%

\section{Iteration of the excess decay}\label{sec:iter}
In this section, we prove the decay of the excess on all scales. We iteratively define shifted rescalings of $(\theta, u)$ verifying the zero average assumption of Proposition \ref{prop:excessdecay} as well as \eqref{eq:uwithf} and therefore allowing the decay of the excess when passing at scale $\mu_0$. From the decay of the excess on all scales, we deduce H\"older continuity by means of Campanato's Theorem. In contrast to  Navier-Stokes, we need our estimates to be quantitative, since it is not known whether local smoothness for SQG follows from a mere $L^\infty$-bound; instead we need to prove spatial $C^{\delta}$-H\"older continuity of the velocity for a $\delta >1-2\alpha$ (see Lemma \ref{lem:hoeldertosmooth}). The main mechanism of the iteration is the invariance of the equation under the following change of variables realizing the zero average assumption on $B_{1/4}. $ The latter has been exploited previously in \cite{CaffarelliVasseur2010, ConstantinWu2009}. 

\begin{lemma}[Change of variables at unit scale]\label{lem:changeofvar} Let $\alpha \in (\frac{1}{4}, \frac 12)$ and $\sigma \in (0, 2\alpha)$. Let $(\theta, u)$ be a suitable weak solution of \eqref{eq:SQG} on $\R^2 \times [-4^{2\alpha},0]$. 
Fix $(x, t) \in Q_1.$ Define $\theta_0(y,s):= \theta(y+x+x_0(s), s+t)\, $ and $u_0(y,s):=u(y+x+x_0(s), s+t) - \dot x_0(s)\,,$ where 
\begin{equation}\label{eq:ODE}
\begin{cases}\dot x_0(s) &= \displaystyle \mean{B_\frac{1}{4}} u(y + x_0(s)+{\new x},s+t) \,dy \\
x_0(0)&=0 \,.
\end{cases}
\end{equation}
Then $(\theta_0, u_0)$ is a suitable weak solution of \eqref{eq:SQG} on $\R^2 \times [-1, 0]$ with $[u_0(s)]_{B_{1/4}}=0$ for $s \in[-1,0]\, .$ Moreover, there exists universal $\varepsilon_1=\varepsilon_1(p, \alpha) \in (0,\frac{1}{2})$ and $C_1\geq1$  such that if 
\begin{equation}\label{eq:changeofvarhyp}
\left(\mean{Q_1(x,t)} \lvert u \rvert^p \,dy \, ds\right)^\frac{1}{p} \leq \varepsilon_1\, ,
\end{equation}
then 
\begin{equation}
E(\theta_0, u_0 ; \frac{1}{4}) \leq C_1 E(\theta, u;x,t,1) \, .
\end{equation}
\end{lemma}
\begin{remark}[Change of variables at scale $r$]\label{rem:changeofvar} Under the hypothesis of Lemma \ref{lem:changeofvar}, the rescaled pair ($\theta_r$, $u_r$) is still a suitable weak solution for $r \in (0,1)$ (see \eqref{eq:rescaling}) and we can apply the change of variables of Lemma \ref{lem:changeofvar} to it. More precisely, we define for $(x,t) \in Q_1$
\begin{equation}
\theta_{r,0}(y,s):=r^{2\alpha-1} \theta(r(y +  x_0(s)+ x), r^{2\alpha}(s+t))
\end{equation}
and 
\begin{equation}
u_{r,0}(y,s)= r^{2\alpha -1} u(r(y + x_0(s)+x), r^{2\alpha}(s + t))-\dot x_0(s)\,,
\end{equation}
where $\dot x_0(s)=r^{2\alpha-1} \mean{B_{1/4}} u(r(y+x_0(s)+x), r^{2\alpha}(s+ t))$ with $ x_0(0)=0\,.$ Observe that equivalently, by considering $\tilde{x}_0(s):=r^{1-2\alpha} x_0(s)$, we can write 
\begin{equation}\label{eq:rescaledchangeofvar}
\theta_{r,0}(y,s):=r^{2\alpha-1} \theta(r y + r^{2\alpha} \tilde x_0(s)+r x, r^{2\alpha}(s+t))
\end{equation}
and 
\begin{equation}\label{eq:rescaledchangeofvaru}
u_{r,0}(y,s)= r^{2\alpha -1} \left( u(r y + r^{2\alpha} \tilde x_0(s)+ rx, r^{2\alpha}(s + t))-\dot{ \tilde{x}}_0(s)\right) \,.
\end{equation}
\end{remark}

\begin{proof}
By Peano, the ODE \eqref{eq:ODE} admits a solution and we claim it is unique since the vectorfield generating the flow is log-Lipschitz and hence satisifies the Osgood uniqueness criterion \cite[Chapter II.7 and  III.12.7-8]{Walter1998}. Indeed, we know from Theorem \ref{thm:LerayConstWu} that $
\theta \in L^\infty(\R^2 \times [-2, 0])$ and
$
u \in L^\infty( [-2,0],  {BMO}(\R^2)) \,.$
 We estimate for fixed time $\tau$ as long as $\lvert \xi - \zeta \rvert \leq \frac{1}{4}$, using also the bound  $\lvert B_\frac{1}{4}(\xi) \Delta B_\frac{1}{4}(\zeta) \rvert \lesssim \lvert \xi - \zeta \rvert$ on the volume of the symmetric difference, 
\begin{align}
\bigg| \, \mean{B_\frac{1}{4}(\xi)} u(y+x, \tau) \, dy  &  - \mean{B_\frac{1}{4}( \zeta)} u(y+x,\tau) \, dy \, \bigg|  \\
&\lesssim \lvert \xi-\zeta \rvert \bigg| \, \mean{B_\frac{1}{4}(\xi)\setminus B_\frac{1}{4}( \zeta) } u(y+x,\tau) \, dy- \mean{B_\frac{1}{4}(\zeta) \setminus B_\frac{1}{4}(\xi)} u(y+x,\tau) \, dy \,  \bigg| \\
&\lesssim  \lvert \xi-\zeta \rvert  \bigg( \, \mean{B_\frac{1}{4}(\xi) \setminus B_\frac{1}{4}(\zeta)}  \big| u(y+x,\tau) - [u(\cdot+x, \tau)]_{B_{3/8}(\frac{1}{2}(\xi + \zeta))}  \big| \, dy \\
& \, \hspace{5em} +\mean{B_\frac{1}{4}(\zeta) \setminus B_\frac{1}{4}(\xi)}  \big| u(y+x, \tau) - [u(\cdot +x,\tau)]_{B_{3/8}(\frac{1}{2}(\xi + \zeta))}  \big| \, dy \bigg) \, .
\end{align}
Recall from the John-Nirenberg inequality \cite{JohnNirenberg1961} that {BMO} functions are exponentially integrable, that is for every $f \in{BMO}(\R^2)$ there exist constants $c_1, c_2 >0$ such that for any ball $B$ in $\R^2$ 
$$ \lvert\{ x \in B: \lvert f - [f]_B\rvert > \lambda \} \rvert \leq c_1 \exp(- c_2 \lambda \norm{f}_{{BMO}}^{-1}) \lvert B \rvert \, .$$
As an immediate consequence, we observe that for $A \geq C \norm{f}_{{BMO}}\,$ and any ball $B$ in $\R^2$
$$ \sup_{B } \mean{B} e^{\frac{\lvert f- [f]_B \rvert}{A}} \, dx < + \infty \,.$$
We now estimate the last two contributions, setting $z:= x + \frac 12 (\xi + \zeta)$ and using Jensen
\begin{align}
\mean{B_\frac{1}{4}(\xi) \setminus B_\frac{1}{4}(\zeta)}  &\big| u(y+x, \tau) - [u(\cdot+x,\tau)]_{B_{3/8}(\frac{1}{2}(\xi + \zeta))}  \big| \, dy \\
 &\lesssim \norm{u(\tau)}_{BMO} \, \log \bigg(\lvert B_\frac{1}{4}(\xi) \setminus B_\frac{1}{4}(\zeta) \rvert^{-1} \,\mean{B_\frac{3}{8}(z)}  e^{ \frac{1}{A} \lvert u(y, \tau) - [u(\tau)]_{B_{3/8}(z)} \rvert }\, dy  \bigg) \\
& \lesssim \norm{u(\tau)}_{ {BMO}} \log ( C \lvert \xi - \zeta\rvert^{-1} ) \, .
\end{align}
We infer that the function $\xi \mapsto \mean{B_{1/4}(\xi)} u(y, s+t)$ is log-Lipschitz in space for $s \in [-1, 0]$, uniformly in time, and there is a unique solution to \eqref{eq:ODE}. 
Observe also that the dependence with respect to the point $(x,t) \in Q_1$ is log-Lipschitz. The functions $\theta_0$ and $u_0$ as in the statement are now well-defined. We remark first that, in the sense of distributions, $\div u_0= 0$ and
\begin{align}
\partial_s \theta_0 + u_0 \cdot \nabla_y \theta_0 &= (\partial_t \theta + \dot{x}_0(s) \cdot \nabla \theta) \big \vert_{(y+x+ x_0(s),s+t)} + (u \cdot \nabla \theta- \dot{x}_0(s) \cdot \nabla \theta)\big \vert_{(y +x+ x_0(s), s+t)} \\
&= (-\Delta)^\alpha \theta(y+x+x_0(s), s+t)\\
 &= (-\Delta)^\alpha \theta_0 \,,
\end{align}
so that $(\theta_0, u_0)$ is a distributional solution of \eqref{eq:SQG}. It is straightforward to check that $(\theta_0, u_0)$ is in fact a suitable weak solution. Moreover, $u_0(y,s) =\mathcal{R}^\perp \theta_0 (y,s)- \dot{x_0}(s)$ and 
$$[u_0(s)]_{B_\frac{1}{4}}= \mean{B_\frac{1}{4}} u(y+ x+x_0(s), s+t) \, dy - \dot x_0(s)= 0$$
by construction. Assume now that \eqref{eq:changeofvarhyp} holds for an $\varepsilon_1 \in (0, \frac{1}{2})$ yet to be chosen small enough. As long as $x_0(s) \in B_{3/4}$ we estimate
\begin{align}
\lvert \dot x_0(s) \rvert \leq \bigg(\, \mean{B_\frac{1}{4}+ x_0(s)} \lvert u \rvert^p(x+y, s+t) \, dy \bigg)^\frac{1}{p} \leq 4^{\frac{2}{p}}\bigg(\, \mean{B_1} \lvert u \rvert^p(x+y, s+t) \, dy \bigg)^\frac{1}{p}\,,
\end{align}
so that for $s\in (-1,0]$
\begin{equation}\label{eq:changvarcenter}
\lvert x_0(s) \rvert \leq \norm{\dot{x_0}}_{L^p((-1,0))} \lvert s \rvert^{1-\frac{1}{p}} \leq 4^\frac{2}{p} \bigg(\,\mean{Q_1(x,t)} \lvert u \rvert^p\, dy \, d \tau \bigg)^\frac{1}{p}\lvert s \rvert^{1-\frac{1}{p}}  \, .
\end{equation}
Choosing  $\varepsilon_1\leq \frac{3}{16} 4^{-\frac{2}{p}}\,$ the assumption \eqref{eq:changeofvarhyp} guarantees that $x_0(s) \in B_{3/16}\subset B_{3/4}$ for $s\in [-1,0]\, .$ We then estimate, using again that $B_{1/4}+x_0(s) \subseteq B_1$ for $s\in [-1,0]$, 
\begin{align}
E^S(\theta_0; \frac{1}{4}) &\leq \bigg(\, \mean{Q_\frac{1}{4}} \lvert \theta(y+x+x_0(s), s+t)- (\theta)_{Q_1(x,t)} \rvert^p \, dy \, ds \bigg)^\frac{1}{p} + \lvert (\theta_0)_{Q_\frac{1}{4}} - (\theta)_{Q_1(x,t)} \rvert \\
&\leq 2 \bigg(\,\mean{-(1/4)^{2\alpha}}^0 \mean{B_\frac{1}{4}+ x_0(s)} \lvert \theta(x+y, s+t)- (\theta)_{Q_1(x,t)} \rvert^p \, dy \, ds \bigg)^\frac{1}{p} \\
&\leq 2 \,4^{(2+2\alpha)\frac{1}{p}} \bigg(\, \mean{Q_1(x,t)} \lvert \theta - (\theta)_{Q_1(x,t)} \rvert^p \, dy \, ds\bigg)^\frac{1}{p} \\
&\leq 8 \, E^S(\theta;x,t,1) \, .
\end{align}
Proceeding analogously, we have that  
$E^V(u_0; \frac{1}{4}) \leq 16\, E^V(u; x,t,1) \, .$
As for the non-local part of the excess, we observe that for $R \geq \frac{1}{16}$ we have $B_R(x_0(s)) \subseteq B_{R+3/16} \subseteq B_{4R}$ and hence
\begin{align}
&E^{NL}(\theta_0; \frac{1}{4}) \leq \bigg(\, \mean{-(1/4)^{2\alpha}}^0 \sup_{R \geq \frac{1}{16}} \bigg\{ \Big(\frac{1}{4R} \Big)^{\sigma p} \bigg(\, \mean{B_R} \lvert \theta_0 -[\theta]_{B_1(x)} \rvert^\frac{3}{2} \, dy \,\bigg)^\frac{2p}{3} + \lvert [\theta_0]_{B_\frac{1}{4}} - [\theta]_{B_1(x)} \rvert^p \bigg\} \, ds \bigg)^\frac{1}{p} \\
&\leq 8 \bigg(\, \mean{-(1/4)^{2\alpha}}^0 \sup_{R \geq \frac{1}{16}} \Big(\frac{1}{4R} \Big)^{\sigma p} \bigg(\, \mean{B_{4R}} \lvert \theta(y+x, s+t) -[\theta]_{B_1(x)} \rvert^\frac{3}{2} \, dy\bigg)^\frac{2p}{3} \, ds \bigg)^\frac{1}{p} +4^{\sigma +1} E^S(\theta;x,t, 1) \\
&\leq 16 \left( E^{NL}(\theta;x,t,1) + E^S(\theta; x,t, 1) \right) \, .
\end{align}
\end{proof}

\begin{theorem}\label{prop:ereg1} Let $\alpha \in (\frac{1}{4}, \frac{1}{2})$, $\sigma \in (0, 2\alpha)$, {\new $p > \max \left\{\frac{1+\alpha}{\alpha}, \frac{2\alpha}{\sigma}\right\}$} and $\gamma \in [1-2\alpha, \sigma - \frac{2\alpha}{p})$. There exists $\varepsilon_2 \in \left(0, \frac{1}{2}\right)$ (depending only on $\alpha, \sigma, p$ and $ \gamma$) such that the following holds: Let $(\theta, u)$ be a suitable weak solution to \eqref{eq:SQG}--\eqref{eq:u} on $\R^2 \times (-4^{2\alpha},0]$.
Assume that for any $(x,t) \in Q_1$ it holds that
\begin{equation}\label{eq:ereg1ass}
E(\theta_0, u_0; \frac{1}{4}) \leq \varepsilon_2 \,,
\end{equation}
where $(\theta_0, u_0)$ is obtained from $\theta$ through the change of variables of Lemma \ref{lem:changeofvar}
. Then
\begin{equation}\label{eqn:ctheta}
\theta \in C^{\delta, \frac{p-1}{p} \delta} (Q_1)  \qquad \mbox{where} \qquad {\new \delta:= \gamma - \frac{1}{p-1}\left[1-2\alpha + \frac{2\alpha}{p} \right]\, .}
\end{equation}
\end{theorem}

\begin{remark}[Role of the parameters] \label{rem:parameter} The parameter $p$ is crucial since it determines the dimension of the singular set (see proof of Theorem \ref{thm:main}): the lower the power $p,$ the better the dimension estimate. All the other parameters are of technical nature; yet, the range of admissible parameters is sufficiently large to allow us to conclude the desired estimate on the size of the singular set for all fractional orders for which the latter is meaningful (i.e. for $\alpha> \alpha_0$, see Remark \ref{rem:alpha0}). We deliberately choose to leave all the parameters free to increase the readability of the paper; but one could also read the paper fixing the parameters as in the proof of Theorem \ref{prop:ereg2}. Let us now comment on the role of the single parameters in more detail:
\begin{itemize}
\item $p$ cannot go below the threshold $\frac{1+\alpha}{\alpha}\,:$ This corresponds to spacetime integrability that guarantees the compactness of the $(p-1)$-energy inequality (see Lemma \ref{lem:compactness}) which in turn is \emph{the} crucial ingredient of the excess decay. The requirement $p> \frac{2\alpha}{\sigma}$ on the other hand is purely technical and harmless for $\sigma$ close to $2\alpha\,.$
\item $\sigma$ captures the decay at infinity of the non-local part of both the fractional Laplacian and the velocity and should be thought arbitrarily close to $2\alpha \, .$
\item $\gamma$ describes the decay of the excess when passing to the smaller scale $\mu_0.$ In order to apply the excess decay of Proposition \ref{prop:excessdecay} iteratively, we have to verify its smallness requirement along a sequence of by $\mu:=\frac{\mu_0}{4}$ rescaled solutions which is possible only if the decay rate beats the supercritical scaling of the excess (see Remark \ref{rem:rescalingexcess}), i.e. $\gamma \geq 1-2\alpha \, .$
\item The exponent of the local H\"older continuity in space, $\delta\,,$ is obtained from $\gamma\,,$ but  considerably worsened. This stems form the fact that in order to use the decay of the excess on all scales to deduce H\"older continuity via Campanato's Theorem, we have to control the effect of the flow of Lemma \ref{lem:changeofvar}. This loss in the H\"older continuity exponent is peculiar to SQG and is not observed in the similar results for Navier-Stokes. 
\end{itemize}    
\end{remark}

\begin{proof} Let $\varepsilon_1>0$, $C_1\geq1$ be the universal constants from Lemma \ref{lem:changeofvar}. 

The proof relies on an iterative construction. We fix $(x,t) \in Q_1$. {\new We obtain the suitable weak solution $(\theta_0, u_0)$ by applying Lemma \ref{lem:changeofvar} to $(\theta, u)$ at the point $(x,t) \, .$ This first change of variables does two things: It translates $(x, t)$ to the origin $(0, 0)$ and it produces a new suitable weak solution whose velocity $u_0$ has zero average on $B_{1/4}\, .$ Hereafter, the excess will always be centered in $(0,0).$ 

Let $\mu:=\frac{\mu_0}{4} \in (0, \frac{1}{8})$ where $\mu_0, \,\varepsilon_0 \in (0, \frac{1}{2})$ are given by Proposition \ref{prop:excessdecay} with {$c=(16C_1)^{-1}$}. For $k\geq 1$ we iteratively define a new suitable weak solution $(\theta_k, u_k)$ which we obtain from $(\theta_{k-1}, u_{k-1})$ by first rescaling it at scale $\mu$ according to \eqref{eq:rescaling}, i.e. we set
\begin{equation}\label{eq:rescalingthetak-1}
\theta_{k-1, \mu}(y,s):= \mu^{2\alpha-1} \theta_{k-1}(\mu y, \mu^{2\alpha}s) \quad u_{k-1, \mu}(y,s):= \mu^{2\alpha-1} u_{k-1}(\mu y, \mu^{2\alpha}s)\,,
\end{equation}
and second, by applying the change of variables of Lemma \ref{lem:changeofvar} to $(\theta_{k-1, \mu}, u_{k-1, \mu})$ at the point $(0,0) \,.$ This change of variables produces a new suitable weak solution, which we call $(\theta_k, u_k)\,,$ that evolves along the flow $\tilde x_k$ and whose velocity $u_k$ has zero average on $B_{1/4}\,.$ 
Indeed, setting $x_k(s):= \mu^{2\alpha-1} \tilde x_k(s)\,$ (compare also with Remark \ref{rem:changeofvar}),} we define iteratively for $k \geq 1$ 
\begin{equation}\label{eq:theta_k}
\theta_k(y,s):= \theta_{k-1, \mu}(y + \tilde x_k(s),s) =\mu^{2\alpha-1} \theta_{k-1}(\mu y + \mu^{2\alpha} x_k(s), \mu^{2\alpha}s)\,
\end{equation}
and 
\begin{equation}\label{eq:u_k}
u_{k}(y,s):= u_{k-1, \mu}(y + \tilde x_k(s),s)- \dot{\tilde x}_k(s)=\mu^{2\alpha-1}( u_{k-1}(\mu y {\new +\mu^{2\alpha} x_k(s)}, \mu^{2\alpha}s ) - \dot x_k(s))\, ,
\end{equation}
where 
\begin{equation}
\begin{cases}\dot x_k(s) &= \displaystyle \mean{\mu^{2\alpha-1}x_k(s) +B_\frac{1}{4}} u_{k-1}(\mu y,\mu^{2\alpha}s) \,dy \\
x_k(0)&=0 \,.
\end{cases}
\end{equation}
Observe that by Lemma \ref{lem:changeofvar} and scaling invariance, $(\theta_k, u_k)$ are suitable weak solutions of \eqref{eq:SQG} for all $k\geq 0$ and $[u_k(s)]_{B_{1/4}}=0$ for all $s \in [-1,0]\, .$ 

Next, we want to deduce the H\"older continuity of $\theta$ assuming \eqref{eq:ereg1ass} is enforced. 
To this end, we break the parabolic scaling and we consider in Step 3 a new excess of $\theta_0$ made on modified cylinders. This in turn is helpful to get sharper estimates at the level of the change of variable, performed in Step 2, since the translation has less time to act. Finally, we rewrite this decay in Step 4 in terms of $\theta$ rather than $\theta_0$, and we apply Campanato's Theorem to deduce the H\"older continuity of $\theta$ in Step 5.

\textit{Step 1: excess decay on the sequence of solutions after the change of variable.
Let $\alpha$, $\sigma$, $p $ and $\gamma$ as in the statement. There exists a universal constant $\bar \varepsilon_2 \in (0, \frac{1}{2})$ (depending only on $\alpha$, $\sigma$, $\gamma$ and $p$) such that if $\varepsilon_2 \in (0, \bar \varepsilon_2]$ and if $(\theta,u)$ is a suitable weak solution to \eqref{eq:SQG} on $\R^2 \times (-2^{2\alpha},0]$ with \eqref{eq:ereg1ass},
then for every $k\geq 0$ the excess of $(\theta_k, u_k)$ (see \eqref{eq:theta_k}-- \eqref{eq:u_k}), decays at scale $\mu$:
\begin{align} \label{eq:it1}
E(\theta_k, u_k; \mu ) &\leq C_1^{-1}  { \mu^{\gamma(k+1)} \mu^{(2\alpha-1)k} \varepsilon_2 }\\
E(\theta_k, u_k ; \frac{1}{4} ) &\leq  \mu^{(\gamma-(1-2\alpha))k} \varepsilon_2 \label{eq:it2} \, ,
\end{align}
where $C_1$ is the universal constant from Lemma \ref{lem:changeofvar}.
}\\
We proceed by induction on $k \geq 0$. 

\textit{The case $k=0$.} Let $\varepsilon_2 \in (0,\bar \varepsilon_2]$ for some $\bar \varepsilon_2 \in (0, \frac{1}{2})$ to be chosen later and assume that \eqref{eq:ereg1ass} holds. We only need to show \eqref{eq:it1}. If 
$$ \bar \varepsilon_2 \leq 4^{2\alpha-1} \varepsilon_0\,,$$
then by Proposition \ref{prop:excessdecay} and \eqref{eq:ereg1ass}
\begin{equation}
E(\theta_0, u_0; \mu)=E(\theta_0, u_0; \frac{\mu_0}{4}) \leq C_1^{-1}  \left(\frac{\mu_0}{4}\right)^\gamma E(\theta_0, u_0; \frac{1}{4}) \leq C_1^{-1}  \mu^\gamma  \varepsilon_2 \,.
\end{equation}
\textit{The inductive step.} By the inductive hypothesis, we can assume that 
\begin{itemize}
\item $E(\theta_{k-1}, u_{k-1}; \mu) \leq C_1^{-1}  \mu^{k\gamma} \mu^{(k-1)(2\alpha-1)} \varepsilon_2 \,, $
\item $E(\theta_{k-1}, u_{k-1}; \frac{1}{4}) \leq  \mu^{(k-1)(\gamma-(1-2\alpha))} \varepsilon_2 \, .$
\end{itemize}
We recall that $(\theta_k, u_k)$ is obtained by applying the change of variables of Lemma \ref{lem:changeofvar} to $(\theta_{k-1, \mu}, u_{k-1, \mu})$ at the point $(0,0)$. Using the inductive assumption and that $[u_{k-1}(s)]_{B_{1/4}}=0$ for $s \in [-1,0]$, we can verify the smallness  assumption of Lemma \ref{lem:changeofvar}. Indeed,
\begin{align}
\bigg(\, \mean{Q_1}\lvert u_{{k-1}, \mu} \rvert^p \, dy \, ds\bigg)^\frac{1}{p} &= \mu^{2\alpha-1} \bigg(\, \mean{Q_\mu} \lvert u_{k-1} \rvert^p \, dy \, ds \bigg)^\frac{1}{p} \\
&\leq \mu^{2\alpha-1} \left(4\mu \right)^{-(2+2\alpha)\frac{1}{p}} \bigg(\, \mean{Q_\frac{1}{4}} \lvert u_{k-1} \rvert^p \, dy \, ds \bigg)^\frac{1}{p} \\
&=\mu^{2\alpha-1} \left(4\mu \right)^{-(2+2\alpha)\frac{1}{p}}  \bigg(\, \mean{Q_\frac{1}{4}} \lvert u_{k-1}(y,s) -[u_{k-1}(s)]_{B_\frac{1}{4}}\rvert^p \, dy \, ds \bigg)^\frac{1}{p} \\
&= \mu^{2\alpha-1} \left(4\mu \right)^{-(2+2\alpha)\frac{1}{p}}   E^V(u_{k-1}; \frac{1}{4})\\
 &\leq  \mu^{2\alpha-1} \left(4\mu \right)^{-(2+2\alpha)\frac{1}{p}}   \varepsilon_2 \, .
\end{align}
Choosing $\bar \varepsilon_2$ even smaller, namely, 
\begin{equation}
\bar \varepsilon_2 :=\min \big\{ ( 4^{2\alpha-1} \varepsilon_0 ,  \mu^{1-2\alpha} \left(4 \mu\right)^{(2+2\alpha)\frac{1}{p}}\varepsilon_1 \big\} 
\end{equation}
we have 
$$\bigg(\,\mean{Q_1}\lvert u_{{k-1}, \mu} \rvert^p \, dy \, ds \bigg)^\frac{1}{p} \leq \varepsilon_1 \, .$$
By Lemma \ref{lem:changeofvar}, Remark \ref{rem:rescalingexcess} and the inductive hypothesis, we deduce that
\begin{equation}
E(\theta_k, u_k;\frac{1}{4}) \leq C_1 E(\theta_{k-1, \mu}, u_{k-1, \mu}; 1) = C_1 \mu^{2\alpha-1} E(\theta_{k-1}, u_{k-1}; \mu) \leq \mu^{k(\gamma-(1-2\alpha))} \varepsilon_2 \,,
\end{equation}
showing the second inequality and, recalling the choice of $\varepsilon_2 \in (0, \bar \varepsilon_2)$, showing in particular that 
$$ E(\theta_k, u_k;\frac{1}{4}) \leq  4^{2\alpha-1} \varepsilon_0 \, .$$
Since by construction $[u_k(s)]_{B_{1/4}}=0$ for $s \in [-1,0]$, we infer from Proposition \ref{prop:excessdecay} and the inductive assumption that
\begin{equation}
E(\theta_k, u_k; \mu) \leq C_1^{-1} \mu^\gamma E(\theta_k, u_k; \frac{1}{4}) \leq C_1^{-1}   \mu^\gamma \mu^{k(\gamma-(1-2\alpha))} \varepsilon_2 = C_1^{-1}  \mu^{(k+1)\gamma} \mu^{k(2\alpha-1)} \varepsilon_2 \, .
\end{equation}

\textit{Step 2: bound on the translation in the change of variables.  We observe that $\theta_k$ is just a shifted and rescaled (by $\mu^k$, according to the natural scaling \eqref{eq:rescaling}) version of $\theta_0$. Indeed, notice that by construction, one can verify inductively for $k\geq 1$
\begin{equation}\label{eqn:theta-k-other}
\theta_k(y,s)= \theta_{0, \mu^k}(y+ \mu^{-k} r_k(s), s) \, ,
\end{equation}
where {\new $\theta_{0, \mu^k}(y,s):= \mu^{k(2\alpha-1)} \theta_0(\mu^k y, \mu^{2\alpha k} s)$ and }
\begin{equation}
r_k(s) := \mu^{2\alpha-1}\sum_{j=1}^k \mu^j x_j(\mu^{2\alpha(k-j)} s) \, .
\end{equation}
We claim that the center of the cylinders don't move too much, namely for $s\in [-1,0]$
\begin{equation}
\lvert r_k(s) \rvert \leq C  \varepsilon_2\lvert s\rvert^{1-\frac{1}{p}}\mu^{2\alpha k(1-\frac 1 p)-(1-2\alpha+ \frac 2 p)} \,. 
\end{equation}
}
Indeed, for $j\geq 1$ we estimate as long as $x_j(s) \in B_{1/4}$
\begin{align}
\lvert \dot x_j(s) \rvert &\leq \bigg(\,\mean{\mu^{2\alpha} x_j(s) + B_\frac{\mu}{4}} \lvert u_{j-1}(y, \mu^{2\alpha}s )\rvert^p \, dy \bigg)^\frac{1}{p}\\
&\leq \mu^{-\frac{2}{p}} \bigg(\, \mean{B_\frac{1}{4}} \lvert u_{j-1}(y, \mu^{2\alpha}s)- [u_{j-1}(\mu^{2\alpha}s)]_{B_\frac{1}{4}} \rvert^p \, dy \bigg)^\frac{1}{p} \,,
\end{align}
where we used that $[u_{j-1}(\mu^{2\alpha}s)]_{B_{1/4}} =0$ uniformly in time. In particular, 
\begin{equation}
\norm{\dot x_j}_{L^p((-1,0))} \leq \mu^{-\frac{2}{p}} 4^{-\frac{2\alpha}{p}} E^V(u_{j-1};\frac{1}{4}) \leq \mu^{-\frac{2}{p}} E(\theta_{j-1},u_{j-1};\frac{1}{4})
\end{equation}
and hence for $s\in [-1,0]$ we have, using \eqref{eq:it2}, $$ \lvert x_j(s) \rvert \leq  \mu^{-\frac{2}{p}} E(\theta_{j-1},u_{j-1};\frac{1}{4}) \lvert s \rvert^{1-\frac{1}{p}}  \leq  \varepsilon_2 \lvert s \rvert^{1-\frac{1}{p}} \mu^{-\frac{2}{p}} \, .$$
Collecting terms, we have 
\begin{align}
\lvert r_k(s) \rvert &\leq  \varepsilon_2\lvert s\rvert^{1-\frac{1}{p}} \mu^{2\alpha k(1-\frac 1 p)}\mu^{2\alpha-1} \mu^{-\frac{2}{p}} \sum_{j=1}^k \mu^{(1-2\alpha(1-\frac 1 p ))j} \leq C   \varepsilon_2\lvert s\rvert^{1-\frac{1}{p}}\mu^{2\alpha k(1-\frac 1 p)-(1-2\alpha+ \frac 2 p)}   \, .
\end{align}

\textit{Step 3: Decay of a modified excess of $\theta_0$.  We claim that for every $r\in (0, \mu^2)$
\begin{equation}\label{eq:Campanato}
 \bigg(\, \mean{- r^{\frac{p}{p-1}}}^0 \mean{B_{\frac{r}{4}}} \lvert \theta_0 - (\theta_0)_{B_{\frac{r}{4}}\times (- r^\frac{p}{p-1},0]} \rvert^p \, dy \, ds\bigg)^\frac{1}{p} \leq  8C_1^{-1}   \, \mu^{-2} r^{\gamma-\left[\frac{1-2\alpha}{p-1} +\frac{2\alpha}{p(p-1)}\right]}  \varepsilon_2 \, .
\end{equation}}\\
Observe that by the scaling of the excess $\mu^{2\alpha-1} E^S(\theta_k;\mu)= E^S(\theta_{k, \mu};1)$ and by \eqref{eqn:theta-k-other} 
\begin{align}
\theta_{k, \mu}(y,s)= \theta_{0, \mu^{k+1}}(y+ \mu^{-(k+1)} r_k(\mu^{2\alpha}s), s) \,.
\end{align}
We introduce the set  $$I_{k+1}:=\Big(-\mu^{(k+1)((1-2\alpha)\frac{p}{p-1} + \frac{2\alpha}{p-1})}, 0\Big] \,.$$
If $ s \in I_{k+1}$ we can ensure, by an appropriate choice of $\varepsilon_2$, that $r_k(\mu^{2\alpha} s) \in B_{3\mu^{k+1}/4}\,.$ Indeed
\begin{equation}
\lvert r_k(\mu^{2\alpha} s) \rvert  \leq C  \varepsilon_2 \mu^{2\alpha (k+1)(1-\frac 1 p)-(1-2\alpha+ \frac 2 p)}  \mu^{(k+1)(1-2\alpha)} \mu^{(k+1) \frac{2\alpha}{p}}=  C  \varepsilon_2 \mu^{k+1}  \mu^{-(1-2\alpha+ \frac 2 p)} \,,
\end{equation}
for $ s \in I_{k+1}$ by Step 2. It is thus enough to choose $\varepsilon_2$ (if necessary) even smaller, or more precisely, we set $$\varepsilon_2 := \min \left\{ \bar \varepsilon_2, \frac{3}{4}C^{-1}\mu^{(1-2\alpha+ \frac 2 p)} \right\} \, .$$
We now estimate, by adding and substracting $ (\theta_{k,\mu})_{Q_1}$ and H\"older
\begin{align}
\bigg(\, \mean{I_{k+1}} \mean{B_{\frac{1}{4}} } \lvert \theta_{0,\mu^{k+1}}(y,s) - (\theta_{0,\mu^{k+1}})_{B_\frac{1}{4}\times I_{k+1} } \rvert^p \bigg)^\frac{1}{p} 
\leq 2 \bigg(\, \mean{I_{k+1}} \mean{B_{\frac{1}{4}}} \lvert \theta_{0,\mu^{k+1}}(y,s) - (\theta_{k,\mu})_{Q_1} \rvert^p \, dy \, ds \bigg)^\frac{1}{p} \, .
\end{align}
Since $\mu^{-(k+1)}r_k(\mu^{2\alpha}s) \in B_{3/4}$ we have $B_{1/4} \subseteq \mu^{-(k+1)} r_k(\mu^{2\alpha}s)+B_1$ as long as $s\in I_{k+1}$, so that 
\begin{align}
 \bigg(\,\mean{I_{k+1}} \mean{B_\frac{1}{4}} \lvert \theta_{0,\mu^{k+1}}(y,s) &- (\theta_{k, \mu} )_{Q_1}  \rvert^p \bigg)^\frac{1}{p} 
\\& \leq 4^\frac{2}{p} \bigg(\,\mean{I_{k+1}} \mean{\mu^{-(k+1)} r_k(\mu^{2\alpha}s)+B_1} \lvert \theta_{0,\mu^{k+1}}(y,s) - (\theta_{k, \mu})_{Q_1}  \rvert^p \, dy \, ds \bigg)^\frac{1}{p} \\
 &= 4^{\frac{2}{p}} \bigg(\,\mean{I_{k+1}} \mean{B_1} \lvert \theta_{0,\mu^{k+1}}(y + \mu^{-(k+1)} r_k(\mu^{2\alpha}s),s) - (\theta_{k, \mu})_{Q_1}  \rvert^p \, dy \, ds \bigg)^\frac{1}{p} \
 \\
 &\leq 4^{\frac{2}{p}} \lvert I_{k+1} \rvert^{-\frac{1}{p}} \bigg(\, \mean{Q_1} \lvert \theta_{k, \mu} (y, s) - (\theta_{k, \mu})_{Q_1} \rvert^p \, dy \, ds \bigg)^\frac{1}{p} \\
 &= 4^{\frac{2}{p}} \lvert I_{k+1} \rvert^{-\frac{1}{p}}  E^S(\theta_{k, \mu};1) =  4^{\frac{2}{p}} \lvert I_{k+1} \rvert^{-\frac{1}{p}} \mu^{2\alpha-1} E^S(\theta_k;\mu)
 \, .
\end{align}
Combining the previous inequality with Step 1 and observing that $\mu^{(k+1)2\alpha} I_{k+1}= \big(- \mu^{(k+1)\frac{p}{p-1}},0\big]$, we deduce that for $k \geq 1$
\begin{align}
&\hspace{-5em}\bigg(\, \mean{\mu^{(k+1)2\alpha} I_{k+1}} \mean{B_{\frac 1 4 \mu^{k+1}}} \lvert \theta_0(y,s) - (\theta_0)_{B_{\frac{1}{4}\mu^{k+1}} \times \mu^{(k+1)2\alpha} I_{k+1} } \rvert^p \ dy \, ds\bigg)^\frac{1}{p}\\
&= \mu^{(k+1)(1-2\alpha)}\bigg(\, \mean{I_{k+1}}  \mean{B_{\frac{1}{4}} } \lvert \theta_{0,\mu^{k+1}}(y,s) - (\theta_{0,\mu^{k+1}})_{B_\frac{1}{4}\times I_{k+1}} \rvert^p \bigg)^\frac{1}{p}   \\
&\leq 4 C_1^{-1}   \lvert I_{k+1} \rvert^{-\frac{1}{p}} \mu^{(k+1)\gamma} \varepsilon_2 =  4 C_1^{-1}  \mu^{(k+1)\left(\gamma-\left[\frac{1-2\alpha}{p-1} +\frac{2\alpha}{p(p-1)}\right]\right)}  \varepsilon_2\, .
\end{align}
This gives \eqref{eq:Campanato} for $r= \mu^{k+1}$ for some $k\geq 1\,.$ For $r\in (\mu^{k+2}, \mu^{k+1})$ instead, we observe that 
\begin{align}
 \bigg(\, \mean{- r^{\frac{p}{p-1}}}^0 \mean{B_{\frac{r}{4}}} & \lvert \theta_0  -  (\theta_0)_{B_{\frac{r}{4}}\times (- r^\frac{p}{p-1},0]} \rvert^p \, dy \, ds\bigg)^\frac{1}{p} \\
& \leq  2 \left( \frac{\mu^{k+1}}{r}\right)^{\frac 1p(2+ \frac{p}{p-1})}\bigg(\, \mean{- \mu^{(k+1)\frac{p}{p-1}}}^0 \mean{B_{\frac{1}{4} \mu^{k+1}}} \lvert \theta_0 - (\theta_0)_{B_{\frac{1}{4} \mu^{k+1}}\times (- \mu^{(k+1)\frac{p}{p-1}},0]} \rvert^p \, dy \, ds\bigg)^\frac{1}{p}  \\
&\leq 8C_1^{-1}   \, \left( \frac{\mu^{k+1}}{r}\right)^{\frac 1p(2+ \frac{p}{p-1})} \mu^{(k+1)\left(\gamma-\left[\frac{1-2\alpha}{p-1} +\frac{2\alpha}{p(p-1)}\right]\right)}  \varepsilon_2\\
&\leq 8C_1^{-1}   \, \mu^{-2} r^{\left(\gamma-\left[\frac{1-2\alpha}{p-1} +\frac{2\alpha}{p(p-1)}\right]\right)}\varepsilon_2 \,.
\end{align}

\textit{Step 4: Decay of a modified excess of $\theta$. There exists a $r_0=r_0(\norm{u}_{L^{p+1}(Q_{3/2})})>0$ such that for every $r\in (0, r_0)$ and for every $(x,t) \in Q_1$
\begin{equation}
 \bigg(\, \mean{t- r^{\frac{p}{p-1}}}^t \mean{B_{\frac{r}{8}}(x)} \lvert \theta - (\theta)_{B_{\frac{r}{8}}(x)\times (t- r^\frac{p}{p-1},t]} \rvert^p \, dy \, ds\bigg)^\frac{1}{p} \leq  32 \, \mu^{-2} r^{\gamma-\left[\frac{1-2\alpha}{p-1} +\frac{2\alpha}{p(p-1)}\right]} \varepsilon_2 \,.
\end{equation}
}\\
Since by Theorem \ref{thm:LerayConstWu} $u\in L^\infty([-(3/2)^{2\alpha}, 0],  {BMO}(\R^2))\,,$ we have $u \in L^q_{loc}(\R^2 \times [-(3/2)^{2\alpha}, 0])$ for any $q \in [1, \infty)\,.$ Fix $(x,t) \in Q_1$. As long as $x_0(s) \in B_{1/4}$ and $ \lvert s \rvert <{\new \frac{1}{5}}\,,$ we have the estimate
\begin{equation}\label{eq:flowest}
\lvert x_0(s) \rvert \leq \lvert s \rvert^{1-\frac{1}{p+1}} \norm{u}_{L^{p+1}(Q_{3/2})} \,.
\end{equation}
In particular for $0 \leq \lvert s \rvert^\frac{p}{p+1}  \leq \min \left\{\frac{1}{4} \norm{u}_{L^{p+1}(Q_{3/2})}^{-1},{\new 5^{-\frac{p}{p+1}}}\right\}$ the estimate \eqref{eq:flowest} holds. Let now  $$r_0:= \min \bigg\{\mu^2, \bigg( \frac 14 \norm{u}_{L^{p+1}(Q_{3/2})}^{-1} \bigg)^{1- \frac{1}{p^2}}, \bigg( \frac{1}{8} \norm{u}_{L^{p+1}(Q_{3/2})}^{-1} \bigg)^{p^2-1} \bigg \}\,.$$ 
{\new Recalling that $\mu^2 \leq \frac{1}{64}\,,$} we observe that for all $r \in (0, r_0)$, $(x,t) \in Q_1$ and $s\in (-r^\frac{p}{p-1}, 0]$
\eqref{eq:flowest} holds and we have 
\begin{equation}\label{eq:flowest2}
\lvert x_0(s) \rvert \leq \norm{u}_{L^{p+1}(Q_{3/2})} r^\frac{p^2}{p^2-1} \leq r \left(\norm{u}_{L^{p+1}(Q_{3/2})} r_0^\frac{1}{p^2-1}\right) {\new \leq } \frac{r}{8}\,.
\end{equation}
Hence we can estimate by the triangular inequality and H\"older, by \eqref{eq:flowest2} and by Step 3
\begin{align}
\bigg(\,\mean{t-r^\frac{p}{p-1}}^t \mean{B_\frac{r}{8}(x)} \lvert \theta &-  (\theta)_{B_\frac{r}{8}(x) \times (t-r^\frac{p}{p-1}, t]} \rvert^p \,dy \, ds \bigg)^\frac{1}{p} \\
&\leq 2 \bigg(\,\mean{-r^\frac{p}{p-1}}^0 \mean{B_\frac{r}{8}(x)} \lvert \theta(y,s+t)- (\theta_0)_{B_\frac{r}{4} \times (r^\frac{p}{p-1}, 0]} \rvert^p \,dy \, ds \bigg)^\frac{1}{p} \\
&\leq  4 \bigg(\,\mean{-r^\frac{p}{p-1}}^0 \mean{B_\frac{r}{4}(x + x_0(s))} \lvert \theta(y,s+t)- (\theta_0)_{B_\frac{r}{4} \times (r^\frac{p}{p-1}, 0]} \rvert^p \,dy \, ds \bigg)^\frac{1}{p} \\
&\leq 32\, \mu^{-2} r^{\gamma - \left[ \frac{1-2\alpha}{p-1} + \frac{2\alpha}{p(p-1)}\right]} \varepsilon_2 \, .
\end{align}

\textit{Step 5: By Campanato's Theorem, we deduce that $\theta$ is H\"older continuous in $Q_1$.}\\
By a variant of Campanato's Theorem \cite[Theorem 2.9.]{Giusti2003}, we deduce from Step 4 that \eqref{eqn:ctheta} holds. {\new  Indeed, observe that the sets $B_r(x) \times (t-r^{p/(p-1)}, t]$ are nothing else but balls with respect to the metric  $d((x,t), (y,s)):= \max \{ \lvert x-y \rvert, \lvert t-s \rvert^{(p-1)/p} \}$ on spacetime where in time, as usual for parabolic equations, we only look at backward-in-time intervals. The proof of this version of Campanato's Theorem follows, for instance, line by line \cite[Theorem 1]{Schlag1996} when replacing the parabolic metric by $d\,.$}

\end{proof}

\section{$\varepsilon$-regularity results and proof of Theorem \ref{prop:ereg2}}\label{sec:ereg}
{\new
In this section we prove some $\varepsilon$-regularity results, including Theorem~\ref{prop:ereg2} and its more precise version in Corollary~\ref{prop:fixedscale}, by meeting the smallness requirement of Theorem \ref{prop:ereg1}. As a first result, we deduce in Corollary~\ref{thm:Scheffer} an $\varepsilon$-regularity criterion in terms of a spacetime integral of $\theta$ and $u$ that constitutes an analogue of Scheffer's Theorem~\cite{Scheffer77} for the Navier-Stokes system. As in the case of Navier-Stokes, it implies that the singular set of suitable weak solutions is compact in spacetime (see Step 1 in the proof of Theorem~\ref{thm:main}). In the context of the SQG equation though, in contrast to Navier-Stokes, Corollary~\ref{thm:Scheffer} cannot be used to obtain their almost everywhere smoothness (or any estimate on the dimension of the singular set): The fact that the $L^\infty$-norm is a controlled quantity 
necessitates to rely on spacetime integrals of derivatives of $\theta$ to show local smoothness. In order to pass from Theorem \ref{prop:ereg1} to an $\varepsilon$-regularity criterion involving only fractional space derivatives of $\theta\,,$which are globally controlled through the energy, we need to overcome the following  difficulties:
\begin{itemize}
\item The excess $E^S\,$ related to $\theta$ involves the spacetime average of $\theta\,.$ In particular, in order to use a standard Poincar\'e inequality \eqref{eq:PoincareStd} to pass to a differential quantity, we would need some fractional differentiability in time too. Using the parabolic structure of the equation, we will be able to circumvent this and to establish in Lemma \ref{lem:nonlinearPoincare} a Poincar\'e inequality which is nonlinear but involves only fractional space derivatives. 
\item The $\varepsilon$-regularity criterion of Theorem \ref{prop:ereg1} features the composition of $\theta$ with the flow $x_0\,,$ so that we need some control on the tilting effect of the flow. We will see that at scale $r$, the flow shifts the center of the excess in space by at most $r^{2\alpha - 2/q} \norm{u}_{L^\infty L^q}\,$ (see \eqref{eq:effectflow}). As a consequence, at scale $r$, all quantities related to the excess of $\theta$  will no longer live on parabolic cylinders but rather on modified cylinders $\mathcal{Q}(x,t;r)$ in spacetime, approximately of radius $r^{2\alpha}$ in time and $r^{2\alpha-2/q}$ in space.
Morally $q= \infty$; however, since the Riesz-transform is bounded from $L^\infty \rightarrow {BMO}$ and not from $L^\infty \rightarrow L^\infty$, we introduce the parameter $q$ which should be thought to be arbitrarily large. 
\item We set the excess in $L^p$ for $p > \frac{1+\alpha}{\alpha}$ in order to gain the compactness of the $(p-1)$-energy inequality. In order to exploit the $L^2 W^{\alpha,2}$-control given by the energy via the nonlinear Poincar\'e inequality described in the first point, we are lacking some higher integrability in time. We bypass this issue by factoring out $p-2$ powers of $\theta$ in $L^\infty$.
\end{itemize} 
}

\subsection{An analogue of Scheffer's theorem}\label{sec:Scheffer}{\new We provide a first $\varepsilon$-regularity result featuring spacetime integrals of $\theta$ and $u\,.$ Observe that in agreement with the previous discussion, smooth solutions of \eqref{eq:SQG}--\eqref{eq:u} do, in general, not verify the $\varepsilon$-regularity criterion \eqref{eq:eregcritscheffer} at any small scale.}
\begin{corollary}\label{thm:Scheffer} Given $\alpha \in (\frac{1}{4}, \frac{1}{2})$ there exists  $\varepsilon=\varepsilon(\alpha)>0$ such that if $p=p(\alpha):= \frac{6}{4\alpha-1}$ and $\theta$ is a suitable weak solution of \eqref{eq:SQG}--\eqref{eq:u} on $\R^2 \times(t-(4r)^{2\alpha} , t+ r^{2\alpha}/4]$ satisfying 
\begin{equation}\label{eq:eregcritscheffer}
\frac{1}{r^{(1-2\alpha)p + 2+2\alpha}} \int_{Q_{4r} (x, t + r^{2\alpha}/4)} \left({\mathcal M}\theta^2 \right)^{\frac{p}{2}}(z,s) + \lvert u \rvert^p(z,s) \, dz \, ds<\varepsilon\,,
\end{equation}
then $\theta\in C^{\frac{1}{2}}\left(Q_r (x, t + r^{2\alpha}/4)\right)\,.$ In particular, $\theta$ is smooth in the interior of $Q_{r/2}(x, t + r^{2\alpha}/4)\supseteq B_{r/2}(x) \times (t-r^{2\alpha}/4, t + r^{2\alpha}/4)$ and hence $(x,t)$ is a regular point. 
\end{corollary}

\begin{proof}[Proof of Corollary \ref{thm:Scheffer}]
Let $\alpha$ and $p$ as in the statement. By translation and scaling invariance, we can assume w.l.o.g. that $(x, t + r^{2\alpha}/4)=(0,0)$ and $r=1\,,$ so that we assume that
\begin{equation}\label{eq:schefferass}
\int_{Q_4} (\mathcal{M} \theta^2)^\frac{p}{2}(x,t) + \lvert u \rvert^p(x,t)\, dx \, dt \leq \varepsilon  \,.
\end{equation}
 for an $\varepsilon$ yet to be chosen small enough. We observe that $p > \max \{\frac{1+\alpha}{\alpha}, \frac{2\alpha}{\sigma} \}\,$ for any $\sigma > \frac{1}{6} \, .$  We define accordingly $\sigma:= \frac{4\alpha}{3} + \frac{1}{6}$ and  $\gamma:=\frac{2\alpha}{3}+\frac{1}{3}\,.$ Since $\sigma \in (\frac{1}{6}, 2\alpha)$ and $\gamma \in [1-2\alpha , 2\alpha - \frac{2\alpha}{p}) \,,$ this is an admissible choice of the parameters of Theorem \ref{prop:ereg1} and we infer from the latter that $\theta \in C^{\delta, \frac{p-1}{p} \delta}(Q_1)\,,$ with $\delta$ given by \eqref{eqn:ctheta}, provided the smallness requirement \eqref{eq:ereg1ass} holds for any $(x,t) \in Q_1$. Since for $\alpha \in (\frac{1}{4}, \frac{1}{2})$
$$\delta=1- 2\alpha +(4\alpha-1) \left( \frac{2}{3}-\frac{4\alpha^2-7\alpha+3}{3(7\alpha-1)} \right) >1-2\alpha  \,,$$
we deduce from Lemma \ref{lem:hoeldertosmooth} that $\theta$ is smooth in the interior of $Q_{1/2} \, .$ We observe that for any $(x,t) \in Q_1$ we have $Q_1(x,t) \subseteq Q_4$ and hence
\begin{equation}
 \bigg( \, \mean{Q_1(x,t)} \lvert u \rvert^p \,dz \, ds \bigg)^\frac{1}{p} \leq \lvert Q_1 \rvert^{-\frac{1}{p}} \bigg(\int_{Q_4} \lvert u \rvert^p \, dz \,ds \bigg)^\frac{1}{p} \leq \varepsilon^\frac{1}{p}\, .
\end{equation}
Requiring $\varepsilon \leq \varepsilon_1^p\,,$ 
we thus deduce from Lemma  \ref{lem:changeofvar} that $E(\theta_0, u_0; \frac{1}{4}) \leq C_1 E(\theta,u; x,t, 1)$ and hence \eqref{eq:ereg1ass} is enforced for any $(x,t) \in Q_1$ if 
\begin{equation}\label{eq:varepsilon2}
\sup_{(x,t) \in Q_1} \, E(\theta,u; x,t, 1) \leq C_1^{-1} \varepsilon_2\,,
\end{equation}
where $\varepsilon_2>0$ is given by Theorem \ref{prop:ereg1}. Using $  \theta ^2
\leq\mathcal{M}  \theta ^2 $ pointwise almost everywhere, we have
\begin{equation}
E^S(\theta; x, t, 1) + E^V(u; x, t, 1) \leq 2 \left( \int_{Q_4} \lvert \theta \rvert^p \, dz \,ds \right)^\frac{1}{p}+ 2 \left( \int_{Q_4} \lvert u \rvert^p \, dz \, ds \right)^\frac{1}{p} \leq 4\varepsilon^\frac{1}{p} \,.
\end{equation}
As for the non-local part of the excess, we estimate, using again $\theta^2 \leq \mathcal{M} \theta^2$ almost everywhere,
\begin{align}
E^{NL}(\theta; x, t, 1) &\leq \Bigg( \,\mean{t-1}^{t} \sup_{R \geq \frac 14} \frac{1}{R^{\sigma p}} \bigg(\, \mean{B_R(x)} \theta^2 \, dz\bigg)^\frac{p}{2} \, ds \Bigg)^\frac{1}{p} + C\lvert (\theta)_{Q_1(x, t)} \rvert \\
&\leq C \Bigg(\, \mean{t-1}^{t} \sup_{R \geq \frac 14} \frac{1}{R^{\sigma p}} \bigg(  \int_{B_\frac{1}{4}(x)}\mean{B_{2R}(z')} \theta^2(z,s)  \, dz \, dz' \bigg)^\frac{p}{2} \, ds \Bigg)^\frac{1}{p} + C \bigg(\int_{Q_4} \lvert \theta \rvert^p \, dz \, ds \bigg)^\frac{1}{p} \\
&\leq C \Bigg( \mean{t-1}^{t} \sup_{R \geq \frac 14} \frac{1}{R^{\sigma p}} \bigg(  \int_{B_\frac{1}{4}(x)}\mathcal{M}\theta^2(z', s) \, dz' \bigg)^\frac{p}{2} \, ds \Bigg)^\frac{1}{p} + C \varepsilon^\frac{1}{p} \leq C^{NL} \varepsilon^\frac{1}{p} \, .
\end{align}
Hence we reach \eqref{eq:varepsilon2} by choosing $
\varepsilon \leq \min \left\{ ( \varepsilon_1)^p ,  (4+ C^{NL})^{-1}  C_1^{-1}\varepsilon_2)^p \right\} \,.$
\end{proof}

{\new 
\subsection{Nonlinear Poincar\'e inequality of parabolic type}\label{sec:poincare}
We introduce the following scaling-invariant quantity which should be understood as a localized version of the dissipative part of the energy:
\begin{align}
\mathcal{E}(\theta; x,t,r) &:= \frac{1}{r^{2(1-2\alpha) + 2}} \int_{Q_r^*(x,t)} y^b \lvert \overline \nabla \theta^* \rvert^2(z,y,s)\, dz \, dy \, ds \, .
\end{align}
The following Lemma and its proof is inspired by \cite{Struwe1981}, where a parabolic Poincar\'e inequality is obtained for the classical, linear heat equation, and by \cite{OzanskiRobinson2019}, where  a nonlinear Poincar\'e inequality of similar nature is also crucially used in a $\varepsilon$-regularity result. 
\begin{lemma}[Nonlinear Poincar\'e inequality of parabolic type]\label{lem:nonlinearPoincare} Let $\alpha \in (0,1) \,.$ There exists a constant $C=C(\alpha)\geq 1$ such that the following holds: Let $Q_r(x,t) \subseteq \R^2 \times (0, \infty)$ and let $(\theta, u)$ be a Leray-Hopf weak solution of \eqref{eq:SQG}--\eqref{eq:Cauchy}. We assume that the velocity field is obtained by 
\begin{equation}
u(z,s) = \mathcal{R}^\perp \theta (z,s) + f(s) \, 
\end{equation}
for some $f \in L^1_{loc}(\R)$ and that it satisfies $\left[u(s)\right]_{B_r(x)} = 0$ for all $s \in [t-(2r)^{2\alpha}, t]$. Then we have for any $q \in \big[2, \frac{2}{1-\alpha} \big]$ that
\begin{align}
\frac{1}{r^{(1-2\alpha) + \frac{2}{q} + \alpha}} &\Bigg(\int_{t- r^{2\alpha}}^{t} \bigg(\int_{B_r(x)} \lvert \theta(z,s) - (\theta)_{Q_r(x,t)} \rvert^q \, dz \bigg)^\frac{2}{q} \, ds \Bigg)^\frac{1}{2} 
\leq  C \big( \mathcal{E}(\theta; x,t, 3r)^\frac{1}{2} \\
&+\bigg( \frac{1}{r^{2(1-2\alpha)+2+2\alpha}} \int_{Q_{2r}(x,t)} \lvert u (z,s)- [u(s)]_{B_{2r}(x)} \rvert^2 \, dz \, ds \bigg)^\frac{1}{2} \mathcal{E}(\theta; x,t, 3r)^\frac{1}{2}  \big) \, .
\end{align}
\end{lemma}

\begin{proof} By translation and scaling invariance (with respect to \eqref{eq:rescaling}), we may assume $(x,t)=(0,0)\,$ and $r=1\,.$

\textit{Step 1: By means of the weighted Poincar\'e inequality \eqref{eq:PoincareWgt}, we reduce the Lemma to an estimate on weighted space averages computed at two different times. To this aim, let $\omega \in C^\infty_{c}(\R^{3}_+  )$ be a weight such that $\omega|_{y=0}$ is a radial non-increasing function, $0 \leq \omega \leq 1$ and $\omega \equiv 1$ on $\overline B_{1} \times [0, 1]$ and $\omega \equiv 0$ outside $B_{2} \times [0, 2)\,.$}\\ We estimate for fixed time
\begin{align}
\left( \int_{B_1} \lvert \theta(x,t) - (\theta)_{Q_1} \rvert^q \, dx \right)^\frac{1}{q} &\leq \left( \int_{B_{2}} \lvert \theta(x,t) -[\theta(t)]_{\omega|_{y=0}, B_2} \rvert^q \omega(x,0) \,dx \right)^\frac{1}{q} \\
&+ \pi^\frac{1}{q} \left(\big| [\theta(t)]_{\omega|_{y=0}, B_2}-(\theta)_{\omega|_{y=0}, Q_2} \big|+ \big|  (\theta)_{\omega|_{y=0}, Q_2} -(\theta)_{Q_1} \big| \right) \,,
\end{align} 
where we used $\omega(\cdot, 0) \equiv 1$ on $\overline B_1\,.$ Reusing this fact and H\"older, we bound the last term by 
\begin{align}
\big| &(\theta)_{\omega|_{y=0}, Q_2} - (\theta)_{Q_1}\big|  \leq \left(\int_{-1}^0 \left( \mean{B_1} \lvert \theta -  (\theta)_{\omega|_{y=0}, Q_2} \rvert^q \,dx \right)^\frac{2}{q} \, dt\right)^\frac{1}{2}  \\
&\leq \pi^{-\frac{1}{q}} \left(\int_{-1}^0 \left( \int_{B_{2}} \lvert \theta -  [\theta(t)]_{\omega|_{y=0}, B_2} \rvert^q \omega(x,0) \, dx \right)^\frac{2}{q} \, dt \right)^\frac{1}{2}  + \left(\int_{-1}^0 \lvert [\theta(t)]_{\omega|_{y=0}, B_2}- (\theta)_{\omega|_{y=0}, Q_2} \rvert^2 \, dt\right)^\frac{1}{2} \,,
\end{align}
so that we deduce by the weighted Poincar\'e inequality \eqref{eq:PoincareWgt} 
\begin{align}
\left(\int_{-1}^0 \left( \int_{B_1} \lvert \theta(x,t)- (\theta)_{Q_1} \rvert^q \, dx \right)^\frac{2}{q} \, dt \right)^\frac{1}{2} \leq C   \Bigg[ &\left( \int_{-1}^0 [\theta(t)]_{W^{\alpha, 2}(B_{2})}^2 \, dt \right)^\frac{1}{2} \\
&+  \left( \int_{-1}^0 \lvert [\theta(t)]_{\omega|_{y=0}, B_2}- (\theta)_{\omega|_{y=0},Q_2} \rvert^2 \, dt \right)^\frac{1}{2}\Bigg] \, .
\end{align}
The first term on the right-hand side can be expressed in terms of the extension by \eqref{eq:GagliardoToExt}. Since the weight $\omega$ is independent of time, the second term can be estimated by
\begin{equation}
\left( \int_{-1}^0 \lvert [\theta(t)]_{\omega|_{y=0}, B_2}- (\theta)_{\omega|_{y=0}, Q_2} \rvert^2 \, dt \right)^\frac{1}{2} \leq \left( \int_{-2^{2\alpha}}^0 \int_{-2^{2\alpha}}^0 \lvert [\theta(t)]_{\omega|_{y=0}, B_2} - [\theta(s)]_{\omega|_{y=0}, B_2} \rvert^2\, ds\, dt \right)^\frac{1}{2} \,.
\end{equation}

\textit{Step 2: We use the equation to estimate the difference of two weighted space averages computed at different times.}\\
We use the weak formulation \eqref{eq:weakform3} of the equation with time-independent test function $\varphi(x,y):= \sign([\theta(t)]_{\omega|_{y=0}, B_2} -[\theta(s)]_{\omega|_{y=0}, B_2}) \omega(x, y)\,.$ We estimate the right-hand side of \eqref{eq:weakform3} from below and the left-hand side from above for $s,t \in [-2^{2\alpha}, 0]$. As for the lower bound, we have
\begin{align}\label{eq:poincare0}
\int (\theta(x,t) - \theta(x,s)) \varphi(x, 0) \, dx &=\big| [\theta(t)]_{\omega|_{y=0}, B_2} -[\theta(s)]_{\omega|_{y=0}, B_2}\big| \int_{B_{2}} \omega(x,0) \, dx \nonumber \\
&\geq \pi  \big| [\theta(t)]_{\omega|_{y=0}, B_2} -[\theta(s)]_{\omega|_{y=0}, B_2}\big|
\end{align}
since $\omega(\cdot, 0) \equiv 1$ on $\overline B_1\,.$ As for the right-hand side, we estimate by H\"older 
\begin{align} \label{eq:poincare1}
\bigg  \lvert \int_{s}^t \int_{\R^3_+} y^b \overline \nabla \theta^* \cdot \overline \nabla \varphi\, dx \, dy \, d\tau \bigg \rvert &\leq 2^{1-\alpha} \norm{\omega}_{C^1} \lvert Q_{2} \rvert^\frac{1}{2} \left( \int_{Q_2^*} y^b \lvert \overline \nabla \theta^*\rvert^2 \, dx \, dy \, d\tau \right)^\frac{1}{2} \, .
\end{align}
Since $u$ is divergence-free and $[u(\tau)]_{B_1}=0$ for $\tau \in [ -2^{2\alpha}, 0]$ by assumption, we can estimate the nonlinear term by H\"older and the Poincar\'e inequality \eqref{eq:PoincareStd} combined with \eqref{eq:GagliardoToExt}
\begin{align}\label{eq:poincare2}
\left \lvert \int_{s}^t \int u\theta \cdot \nabla \varphi|_{y=0} \, dx \, d \tau \right \rvert  &=\left \lvert \int_{s}^t \int (u- [u(\tau)]_{B_1})(\theta- [\theta(\tau)]_{B_2}) \cdot \nabla \varphi|_{y=0}  \, dx \, d \tau \right \rvert \nonumber\\
&\leq \norm{\omega}_{C^1} \left( \int_{Q_{2}} \lvert u- [u(\tau)]_{B_1} \rvert^2 \, dx \, d\tau \right)^\frac{1}{2} \left( \int_{Q_{2} } \lvert \theta - [\theta(\tau)]_{B_2} \lvert^2 \, dx \, d\tau \right)^\frac{1}{2} \nonumber \\
&\leq C \left( \int_{Q_{2}} \lvert u- [u(\tau)]_{B_2} \rvert^2 \, dx \, d\tau  \right)^\frac{1}{2} \left( \int_{Q_3^*} y^b \lvert \overline{\nabla} \theta^* \rvert^2 \, dx \, dy \, d\tau \right)^\frac{1}{2} \, .
\end{align}
Collecting terms, we have for almost every $s,t \in [-2^{2\alpha}, 0]$ that
\begin{align}
\big| [\theta(t)]_{\omega|_{y=0}, B_2} &- [\theta(s)]_{\omega|_{y=0}, B_2} \big| \lesssim \Bigg(\int_{Q_{3}^*} y^b \lvert \overline \nabla \theta^* \rvert^2 \, dx \, dy \, d \tau \Bigg)^\frac{1}{2}\Bigg(1 + \bigg(\int_{Q_{2}} \lvert u - [u(\tau)]_{B_2} \rvert^2 \,dx \, d\tau \bigg)^\frac{1}{2} \Bigg) \, .
\end{align}
Combining this estimate with Step 1, we conclude.
\end{proof}

{\new
\subsection{The non-local part of excess}\label{sec:nonloc}
We recall from the proof of Proposition~\ref{prop:excessdecay} that the excess related to the velocity can be estimated in terms of $\theta\,.$ More precisely, we have the following:
\begin{lemma}\label{lem:excessofu}  Let $\alpha \in (\frac{1}{4}, \frac{1}{2}),$ $Q_r(x,t) \subseteq \R^2 \times (0, \infty)$ and $\theta \in L^p(\R^2 \times [t-(3r/2)^{2\alpha}, t])\,.$ Consider a velocity field of the form $u(z,s)= \mathcal{R}^\perp \theta(z,s) + f(s)$ for some $f \in L^1_{loc}(\R)\,.$ There exists $C=C(p)\geq 1$ such that
\begin{equation}\label{eq:excessofu}
E^V(u; x, t, r) \leq C \Bigg( \bigg( \, \mean{Q_{\frac{3r}{2}}(x,t)} \lvert \theta(z,s)- [\theta(s)]_{B_{\frac{3r}{2}}(x)} \rvert^p \, dz \,ds \bigg)^\frac{1}{p}+  E^{NL}(\theta; x, t, \frac  32 r) \Bigg)\,.
\end{equation}
\end{lemma}
After reducing the Lemma to $r=1$ and $(x,t)=(0,0)$, the proof follows line-by-line the estimate of $E^V$ in the proof of Proposition \ref{prop:excessdecay}. We now bound the quantity $E^{NL}$ in terms of a variant of the sharp maximal function introduced in Section \ref{sec:maximal}. 
\begin{lemma}\label{lem:non-localexcess} Let $\alpha \in (\frac{1}{4}, \frac{1}{2})\,$, $Q_r(x,t) \subseteq \R^2 \times (0, \infty)\,$ and $\theta \in L^\infty(\R^2 \times [t-r^{2\alpha}, t]) \cap L^2([t-r^{2\alpha}, t], W^{\alpha,2}(\R^2))\,.$ If $(p, q, \sigma) \in (\frac{1+\alpha}{\alpha}, \infty)  \times [2, \infty) \times (0, 2\alpha)$ satisfy the admissibility criterion
\begin{equation}\label{eq:psigmaq}
\sigma p-(2+ 2\alpha) - \frac{2}{q^2-1} \geq 0 \, ,
\end{equation}
then there exists a constant $C= C(p)\geq 1$ such that
\begin{equation}
\frac{1}{r^{p(1-2\alpha)}}E^{NL}(\theta; x, t, r)^p \leq C \frac{\norm{\theta}_{L^\infty([t-r^{2\alpha}, t] \times \R^2)}^{p-2}}{r^{p(1-2\alpha)+2}} \int_{t-r^{2\alpha}}^{t} \int_{B_\frac{r}{4}(x)} \lvert \theta^\#_{\alpha, q}(z,s) \rvert^{1+\frac{1}{q^2-1}}\, dz \,ds \, .
\end{equation}
\end{lemma}

\begin{proof} By translation and scaling invariance, we may assume $(x,t)=(0,0)$ and $r=1$. By factoring out $p-2$ powers in $L^\infty\,,$ by adding and subtracting $[\theta(s)]_{B_{1/4}}$ for fixed time $s$ and radius $R$, and by reabsorbing $\lvert [\theta(s)]_{B_1}- [\theta(s)]_{B_{1/4}} \rvert$ in the supremum, we have
\begin{align}
E^{NL}(\theta; 1)^p &\leq (2 \norm{\theta}_{L^\infty(\R^2 \times [-1, 0])})^{p-2} \,  \int_{-1}^0 \sup_{R \geq \frac{1}{4}} \left( \frac{1}{R} \right)^{\sigma p} \left(\, \mean{B_R} \lvert \theta(x,s)- [\theta(s)]_{B_1}\rvert^\frac{3}{2} \, dx \right)^{\frac{4}{3} } \, ds \\
&\lesssim    \norm{\theta}_{L^\infty(\R^2 \times [-1, 0])}^{p-2} \, \int_{-1}^0 \sup_{R \geq \frac 14} \left( \frac{1}{R} \right)^{\sigma p} \left(\,\mean{B_R} \lvert \theta(x,s)- [\theta(s)]_{B_\frac{1}{4}}\rvert^\frac{3}{2} \, dx \right)^\frac{4}{3} \, ds \, .
\end{align}
We estimate the argument of the supremum for fixed time $s$ and radius $R \geq \frac{1}{4} \, $ by the triangular inequality and H\"older
\begin{align}\label{eq:bli}
\bigg(\, \mean{B_R} \lvert \theta - [\theta(s)]_{B_\frac{1}{4}} \rvert^\frac{3}{2} \, dx \bigg)^\frac{4}{3} \nonumber
&\leq 2 \Bigg( \bigg( \,\mean{B_R} \lvert \theta - [\theta(s)]_{B_R} \rvert^{2(1- 1/q^2)} \, dx \bigg)^{1+ \frac{1}{q^2-1}}  + \lvert [\theta(s)]_{B_\frac{1}{4}}- [\theta(s)]_{B_R} \rvert^2 \Bigg)  \nonumber \\
&\leq 4 \left( (4R)^2 \mean{B_R} \lvert \theta- [\theta(s)]_{B_R} \rvert^{2(1- 1/q^2)} \, dx \right)^{1+ \frac{1}{q^2-1}} \,.
\end{align}
For $z \in B_{1/4}$ it holds $B_R \subseteq B_{2R}(z)$, so that by the triangular inequality and by averaging over $z \in B_{1/4}\,,$ we have 
\begin{align}\label{eq:bla}
\mean{B_R} \lvert \theta(x,s)- [\theta(s)]_{B_R}\rvert^{2(1-1/q^2)}\, dx &\leq 4 \, \, \mean{B_\frac{1}{4}} \mean{B_{2R}(z)} \lvert \theta(x,s)- [\theta(s)]_{B_{2R}(z)}\rvert^{2(1-1/q^2)}\, dx \, dz \nonumber \\
&\leq  32 \, R^{2\alpha(1-1/q^2)}\, \mean{B_\frac{1}{4}} \theta^\#_{\alpha, q}(z,s) \, dz \, .
\end{align}
We combine \eqref{eq:bli}--\eqref{eq:bla} and use H\"older to bring the power $1+\frac{1}{q^2-1}$ inside the integral. We obtain 
\begin{equation}
E^{NL}(\theta;1) \leq C \norm{\theta}_{L^\infty(\R^2 \times [-1, 0])}^{p-2} \,  \int_{-1}^0  \int_{B_\frac{1}{4}} \lvert \theta^\#_{\alpha, q}(z,s) \rvert^{1+\frac{1}{q^2-1}} \, dz \, ds \, ,
\end{equation}
provided 
\begin{equation}\label{eq:supfinite}
\sup_{R \geq \frac{1}{4}} \left(\frac{1}{R}\right)^{\sigma p - (2+2\alpha) -\frac{2}{q^2-1}} < + \infty\,.
\end{equation}
Observe that this is ensured through \eqref{eq:psigmaq} and that the supremum can be estimated from above by $4^{p}$ such that the constant of the Lemma depends only on $p\,.$
\end{proof}
}

\subsection{Proof of the Theorem \ref{prop:ereg2}}\label{sec:fixedscale} The following Corollary of Theorem \ref{prop:ereg1} gives a different version of the $\varepsilon$-regularity criterion in terms of $\theta$ rather than its composition with the flow $\theta_0\,.$ Theorem \ref{prop:ereg2} will be an immediate consequence. To this aim, we introduce the following modified balls and cylinders (backwards and centered in time respectively) which are enlarged in space in accordance with the ``intrinsic effect" of the flow (see \eqref{eq:effectflow}):
\begin{align}
&\mathcal{B}(x;r):= B_{K_q r^{2\alpha- 2/q}}(x) \, \text{ and } \mathcal{B}^*(x;r):= \mathcal{B}(x;r) \times [0, r) \,, \\
&\mathcal{Q}(x,t;r) := \mathcal{B}(x;r)\times (t-r^{2\alpha}, t] \, \text{ and } \mathcal{Q}^*(x,t;r) := \mathcal{B}^*(x;r)\times (t-r^{2\alpha}, t]\,,\\
&\mathcal{C}(x,t;r) := \mathcal{B}(x;r)\times (t-r^{2\alpha}, t+r^{2\alpha}) \, \text{ and } \mathcal{C}^*(x,t;r) := \mathcal{B}^*(x;r) \times (t-r^{2\alpha}, t+ r^{2\alpha}) \,,
\end{align}
where 
\begin{equation}\label{eq:Kq}
K_q=K_q(u; x,t,r) : = 2 \max \left \{\norm{u}_{L^\infty([t-r^{2\alpha}, t+r^{2\alpha}], L^q(\R^2))}, r^{1-2\alpha+ 2/q} \right\} \, \,.
\end{equation}
To shorten notation, we will often omit the dependence of $K_q$ on $u$ and $r\,,$ and whenever $(x,t)=(0,0)\,,$ we will omit to specify the center of the balls and cylinders. The following remark justifies that one should really think of $\mathcal{B}(x;r)$ as a enlarged balls of radius $r^{2\alpha - 2/q}\,. $ 

\begin{remark}[Upper bound on $K_q$]\label{rem:upperboundKq} For $0 < r \leq (t/2)^{\frac{1}{2\alpha}}$ by Calderon-Zygmund, Theorem \ref{thm:LerayConstWu} and the energy inequality \eqref{e:g_energy_1} of Leray-Hopf weak solutions
\begin{equation}
\norm{u}_{L^\infty([t-r^{2\alpha}, t+r^{2\alpha}], L^q(\R^2))} \leq C \norm{\theta}_{L^\infty([t-r^{2\alpha}, t+r^{2\alpha}], L^q(\R^2))} \leq C_2 \norm{\theta_0}_{L^2} t^{-\frac{1}{2\alpha}(1- 2/q)}
\end{equation}
such that for $0<r \leq r_0:= \min \left \{  (t/2)^{\frac{1}{2\alpha}},  \left(  C_2 \norm{\theta_0}_{L^2} t^{-\frac{1}{2\alpha}(1- 2/q)}\right)^{1/(1-2\alpha + 2/q)} \right \}$ and any $x \in \R^2$
\begin{equation}\label{eq:upperboundKq}
K_q(u; x,t,r) \leq 2C_2 \norm{\theta_0}_{L^2} t^{-\frac{1}{2\alpha}(1- 2/q)} \,.
\end{equation}
\end{remark}
{\new 

\begin{corollary}\label{prop:fixedscale} Let $\alpha_0:=  \frac{1+\sqrt{33}}{16}$, $\alpha \in [ \alpha_0, \frac 12)\,,$ $q \geq 8 \,$ and  $p=p(q):= \frac{1+\alpha}{\alpha} + \frac 1q \,.$ There exists a universal $\varepsilon=\varepsilon(\alpha)\in (0,1)$ such that the following holds: If $(\theta, u)$ is a suitable weak solution to \eqref{eq:SQG}--\eqref{eq:u} on $\R^2 \times [t- (4r)^{2\alpha}, t]$ satisfying
\begin{equation}\label{eq:hypfixedscale}
\frac{\norm{\theta}_{L^\infty(\R^2 \times [t-(4r)^{2\alpha}, t])}^{p-2}}{(4r)^{p(1-2\alpha)+ 2}}   \left( \int_{\mathcal Q^*(x,t;4r)} y^b \lvert \overline \nabla \theta^* \rvert^2 \, dz \, dy \,ds + \int_{\mathcal Q(x,t;4r)} \lvert \theta^\#_{\alpha,q}  \rvert^{1 + \frac{1}{q^2-1}} \, dz \, ds \right) \leq \varepsilon \,,
\end{equation}
then $\theta$ is smooth in the interior of $Q_{r/2}(x,t)\,.$ 
\end{corollary}
\begin{remark}[Justification of $\alpha_0$]\label{rem:alpha0} $\alpha_0$ is the threshold until which both the smallness hypothesis of Corollary \ref{prop:fixedscale} is verified, at sufficiently small scale, for smooth solutions at any point $(x,t)$ in spacetime and the dimension estimate of Theorem  \ref{thm:main} is non-trivial. Indeed, for $\alpha> \alpha_0$ it holds
\begin{equation}\frac{1}{2\alpha} \Big(\frac{1+\alpha}{\alpha}(1-2\alpha)+ 2 \Big) < 3 \,.
\end{equation}
\end{remark}

Before proceeding with the proof, let us show that Theorem \ref{prop:ereg2} is an immediate consequence of Corollary \ref{prop:fixedscale}. 

\begin{proof}[Proof of Theorem \ref{prop:ereg2}] Let $\alpha\,, p$ and $q$ as in the statement and assume that \eqref{eq:blabla} holds. Observe that $\mathcal{C}(x,t; r) \supseteq \mathcal{Q}(x, t + r^{2\alpha}/16; r)\,.$ By \eqref{eq:pointwiseestforsharpmaximalfunction} and H\"older we deduce the pointwise estimate
\begin{equation}\label{eq:bliblablu}
\theta^\#_{\alpha, q}(z,s) \leq\mathcal{M}\big( \left(D_{\alpha, 2} \,\theta \right)^{2(1-1/q^2)}\big)(z,s) \leq \big[\mathcal{M}\big( \left(D_{\alpha, 2} \,\theta \right)^2 \big)\big]^{1-1/q^2}(z,s) \,.
\end{equation}
We infer that $\theta$ satisfies \eqref{eq:hypfixedscale} at the radius $r/4$ and the point $(x, t+ r^{2\alpha}/16)\,.$ We deduce from Corollary \ref{prop:fixedscale} that $\theta$ is smooth in the interior of $Q_{r/8}(x, t+r^{2\alpha}/16)$ which contains the open ball $B_{r/8}(x) \times (t-r^{2\alpha}/16, t+r^{2\alpha}/16) \, .$
\end{proof}

\begin{proof}[Proof of Corollary \ref{prop:fixedscale}] By translation and scaling invariance, we assume $(x,t)=(0,0)$ and $r=1\, .$

\textit{Step 1: We tune the free parameters $\sigma$ and $\gamma$.}\\
We define $\sigma := 2\alpha - \frac{1}{q^2}\,.$ Observe that with this choice the triple $(\sigma, p, q)$ satisfies the hypothesis of Lemma \ref{lem:non-localexcess}. Indeed, recalling that by assumption $\alpha \geq \alpha_0 > \frac{2}{5}$, we have for $q \geq 8$ that 
\begin{align}
\sigma p - (2+2\alpha) - \frac{2}{q^2-1} =  \frac{1}{q}\left(2\alpha- \frac{1}{q}\left( \frac{1+\alpha}{\alpha} + \frac{1}{q} +\frac{2q^2}{q^2-1}\right)\right) > \frac{1}{q}\left(2\alpha - \frac{5}{4}\right) >0 \, .
\end{align}
We introduce also $\gamma:=2\alpha - \frac{4\alpha^2}{1+\alpha} \in [1-2\alpha, \sigma - 2\alpha/p)$, so that the triple $(\sigma, p, \gamma)$ satisfies the hypothesis of Theorem \ref{prop:ereg1}. We conclude from the latter that $\theta \in C^{\delta, (1-1/p) \delta}(Q_r) \subseteq L^\infty ( (-1,0); C^{\delta}(B_1))\,,$  with $\delta$ given by \eqref{eqn:ctheta},
provided
\begin{equation}\label{eq:smallnessreq}
\sup_{(x,t) \in Q_1} E(\theta_{0}, u_{0}; \frac{1}{4}) \leq \varepsilon_2\,,
\end{equation}
where for $(x,t) \in Q_1$ fixed, we define $\theta_{0}(z,s)= \theta(z+x_0(s)+x, s+t)\,,$ $u_{0}(z,s)=u(z+x_0(s)+x, s+t) - \dot x_0(s)$ and $x_0$ is the flow given by Lemma \ref{lem:changeofvar}. We have
$$\delta= \gamma - \frac{1}{p-1} \left[ 1-2\alpha + \frac{2\alpha}{p} \right] > 2\alpha -\frac{4\alpha^2}{1+\alpha} -  \alpha \left[ 1-2\alpha + \frac{2\alpha^2}{1+\alpha} \right]\,, $$
where the right-hand side exceeds $1-2\alpha$ for $\alpha\geq \sqrt{2}-1$, so in particular for $\alpha \geq \alpha_0\,.$
We deduce from Lemma \ref{lem:hoeldertosmooth} that $\theta$ is smooth in the interior of $Q_{1/2}\,.$ We are thus left to verify that \eqref{eq:smallnessreq} can be enforced by requiring \eqref{eq:hypfixedscale}.

\textit{Step 2: We bound the full excess $\sup_{(x,t) \in Q_1}E(\theta_0, u_0; \frac 14)\,.$}\\
For $(x,t) \in Q_1\,$ fixed, we estimate by factoring out $p-2$ powers in $L^\infty$, by Lemma \ref{lem:nonlinearPoincare}, by H\"older and Young
\begin{align}
E^S(\theta_{0}; \frac{1}{4})^p &\lesssim \norm{\theta_{0}}_{L^\infty(Q_\frac{1}{4})}^{p-2} \mathcal{E}(\theta_{0}; \frac{3}{4}) + \norm{\theta_{0}}_{L^\infty(Q_\frac{1}{4})}^{p-2} \mathcal{E}(\theta_{0}; \frac{3}{4}) E^V(u_{0}; \frac{1}{2})^2 \\
&\lesssim \norm{\theta_{0}}_{L^\infty(\R^2 \times [-1, 0])}^{p-2} \mathcal{E}(\theta_{0}; 1) +\left( \norm{\theta_{0}}_{L^\infty(\R^2 \times [-1, 0])}^{p-2} \mathcal{E}(\theta_{0}; 1) \right)^{1+\frac{2}{p-2}} 
+ E^V(u_0; \frac{1}{2})^p \, .
\end{align}
Moreover, by Lemma \ref{lem:excessofu} and \eqref{eq:PoincareStd} combined with \eqref{eq:GagliardoToExt}, recalling that  $u_{0}=\mathcal{R}^\perp \theta_{0}- \dot x_0(s)$, we have
\begin{align}
E^V(u_{0}; \frac{1}{2})^p &\lesssim \int_{Q_\frac{3}{4}} \lvert \theta_{0} - [\theta_0(s)]_{B_\frac{3}{4}} \rvert^p\, dz \, ds + E^{NL}(\theta_{0}; \frac{3}{4})^p \lesssim  \norm{\theta_{0}}_{L^\infty(Q_\frac{3}{4})}^{p-2} \mathcal{E}(\theta_{0}; 1) + E^{NL}(\theta_{0}; \frac{3}{4})^p \, .
\end{align}
Collecting terms and applying Lemma \ref{lem:non-localexcess}, we deduce that for every $(x,t) \in Q_1$ with a constant $C=C(\alpha)\geq 1$ ( observe that $p$ cannot exceed $\frac{1+\alpha}{\alpha}+1$)
\begin{align}\label{eq:Step1}
E(\theta_{0}, u_{0}; \frac 14)^p 
\leq C \bigg(  \norm{\theta_{0}}_{L^\infty(\R^2 \times [-1, 0])}^{p-2} &\mathcal{E}(\theta_{0}; 1) +\left( \norm{\theta_{0}}_{L^\infty(\R^2 \times [-1, 0])}^{p-2} \mathcal{E}(\theta_{0}; 1) \right)^{1+\frac{2}{p-2}} \nonumber \\
&+ \norm{\theta_0}_{L^\infty(\R^2 \times [-1, 0])}^{p-2} \int_{-1}^0 \int_{B_\frac{1}{4}} |(\theta_0)^\#_{\alpha, q}|^{1 + \frac{1}{q^2-1}} \, dz \, ds \bigg)
\,.
\end{align}

\textit{Step 3: We estimate the tilting effect of the flow. To this aim, introduce a parameter $q \geq 8$ (to ensure that $2\alpha -2/q\geq 1/4>0$).\\}
For $(x,t) \in Q_1$ and $s \in [-1, 0]$ we have by the definition of the flow \eqref{eq:ODE}
\begin{equation}
\lvert \dot x_0(s) \rvert \leq \lvert B_\frac{1}{4} \rvert^{-\frac{1}{q}} \norm{u}_{L^\infty( [t-1, t ], L^q(\R^2))} \leq 2 \norm{u}_{L^\infty([ -4^{2\alpha},0], L^q(\R^2))}\leq  K_q(u;4) \, .
\end{equation}
Hence the center of the excess in space can be shifted by at most
\begin{equation}\label{eq:effectflow}
\sup_{(x,t) \in Q_1} \sup_{s \in [-1, 0]} \lvert x_0(s) \rvert \leq K_q(u;4)  \,.
\end{equation} 
Since $\theta_{0}$ is just a spatial translation of $\theta$, we estimate  
$\norm{\theta_{0}}_{L^\infty(\R^2 \times [-1, 0])} \leq \norm{\theta}_{L^\infty(\R^2 \times [- 4^{2\alpha},0])}  \,$ for $(x,t) \in Q_1\,.$ Recall that the extension is obtained by convolution with a Poisson kernel. Since translation and convolution commute, we have $(\theta_0)^*(z,y,s)= \theta^*(z+ x_0(s)+x, y,s+t)$ and hence
\begin{align}\label{eq:newdiffquant}
\mathcal{E}(\theta_{0}; 1)
&=\int_{Q_1^*} y^b \lvert \overline \nabla \theta^* \rvert^2 (z+ x_0(s)+x, y,s+t) \, dz \, dy \, ds \leq \int_{\mathcal{Q}^*(4)} y^b \lvert \overline{\nabla} \theta^* \rvert^2 (z,y, s) \, dz \, dy \, ds \,.
\end{align}
We used in the last inequality that for $(x,t) \in Q_1$ it holds $t-1\geq - 4^{2\alpha}$ and
\begin{equation}\label{eq:runningoutofnames}
 B_1(x) + x_0(s) \subseteq B_{2} + B_{K_q(u;4)} \subseteq \frac 14 \mathcal{B}(4) + \frac{1}{4^{2\alpha- 2/q}} \mathcal{B}(4) \subseteq   \mathcal{B}(4)
\end{equation}
for any $s \in [-1, 0]$ by \eqref{eq:effectflow}. As for the remaining term in \eqref{eq:Step1}, we observe that $(\theta_0)^\#_{\alpha,q} (z,s) = \theta^\#_{\alpha, q}(z+x_0(s)+x, s+t)$ and we reuse \eqref{eq:runningoutofnames} to estimate 
\begin{align}\label{eq:newtails}
\int_{-1}^0 \int_{B_\frac{1}{4}} |(\theta_0)^\#_{\alpha, q}|^{1 + \frac{1}{q^2-1}} \, dz \, ds
&\leq \int_{\mathcal{Q}(4)}|\theta^\#_{\alpha, q}|^{1 + \frac{1}{q^2-1}} \, dz \, ds \,.
\end{align}
Combining \eqref{eq:Step1}, \eqref{eq:newdiffquant} and \eqref{eq:newtails}, we reach \eqref{eq:smallnessreq} by requiring \eqref{eq:hypfixedscale}\,.
\end{proof}
}
}

\section{The singular set and proof of Theorem \ref{thm:main}}\label{sec:dimension}

We recall the box-counting dimension of a (compact) set $\mathcal S \subseteq \R^3\,:$ For every $\delta \in (0,1)$ we denote by $N(\delta)$ the minimal number of sets of diameter $\delta$ needed to cover $\mathcal S\,.$ We then define
$$\dim_b(\mathcal{S}) := \limsup_{\delta \to 0} - \log_\delta ( N(\delta)) \, .$$
It is well-known that the box-counting dimension controls the Hausdorff dimension $\dim_\mathcal{H}\,,$ i.e. $\dim_\mathcal{H} \mathcal{S} \leq \dim_b \mathcal{S} \, .$

\begin{proof}[Proof of Theorem \ref{thm:main}] \textit{Step 1: Fix $t>0$ and define $\mathcal{S}:= {\rm Sing} \, \theta \cap \left[\R^2 \times [t, \infty) \right]\,.$ We claim that $\mathcal{S} $ is a compact set in spacetime. }\\
From the definition it is clear that $\mathcal{S}$ is closed and we claim that it is also bounded. Indeed, let $p$ be as in Corollary \ref{thm:Scheffer} and let $p' \geq p$. From the maximal function estimate and Calderon-Zygmund 
\begin{align}\label{eqn:torecall}
\int_t^\infty \int (\mathcal{M} \theta^2)^\frac{p'}{2} + \lvert u \rvert^{p'} \, dx \, ds &\lesssim \int_t^\infty \int \lvert \theta \rvert^{p'} \, dx \, ds \lesssim \norm{\theta}_{L^\infty(\R^2 \times [t, \infty))}^{p'-2(1+\alpha)} \norm{\theta}_{L^{2(1+\alpha)}(\R^2 \times [t, \infty))}^{2(1+\alpha)} 
\\&\leq C(\|\theta_0\|_{L^2}) t^{- \frac{p'-2(1+\alpha)}{2\alpha}} \,,
\end{align}
by Theorem \ref{thm:LerayConstWu} and the global energy inequality \eqref{e:g_energy_1} of Leray-Hopf weak solutions. By the absolute continuity of the integral we deduce that for every $\varepsilon>0$ there exists $M=M(\theta, \varepsilon)>0$ and $T^*=T^*(\|\theta_0\|_{L^2}, \varepsilon)>0$ such that
\begin{equation}
\int_t^\infty \int_{\R^2 \setminus B_M}(\mathcal{M} \theta^2)^\frac{p}{2} + \lvert u \rvert^p \, dx \, ds  +\int_{T^*}^\infty \int_{\R^2}(\mathcal{M} \theta^2)^\frac{p}{2} + \lvert u \rvert^p \, dx \, ds  < \varepsilon \, ,
\end{equation}
which, by choosing $\varepsilon$ as in Corollary \ref{thm:Scheffer}, means that $\mathcal{S} \subset B_{M+4}(0) \times [0, T^*+4^{2\alpha}] \, .$

{\new
\textit{Step 2:  Let $\alpha \in (\alpha_0, \frac{1}{2})$ (otherwise the dimension estimate is trivial by Remark \ref{rem:alpha0}). We show that for every $q\geq 8$ we have 
\begin{equation}\label{eq:betaq}
\dim_{b} \mathcal S \leq \frac{1}{2\alpha - \frac{2}{q}} \left(\left( \frac{1+\alpha}{\alpha} + \frac{1}{q}\right)(1-2\alpha) +2  \right)=: \beta (q)  \, .
\end{equation}}
Indeed, fix $q \geq 8$ and define $p_q:= \frac{1+\alpha}{\alpha} + \frac{1}{q} \, .$  From Corollary \ref{prop:fixedscale}, we know that if  $(x,s) \in \mathcal{S}\,,$ then for every $r \in (0, (t/2)^\frac{1}{2\alpha})$ it holds
\begin{equation}
 \frac{1}{r^{p_q(1-2\alpha)+2}} \int_{\mathcal{C}^*(x,s; r)} y^b \lvert\overline{\nabla} \theta^* \rvert^2 \, dz \, dy \, d\tau + \int_{\mathcal C(x,s;r)} \lvert \theta^\#_{\alpha, q} \rvert^{1+\frac{1}{q^2-1}} \, dz \, d \tau  > \varepsilon \norm{\theta}_{L^\infty( \R^2 \times [t/2, \infty))}^{-(p-2)} =: 2\varepsilon_3 \,,
\end{equation}
where $\varepsilon=\varepsilon(\alpha)>0$ is universal and in particular independent of $r$. By Theorem \ref{thm:LerayConstWu}, the threshold $\varepsilon_3$ depends on $t>0$ and $\norm{\theta_0}_{L^2}$ only.
 Following Remark \ref{rem:upperboundKq}, we observe that with the notation of Remark \ref{rem:upperboundKq}
\begin{equation}\label{eq:Lq}
K_q(u; x,s,r) \leq \max\left  \{ 2C_2 \norm{\theta_0}_{L^2} t^{-\frac{1}{2\alpha}(1-2/q)}, 1 \right\}=: L_q
\end{equation}
 for $r \in (0, \delta_0] \, ,$ where
\begin{equation}
\delta_0:= \min \left \{ (t/2)^\frac{1}{2\alpha}, (L_q/2)^{1/(1-2\alpha+2/q)} ,1 \right \} \, .
\end{equation}
Observe that $L_q$ depends only on $\norm{\theta_0}_{L^2}$ and $t>0\,.$ For $\delta \in (0, \delta_0)\,,$ we define the collection $\Gamma_\delta$ containing balls $ B_{\sqrt{2} L_q \delta^{2\alpha-2/q}}(x,s)$ centered at some point $(x,s) \in \R^2 \times [t, \infty)$ satisfying 
$$\int_{ \mathcal{C}^*(x,s;\delta)} y^b \lvert\overline{\nabla} \theta^* \rvert^2 \, dz \, dy \, d\tau + \int_{\mathcal{C}(x,s;\delta)} \lvert \theta^\#_{\alpha, q} \rvert^{1 + \frac{1}{q^2-1}} \, dz \, d\tau \geq \varepsilon_3\, \delta^{p_q(1-2\alpha)+2} \,.$$
Observe that $\{\Gamma_\delta \}_\delta$ is a family of coverings of $\mathcal{S} \,$ consisting of Euclidean balls in spacetime. By the Vitali covering Lemma, there exists for every $\delta$  a countable, disjoint family $\{ B^i  \}_{i \in I}$ such that 
$$\mathcal{S} \subseteq \bigcup_{ i \in I } 5B^i \,$$
and $B^i \in \Gamma_\delta\,,$ in particular $B^i=B_{\sqrt{2} L_q \delta^{2\alpha- 2/q}}(x_i,s_i)$  for some $(x_i, s_i) \in \R^2 \times [t, \infty) \,.$ Observe that by Lemma \ref{lem:sharpmax}, Theorem \ref{thm:CF} and the global energy inequality \eqref{e:g_energy_1} for Leray-Hopf weak solutions, we have the global control 
\begin{equation}
\int_0^\infty \int_{\R^3_+} y^b \lvert \overline{\nabla} \theta^\ast \rvert^2 \, dz \, dy \, ds + \int_0^\infty \int_{\R^2} \lvert \theta^\#_{\alpha, q} \rvert^{1+ \frac{1}{q^2-1}} \, dz \, ds \leq C \int_0^\infty  [\theta(s)]_{W^{\alpha,2}(\R^2)}^2 \, ds \leq C \norm{\theta_0}_{L^2}^2 \, .
\end{equation}
Therefore, setting $\eta:= 2\sqrt{2} L_q \delta^{2\alpha - 2/q}$ and using the disjointness of $\{B^i\}_{i \in I}\,,$ we can estimate the minimal number $N(\eta)$ of sets of diameter $\eta$ needed to cover $\mathcal{S}$ by
\begin{equation}
N(\eta) \leq \mathcal{H}^0(I) \leq \frac{C \norm{\theta_0}_{L^2}^2}{\varepsilon_3 \delta^{p_q(1-2\alpha) +2}} = \frac{C \norm{\theta_0}_{L^2}^2(\sqrt 2 L_q)^{\frac{p_q(1-2\alpha) +2}{2\alpha - 2/q}}}{\varepsilon_3 \eta^{\frac{p_q(1-2\alpha) +2}{2\alpha - 2/q}}} \,.
\end{equation}
We conclude that 
\begin{equation}
\dim_{b} \mathcal{S} =\limsup_{\eta \to 0} -\log_{\eta} N(\eta) \leq  \frac{p_q(1-2\alpha) +2}{2\alpha - 2/q} \, .
\end{equation}

\textit{Step 3: Conclusion.\\}
By taking the limit $q \to \infty\,$ in Step 2, we conclude that
 $$\dim_{b}({\rm Sing} \, \theta \cap [t, \infty)) \leq \frac{1}{2\alpha} \left( \frac{1+\alpha}{\alpha} (1-2\alpha) + 2\right) \, $$ 
  for every $t>0$.
Writing ${\rm Sing} \, \theta = \bigcup_{n \geq 1} {\rm Sing} \, \theta \cap \big[\R^2 \times \big[\frac{1}{n}, \infty\big)\big]\,,$ we deduce that 
$$\dim_\mathcal{H}({\rm Sing}\, \theta) \leq \sup_{n \geq 1} \, \dim_\mathcal{H}\left({\rm Sing} \, \theta \cap \big[\R^2 \times \big[\tfrac 1n, \infty\big)\big] \right) \leq \frac{1}{2\alpha} \left( \frac{1+\alpha}{\alpha} (1-2\alpha) + 2\right) \, .\qedhere$$
}
\end{proof}

\section{Stability of the singular set}\label{sec:stability}
This section is devoted to the proof of Corollary \ref{cor:stability}. 
We observe that the $\varepsilon$-regularity criterion is ``continuous" under strong $L^p$-convergence. This convergence result together with the observation that smooth solutions satisfy the $\varepsilon$-regularity criterion of Theorem \ref{prop:ereg1} will allow to deduce the required stability.

\begin{lemma}\label{lem:continuity} Let $\alpha_n \in (\frac{1}{4}, \frac{1}{2})$ be such that $\alpha_n \rightarrow \alpha \in (\frac{1}{4}, \frac{1}{2}]$ and consider a sequence of suitable weak solutions $\theta_n$ to \eqref{eq:SQG}--\eqref{eq:u} with $\alpha= \alpha_n$ on $\R^2 \times [-2, 0]$ such that $\theta_n \rightarrow \theta$ strongly in $L^p( \R^2 \times [-2, 0])\,$ and assume that $\theta$ is a classical solution to \eqref{eq:SQG}--\eqref{eq:u} on $ \R^2 \times [-2, 0] $. Then, there exists a universal constant $C>0$ such that uniformly for any $(x,t) \in Q_1$ 
\begin{equation}
\lim_{n \to \infty} E(\theta_{n,0}, u_{n,0}; \frac{1}{4} ) \leq C E(\theta_0, u_0; \frac{1}{4})
\end{equation}
where we denote by $\theta_{n,0}$ and $u_{n,0}$ (and $\theta_0$ and $u_0$ respectively) the change of variables of Lemma \ref{lem:changeofvar} as applied to $\theta_n$ (and $\theta$ respectively).
\end{lemma}

\begin{proof} We fix $(x,t) \in Q_1$ and apply the change of variables of Lemma \ref{lem:changeofvar} to $\theta_n$ and $\theta$ respectively. We denote the corresponding flow by $x_{n,0}$ and $x_0$. 
%
Moreover, we estimate for $s\in [-1,0]$ 
\begin{align}
| x_{n,0}(s) - & x_0(s) | \leq \lvert s \rvert^{1-\frac{1}{p}} \left( \int_{s}^0 \mean{B_\frac{1}{4}(x+x_{n,0}(\sigma))} \lvert u_n(y, \sigma+t)- u(y, \sigma+t) \rvert^p  \, dy \, d\sigma \right)^\frac{1}{p}\\ 
&+ \int_{s}^0 \mean{B_\frac{1}{4}(x)} \lvert u(y+x_{n,0}(\sigma), \sigma+t) - u(y+ x_0(\sigma), \sigma+t) \rvert\, dy \,d\sigma\\
&\leq 4^{\frac{2}{p}}  \lvert s \rvert^{1-\frac{1}{p}} \left(\int_{t-1}^t \int \lvert u_n-u \rvert^p  \, dy \, d\sigma \right)^\frac{1}{p} +  \sup_{\sigma \in [t-1,t]}[u(\sigma)]_{\rm Lip(\R^2)}\int_{s}^0 \lvert x_{n,0}(\sigma)- x_{0}(\sigma) \rvert \, d\sigma \,.
\end{align}
By Calderon-Zygmund, the strong convergence of $\theta_n$ implies that $u_n \rightarrow u$ strongly $L^p$.  Calling $C= \sup_{\sigma \in [t-1,t]}  [u(\sigma)]_{\rm Lip(\R^2)}$ we deduce by Gr\"onwall's inequality that uniformly in $s \in [-1,0]$ 
\begin{equation}\label{eq:gronwall}
\lim_{n\to \infty}\lvert x_{n,0}(s)-x_0(s) \rvert \leq \lim_{n\to \infty} 4^{\frac{2}{p}} 
(1+Ce^C)\left(\int_{t-1}^t \int \lvert u_n-u \rvert^p \, dy \, d\sigma \right)^\frac{1}{p}
=0\,.
\end{equation}
We now claim that $\theta_{n,0} \rightarrow \theta_0$ 
 strongly in $L^p(\R^2 \times [-(1/4)^{2\alpha}, 0])\,.$ Fix $\varepsilon >0 \, .$
We split as before
\begin{align}
\int_{-\left( \frac 14\right)^{2\alpha}}^0 \int \lvert \theta_{n,0}- \theta_0 \rvert^p (y,s) \, dy \, ds &\leq \int_{-\left( \frac 14\right)^{2\alpha}}^0 \int \lvert \theta_{n}(x_{n,0}(s)+y, s+t) - \theta(x_{n,0}(s)+y, s+t ) \rvert^p \, dy \, ds \\
&+ \int_{-\left( \frac 14\right)^{2\alpha}}^0 \int \lvert \theta(x_{n,0}(s)+y, s+t) - \theta(x_0(s)+y, s+t) \rvert^p \, dy \, ds \,.
\end{align}
Using the strong convergence of $\theta_n$ in $L^p(\R^2 \times [-2, 0])$, the first term on the right-hand side doesn't exceed $\frac \varepsilon 3$ for $n$ big enough. Using \eqref{eq:gronwall} and absolute continuity, there exists $R\geq 1$ such that for all $n$ big enough
\begin{equation}
\int_{-\left( \frac 14\right)^{2\alpha}}^0 \int_{\lvert y \rvert \geq R} \lvert \theta(x_{n,0}(s)+y, s+t) - \theta(x_0(s)+y, s+t) \rvert^p \, dy \, ds \leq \frac \varepsilon 3\,.
\end{equation}
By the regularity of $\theta$ and \eqref{eq:gronwall}, we estimate the remaining contribution of the integral over the set $\{\lvert y \rvert \leq R\}$ by $ [\theta]_{{\rm Lip}(\R^2 \times [t- (1/4)^{2\alpha}), t])}^p\norm{u_n-u}_{L^p(\R^2 \times [t-1,t])}^p  R^p\,,$ which by the strong $L^p$-convergence does not exceed $\frac \varepsilon 3$ for $n$ big enough.
The strong convergence of $\theta_{n,0} \rightarrow \theta_0$ in $L^p(\R^2 \times [-(1/4)^{2\alpha}, 0])$ implies also that  $u_{n,0} \rightarrow u_0$ strongly in $L^p(\R^2 \times [-(1/4)^{2\alpha}, 0])\,.$  
As an immediate consequence we obtain
\begin{equation}
\lim_{n \to \infty} E^S(\theta_{n,0}; \frac{1}{4}) + \lim_{n \to \infty} E^V(u_{n,0}; \frac{1}{4})= E^S(\theta_0; \frac{1}{4}) + E^V(u_0; \frac{1}{4})\,.
\end{equation}
and
\begin{align}
 \lim_{n \to \infty} E^{NL}(\theta_{n,0}; \frac{1}{4}) &\lesssim  \lim_{n \to \infty}  \Bigg\{ E^{NL}(\theta_0; \frac{1}{4}) + \left( \mean{-(1/4)^{2\alpha}}^0 \sup_{R \geq \frac{1}{4}} \frac{1}{(4R)^{\sigma p}} \mean{B_R} \rvert \theta_{n,0}- \theta_0 \rvert^p \right)^\frac{1}{p}  \\
  & \hspace{5em} + \left( \mean{-(1/4)^{2\alpha}}^0 \lvert [\theta_0]_{B_\frac{1}{4}}- [\theta_{n,0}]_{B_\frac{1}{4}} \rvert^p \right)^\frac{1}{p} \Bigg\}
 \\
 &= E^{NL}(\theta_0; \frac{1}{4})\,.
\end{align}
\end{proof}

\begin{proof}[Proof of  Corollary \ref{cor:stability}] We argue by contradiction and we let $p:=\frac{6}{4\alpha-1}$ and $\sigma := \frac{4\alpha}3+\frac16$ as in the proof of Corollary \ref{thm:Scheffer}. We may assume that there is a sequence of orders $\alpha_n \in (\frac{1}{4}, \frac{1}{2})$ and initial data $\bar \theta_{n,0} \in H^2$ such that
\begin{itemize}
\item $\lim_{n \to \infty} \alpha_n = \frac{1}{2} \,,$
\item $\norm{\bar \theta_{n, 0}}_{H^2(\R^2)} \leq R$ for all $n\geq 1\,,$
\item the local smooth solutions $\theta_n$ to \eqref{eq:SQG}--\eqref{eq:u} with $\alpha= \alpha_n$ and initial data $\bar \theta_{n,0}$, which exist on an interval $[0, T_1]$, with $T_1 = C \norm{\bar \theta_{n,0}}_{H^2}^{-2} \geq C R^{-2}\,$  bounded from below uniformly in $n$ (see \cite{CotiZelatiVicol2016}), blow-up in finite time.
\end{itemize}
{\new Since $ \theta_n \in L^\infty([0, T_1], H^2)$ implies, for instance, that  $\nabla \theta_n \in L^2([0, T_1], L^{4/\alpha})\,,$ $\theta_n$ satisifes the weak-strong uniqueness criterion of \cite{DongChen2006}\footnote{\new The authors of \cite{DongChen2006} state their result in the form of an asymptotic stability result with respect to perturbations of the initial datum and the right-hand side; however, in absence of any perturbation their energy yield the corresponding weak-strong uniqueness statement, as previously observed for instance in \cite{LiuJiaDong2012}.}} and hence $\theta_n$ coincides on $[0,T_1]$  with the unique suitable weak solution to \eqref{eq:SQG}--\eqref{eq:u} with $\alpha= \alpha_n$ and initial data $\bar \theta_{n,0}$. We have the uniform bound
\begin{equation}\label{eq:unifbound}
\sup_{n \geq 1} \norm{\theta_n}_{L^\infty(\R^2 \times [T_1/2, \infty))} + \sup_{n \geq 1}\norm{\theta_n}_{L^\infty([T_1/2, \infty), L^2(\R^2))} \leq C 
\end{equation}
given by \eqref{e:g_energy_1} and Theorem \ref{thm:LerayConstWu}\,. 
Up to subsequence, we have that $\bar \theta_{n,0} \rightarrow \bar \theta_0 \, $ strongly in $L^2(\R^2) \,.$ Fix $T>0\,.$ By \eqref{eq:unifbound}, $\theta_n$ is uniformly bounded in $L^2(\R^2 \times [T_1/2, T])$ and hence $\theta_n \rightharpoonup \theta$ weakly in $L^2\,.$ The strong convergence in $L^p(\R^2 \times [T_1/2,T])$ is established as in the proof of Lemma \ref{lem:compactness}, Step 2, using an Aubin-Lions type argument. The strong convergence for any $T>0$ is enough to pass to the limit the equation as well as the global energy inequalities \eqref{e:g_energy_1}--\eqref{e:g_energy_2} and hence $\theta$ is a Leray--Hopf solution to \eqref{eq:SQG}--\eqref{eq:u} with $\alpha= \frac 12$ and initial datum $\bar \theta_0$. This Leray--Hopf solution is smooth by \cite{CaffarelliVasseur2010}. 
The blow-up of the strong solutions means that there exists $(x_n, t_n) \in {\rm Sing} \, \theta_n\,$ for $n\geq 1\,.$ 
By Theorem~\ref{thm:main} (and more precisely, noticing that Step 1 in its proof can be made uniform in $n$), 
there exists $M>0$ such that for all $n\geq 1$ 
$$(x_n, t_n ) \in \overline B_M \times [T_1, M] \,.$$
Up to subsequence, we may assume that $(x_n, t_n) \rightarrow (\bar x, \bar t) \in \overline B_M \times [T_1, M]\,.$
Rescaling the sequence of solutions by one single factor (related to $T_1$) and translating them, we may assume that $T_1 =- 2$ and that $(\bar x, \bar t)=(0,0) \in Q_1\,.$ By a further $n$-dependent temporal translation (by $2t_n\to 0$), we may assume that $t_n<0$. By the continuity of translations in $L^p$, we still have in this way that $\theta_n \to \theta$ strongly in  $L^p(\R^2 \times [-2, 0])$.
We claim that there exists $r \in (0, 1]$ such that for any $(x,t) \in Q_1$ 
\begin{equation}\label{eq:smoothsolexcess}
E(\theta_{r,0}, u_{r,0}; \frac{1}{4}) < \frac{\varepsilon_2}{2} \,,
\end{equation}
where we denote by $\theta_{r,0}$ and $u_{r,0}$ the change of variables of Lemma \ref{lem:changeofvar} applied to $(\theta_r, u_r)\,$ (see Remark \ref{rem:changeofvar}) and  by $\varepsilon_2$ the constant from Theorem \ref{prop:ereg1}. This now gives rise to a contradiction: Indeed, by Lemma \ref{lem:continuity}, uniformly for every $(x,t) \in Q_1$, we have the lower semicontinuity $\lim_{n \to \infty} E((\theta_{n})_{r,0},( u_n)_{r,0}; \frac{1}{4}) \leq C E(\theta_{r,0}, u_{r,0}; \frac{1}{4}) $ such that for every $n $ big enough, the smallness requirement of Theorem \ref{prop:ereg1} holds and we deduce that $\theta_{n,r}$ is smooth in $Q_{1/2}$, namely $\theta_n$ is smooth in $Q_{r/2}\,.$ This contradicts the fact that the singular points $(x_n, t_n) \in Q_{r/2}$ for $n$ big enough. We are thus left to prove \eqref{eq:smoothsolexcess}. Indeed, 
\begin{align}
E^S(\theta_{r,0}; \frac{1}{4}) &= r^{2\alpha-1}\left( \mean{r^{2\alpha}t-\left(\frac{r}{4}\right)^{2\alpha}}^{r^{2\alpha}t} \mean{B_\frac{r}{4}(rx+rx_0(r^{2\alpha}(s-t))} \lvert \theta(y, s) - r^{1-2\alpha} (\theta_{r,0})_{Q_\frac{1}{4}}\rvert^p \, dy \,ds \right)^\frac{1}{p}\\
&\leq r^{4\alpha-1} [\theta]_{{\rm Lip(\R^2 \times [-(2r)^{2\alpha},0])}}\,,
\end{align}
where we used again, as in the proof of Lemma \ref{lem:continuity}, the fact that the flow $x_{0}$ is Lipschitz for a regular solution $\theta$. Similarly, 
$E^V(u_{r,0}; \frac{1}{4}) 
\leq  \, r^{2\alpha} [u]_{L^\infty ( [-(2r)^{2\alpha},0]; \rm Lip(\R^2))}\,.$
Finally for the non-local part of the excess, we rewrite
\begin{align}
\, &E^{NL}(\theta_{r,0}; \frac{1}{4}) \\
&= r^{2\alpha-1} \left( \mean{- \left(\frac{r}{4}\right)^{2\alpha}}^{0} \sup_{R \geq \frac{r}{16}} \left(\frac{r}{4R}\right)^{\sigma p} \left(\mean{B_R(x)} \lvert \theta(y+x_0(r^{2\alpha}s), s+r^{2\alpha}t)  - r^{1-2\alpha} [\theta_{r,0}]_{B_\frac{1}{4}} \rvert^\frac{3}{2} \, dy  \right)^\frac{2p}{3} \, ds \right)^\frac{1}{p} \,.
\end{align} 
For fixed time $s\in [-(r/4)^{2\alpha}, 0]\,,$ we estimate the supremum splitting it on the two sets $\{\frac{1}{4} \geq R \geq \frac{r}{16} \}$ and $\{ R \geq \frac{1}{4} \}$. We get that
\begin{equation*}
E^{NL}(\theta_{0, r}; \frac{1}{4}) \leq C 
r^{\sigma p}([\theta]_{{\rm Lip}(\R^2 \times [-(2r)^{2\alpha}, 0])}+\norm{\theta}_{L^\infty(\R^2 \times [-(2r)^{2\alpha}, 0])})\,.\qedhere
\end{equation*}
\end{proof}

\appendix 
\section{Local spacetime regularity of the fractional heat equation}

\begin{lemma}\label{lem:linearizedeq} Let $\alpha \in (0, 1)$ and $p\geq 2$. Consider $\eta \in L^{3/2}_{loc}(\R^2 \times [-1,0])$ with $(\eta)_{Q_1}=0$ and $E^S(\eta;1)+E^{NL}(\eta; 1)< +\infty$ which solves $\partial_t \eta + (-\Delta)^\alpha \eta = 0$ in $Q_{3/4}$. Then $\eta$ is smooth with respect to the space variable and $\eta \in C^{1-1/p}(Q_{1/2})$ in spacetime. Moreover, there exists $\bar C>1$ such that
\begin{equation}
\norm{\eta}_{L^\infty([-(1/2)^{2\alpha}, 0], C^1(B_\frac{1}{2}))} +\norm{\eta}_{C^{1-\frac{1}{p}}(Q_\frac{1}{2})} \leq \bar C ( E^S(\eta;1) + E^{NL}(\eta;1)) \, .
\end{equation} 
The constant $\bar C$ can be chosen uniform in $\alpha$ as long as $\alpha$ is bounded away from $0$.
\end{lemma}
\begin{proof} Using the linearity of the equation, we can assume by a standard regularization argument that $\eta\in C^\infty(\R^2 \times (-1, 0)) $. We multiply the equation by $\eta \varphi|_{y=0}$ where $\varphi$ is a smooth  cut-off between $Q_{11/16}^*$ and $Q_{3/4}^*\,$  with $\partial_y \varphi(\cdot, 0, \cdot)=0$ and obtain, arguing as in Lemma \ref{thm:existencesuitable},
\begin{align}
\int \eta^2(x,t) \varphi (x, 0,t) \, dx  + 2 c_\alpha \int_0^t \int y^b \lvert \nabla \eta^* \rvert^2 \varphi \, dx \, dy \,ds = &\int_0^t \int  \eta^2(x,s) \partial_t \varphi (x,0,s) \, dx \, ds \\ &+ c_\alpha \int_0^t \int y^b (\eta^*)^2\overline{\Delta}_b \varphi \, dx \, dy \, ds \, .
\end{align}
Taking the supremum over $t\in [-(11/16)^{2\alpha}, 0]$ and recalling the support of $\varphi$, we obtain by Lemma \ref{lem:tail} that
\begin{align}\label{eq:energyineqheat}
\sup_{t\in [-(11/16)^{2\alpha}, 0]} \int_{B_\frac{11}{16}} \eta^2(x,t)  \, dx &\leq C \left( \int_{Q_\frac{3}{4}} \eta^2 \, dx \, ds + \int_{Q_\frac{3}{4}^*} y^b \lvert \eta^* \rvert^2  \, dx \, dy \, ds \right) \nonumber \\
&\leq  C \left(E^S(\eta;1)^2+ E^{NL}(\eta;1)^2 \right)
\, .
\end{align}
Let now $p(x, t)$ be the fractional heat kernel on $\R^2\times (0, \infty)$ with explicit form $e^{-\lvert \xi \rvert^{2\alpha}t}$ in Fourier space. By scaling invariance, $p(x,t)= t^{-1/\alpha} p(t^{-1/(2\alpha)}x,1)$ and $p$ is $C^\infty$ and bounded for $t>0\,.$ Let now $P^*(x,y):=[p(\cdot, 1)]^*(x,y)$ be the Caffarelli-Silvestre extension to $\R^{3}_+$ of $x \mapsto p(x,1)$ and observe that from scaling, $p^*(x,y,t) = t^{-1/\alpha} P^*(t^{-1/(2\alpha)}x,t^{-1/(2\alpha)}y)\,.$ Fix a cut-off $\varphi$ between $B_{1/16}^*$ and $B_{1/8}^*$ which is radially symmetric in $x$ and $y$. We define $\tilde{p}(x,y,t):= p^*(x,y,t) \varphi(x,y)\, .$ We proceed as in \cite[Proposition 4.1]{Silvestre2012} to obtain for $(x,t)\in B_{5/8}$
\begin{align}
\eta(x,t)&= \int_{B_\frac{3}{4}} \eta( z, -(3/4)^{2\alpha}) p(x-z, t+(3/4)^{2\alpha}) \varphi(x-z,0) \, dz \\
&+ \int_{-(3/4)^{2\alpha}}^t \int_{B_\frac{3}{4}^*} y^b \eta^*(x,y,\tau) \overline{\Delta}_b \tilde{p}(x-z, y,t-\tau) \,dz \, dy \, d\tau \, .
\end{align}
We argue as in \cite[Proposition 4.1]{Silvestre2012} that $y^b\overline{\Delta}_b \tilde{p}(x,y,t) = \overline{\div}(y^b \overline{\nabla} \tilde{p}(x,y,t))$ is a smooth function in $x$ and $y$ supported in $B_{1/8}^* \setminus B_{1/16}*$ which remains bounded as $t \to 0\,.$ We conclude that for any multi-index $\beta$ with $\lvert \beta \rvert \geq 0$ we have, using \eqref{eq:energyineqheat} and Lemma \ref{lem:tail}, that
\begin{align}\label{eq:spaceregheat}
\norm{\partial_x^\beta \eta}_{L^\infty(Q_\frac{5}{8})} &\lesssim \norm{\eta}_{L^\infty L^2(Q_\frac{3}{4})} + \left( \int_{Q_\frac{3}{4}^*} y^b (\eta^*)^2 \, dx \, dy \, dt \right)^\frac{1}{2} \lesssim E^S(\eta;1) + E^{NL}(\eta; 1) \,.
\end{align} 
To get the spacetime regularity, we observe that for $x\in B_{1/2}$ we can write (for fixed time)
\begin{align}
&\lvert (-\Delta)^\alpha \eta(x,t) \rvert =  \int_{\lvert y \rvert \leq \frac{5}{8}} \frac{\lvert \eta(x,t)- \eta(y,t) \rvert}{\lvert x-y \rvert^{n+2\alpha}} \,dy +   \int_{\lvert y \rvert > \frac{5}{8}} \frac{\lvert \eta(x,t)- \eta(y,t)\rvert}{\lvert x-y \rvert^{n+2\alpha}} \,dy \\
&\lesssim [\eta(t)]_{\rm Lip(B_\frac{5}{8})} + \sum_{i\geq 0} \frac{1}{2^{{ (i-1)(2\alpha-\sigma)}}} \int_{B_{2^{i+1}}\setminus B_{2^i}} \frac{\lvert \eta(x,t)- \eta(y,t) \rvert}{2^{(i-1)(n+\sigma)}} \, dy + \int_{B_1 \setminus B_\frac{5}{8}}\lvert \eta(x,t)- \eta(y,t) \rvert \, dy \\
& \lesssim  [\eta(t)]_{\rm Lip(B_\frac{5}{8})}  + \sup_{R \geq 1} \frac{1}{R^\sigma} \mean{B_R} \lvert \eta(x, t) - [\eta(t)]_1 \rvert \, dx   + \int_{B_1} \lvert \eta(x,t) \rvert \,dx \,.
\end{align}
Using \eqref{eq:spaceregheat} and the equation, we conclude
\begin{equation}
\norm{\partial_t \eta}_{L^p L^\infty(Q_\frac{1}{2})} \lesssim E^S(\eta;1) + E^{NL}(\eta;1) \,,
\end{equation}
which proves that $\eta \in C^{1-\frac{1}{p}}([-(1/2)^{2\alpha}, 0], L^\infty(B_\frac{1}{2}))$ recalling that $W^{1,p}(\R) \hookrightarrow C^{1- \frac{1}{p}}(\R) \, .$ 
\end{proof}

\section{$C^{\delta}$-H\"older continuous solutions are classical for $\delta >1-2\alpha$}
{In \cite{ConstantinWu2008} it is proved that solutions of \eqref{eq:SQG}--\eqref{eq:u}  with $u \in L^\infty([0,T], C^\delta(\R^2))$, $\delta>1-2\alpha$, are smooth. The following Lemma provides a localized version of this result.}

\begin{lemma}\label{lem:hoeldertosmooth} Let $\theta: \R^2 \times (-1, 0] \rightarrow \R$ be a bounded solution of \eqref{eq:SQG}--\eqref{eq:u}. If $\theta \in L^\infty((-1, 0], C^{\delta}(B_2))$ for some $\delta > 1-2\alpha$, then $\theta\in C^{1, \delta-(1-2\alpha)}(B_{1/2 } \times (-1/2, 0])$ in spacetime and in particular, is a classical solution.
\end{lemma}
\begin{proof}
\textit{Step 1: We show that $\theta \in L^\infty([-1/2, 0], C^{1, \delta-(1-2\alpha)}(B_{1/2})) \, .$} \\This follows from showing that $u \in L^\infty((-1, 0], C^{ \delta}(B_1))$ and the general result on fractional advection-diffusion equations \cite[Theorem 1.1]{Silvestre2012}. Let us write $u= u_1+u_2$, where $u_1= \mathcal{R}^\perp(\theta \chi)$ and $u_2 = \mathcal{R}^\perp(\theta(1-\chi))$ for $\chi$ a smooth cut-off in space between $B_{3/2}$ and $B_2\, .$ We estimate by Schauder estimates \cite[Proposition 2.8]{Silvestre2007}
\begin{equation}
\norm{u_1}_{L^\infty((-1, 0], C^{\delta})} \leq C \norm{\theta \chi}_{L^\infty((-1, 0], C^{\delta})} \leq C \norm{\theta}_{L^\infty((-1, 0], C^{\delta}(B_2))} \, .
\end{equation}
Regarding $u_2$, we observe that $[u_2(t)]_{C^k}$ is bounded, uniformly in $t\in (-1,0]$, for all $k\geq 1$. Indeed, consider for instance $k=1$: For $x \in B_1$ and fixed time, by integration by parts (the boundary term at infinity vanishes using the uniform boundedness of $\theta$)
\begin{align}
\partial_j \mathcal{R}_i u_2(x) &= c \int \frac{x_i-z_i}{\lvert x-z \rvert^3} \partial_j ((1-\chi(z)) \theta(z) \, dz 
 &= - c \int_{\lvert z \rvert \geq \frac{3}{2} } \partial_{z_j} \left( \frac{x_i-z_i}{\lvert x-z \rvert^3} \right) (1- \chi(z)) \theta(z) \, dz \,,
\end{align}
so that, using that for $x \in B_1$ and $z \in B_{3/2}^c$ we have $\lvert x - z \rvert \geq \frac{1}{2}$, we have
\begin{equation}
\lvert \partial_j \mathcal{R}_i u_2(x)  \rvert \leq C \norm{\theta}_{L^\infty}\int_{\lvert x-z \rvert \geq \frac{1}{2}} \frac{1}{\lvert x-z \rvert^3} \, dz < + \infty \, .
\end{equation}
We have obtained that 
\begin{equation}
\norm{u}_{L^\infty((-1, 0], C^{ \delta}(B_1))} \leq C (\norm{\theta}_{L^\infty((-1, 0], C^{ \delta}(B_2))}+ \norm{\theta}_{L^\infty(\R^2 \times (-1, 0])}) \, .
\end{equation}
\textit{Step 2: We show that $\theta \in C^{1, \delta-(1-2\alpha)}((-1/2, 0], L^\infty(B_{1/2}))\, .$}\\
Observe that $\theta$ solves a heat equation with right-hand side $u \cdot \nabla \theta \in L^\infty((-1/2,0], C^{\delta-(1-2\alpha)}(B_{1/2}))$, so that
\begin{equation}
\partial_t \theta \in  L^\infty((-1/2,0], C^{\delta-(1-2\alpha)}(B_{1/2})) \, .
\end{equation}
In particular, $\theta \in {\rm Lip}((-1/2, 0], C^{\delta-(1-2\alpha)}(B_{1/2})) \, .$ Repeating the argument, we obtain that $\theta \in C^{1, \delta-(1-2\alpha)}(B_{1/2} \times (-1/2, 0]) \, .$
Higher regularity then follows from energy estimates.
\end{proof}

\section{Existence of suitable weak solutions}\label{a:existencesws}
\begin{proof}[Proof of Theorem ~\ref{thm:existencesuitable}]
Fix $\theta_0 \in L^2(\R^2)$. For $\epsilon>0$, we consider the system with added vanishing viscosity term
\begin{equation}\label{eq:vanishingvis}
\begin{cases} \partial_t \theta + u\cdot \nabla \theta + (-\Delta)^\alpha \theta = \epsilon \Delta \theta \\
u = \nabla^\perp (-\Delta)^{-\frac{1}{2}} \theta \,,
\end{cases}
\end{equation}
complemented with the initial datum $\theta(\cdot, 0) = \theta_0\, .$ For any $\epsilon>0$, the system \eqref{eq:vanishingvis} admits a global smooth solution $\theta_\epsilon:\R^2\times (0, \infty) \to \R$. Moreover, for any $t>0$, $\theta_\epsilon(\cdot, t) \in L^2(\R^2)$ and for any $0 \leq s< t$,  we have the energy equality 
\begin{equation}\label{eq:energyeq}
\int \theta_\epsilon^2(x,t) \, dx + 2 \int_s^t \int \left[\lvert (-\Delta)^\frac{\alpha}{2}\theta_\epsilon\rvert^2(x, \tau) + \epsilon \lvert \nabla \theta_\epsilon\rvert^2 (x,\tau)  \right]\, dx \, d\tau = \int \theta_\epsilon^2(x,s) \, dx \, . 
\end{equation}
Theorem \ref{thm:LerayConstWu} also applies to \eqref{eq:vanishingvis}, so that  $\theta_\epsilon$ is in $L^\infty$ for $t>0$ with the uniform-in-$\epsilon$ bound
\begin{equation}\label{eq:linftyvv}
\norm{\theta_\epsilon(t)}_{L^\infty} \leq Ct^{-\frac{1}{2\alpha}} \norm{\theta_0}_{L^2} \, .
\end{equation}
Finally, with the obvious modifications of the computation in Section \ref{subsec:energyineqsmooth}, we have for any nonnegative test function $\varphi \in C^\infty_c(\R^3_+)$, locally constant in $y$ in a neighbourhood of $y=0$, any nonnegative and convex $f\in C^2(\R)$ and any $t>0$ that
	\begin{align}
&\int_{\R^2} \varphi (x,0,t) f(\theta_\epsilon) (x,t) \, dx +c_{\alpha}\int_0^t \int_{\R^3_+} y^b|\overline{ \nabla}  \theta_\epsilon^* |^2 f''(\theta_\epsilon^*)\varphi \, dx \, dy \, ds 
\\
&\leq \int_0^t \int_{\R^2} \left[f( \theta_\epsilon) (\partial_t \varphi|_{y=0}+ \epsilon \Delta \varphi|_{y=0}) + u_\epsilon  f( \theta_\epsilon) \cdot \nabla \varphi|_{y=0} \right]\, dx \, ds \nonumber + c_{\alpha} \int_0^t\int_{\R^3_+} y^b f( \theta_\epsilon^*) \overline\Delta_b \varphi \, dx \, dy \, ds \, \\
&=: C(\epsilon) + D(\epsilon) \, .
	\end{align}
We want to pass to the limit $\epsilon \to 0$. From \eqref{eq:energyeq} with $s=0$ and the Sobolev embedding of $\dot{W}^{\alpha, 2}(\R^2) \hookrightarrow L^\frac{2}{1-\alpha}(\R^2)$, we infer by interpolation that the family $\{\theta_\epsilon \}_{\epsilon>0}$ is uniformly bounded in $$L^\infty([0, \infty), L^2) \cap L^2([0, \infty), \dot{W}^{\alpha,2}) \hookrightarrow L^{2(1+\alpha)}(\R^2 \times [0, \infty))\, .$$
By Banach-Anaoglu, for any fixed time $T>0$, there exists $\theta \in L^2(\R^2 \times [0, T])$ and a subsequence $\epsilon_k \to 0$ such that $\theta_{\epsilon_k} \rightharpoonup \theta$ weakly in $L^2(\R^2 \times [0, T])$. We now claim that this convergence is in fact strong via an Aubin-Lions type argument in the same spirit as Step 2 of the proof of Lemma \ref{lem:compactness}. Fix $\eta >0$ and a family of mollifiers $\{\phi_\delta \}_{\delta \geq 0} \subseteq C^\infty_c(\R^2)$ in space. For $k,j \geq 1$ we estimate
\begin{equation}\label{eqn:triang}
\norm{\theta_{\epsilon_j}- \theta_{\epsilon_k}}_{L^2(\R^2 \times [0, T])} \leq \norm{\theta_{\epsilon_j} - \theta_{\epsilon_j} \ast \phi_\delta}_{L^2} + \norm{\theta_{\epsilon_k} - \theta_{\epsilon_k} \ast \phi_\delta}_{L^2}  + \norm{(\theta_{\epsilon_j}-\theta_{\epsilon_k}) \ast \phi_\delta}_{L^2} \, .
\end{equation}
The first two contributions converge to $0$ independently of $k$ and $j$ due to a bound of the form $\norm{\theta_{\epsilon_k}-\theta_{\epsilon_k} \ast \phi_\delta}_{L^2(\R^2 \times [0, T])}^2\leq C \delta^{2\alpha}$ obtained as in \eqref{eqn:tocopy} with $\eta_k$ replaced by $\theta_{\eps_k}$. 
We now choose $\delta$ small enough such that this contribution does not exceed $\frac{\eta}{3}\, .$ Having $\delta$ fixed, 
we  claim that the family of curves $\{ t \mapsto \theta_{\epsilon_k}(\cdot, t) \}_{k \geq 1}\, $  is equicontinuous and equibounded with values in $W^{1, \infty}$. Indeed, by the energy equality \eqref{eq:energyeq} with $s=0$ and the Calderon-Zygmund estimate $\norm{u_{\epsilon_k}}_{L^2(\R^2 \times [0, T])} \leq C\norm{\theta_{\epsilon_k}}_{L^2(\R^2 \times [0, T])}$, we estimate
\begin{align}
\norm{\partial_t \theta_{\epsilon_k} \ast \phi_{\delta}}_{L^2([0, T],W^{1, \infty})} &\leq \norm{\div(u_{\epsilon_k} \theta_{\epsilon_k}) \ast \phi_\delta}_{L^2 W^{1, \infty}}  + \norm{\theta_{\epsilon_k} \ast (-\Delta)^\alpha \phi_\delta}_{L^2 W^{1, \infty}}  + \epsilon_k \norm{\theta_{\epsilon_k} \ast \Delta \phi_\delta}_{L^2 W^{1, \infty}} \\
&\leq \left(\norm{u_{\epsilon_k}}_{L^2(\R^2 \times [0, T])} \norm{\theta_0}_{L^2} +2\norm{\theta_{\epsilon_k}}_{L^2(\R^2 \times [0, T])} \right) \norm{\phi_\delta}_{W^{3, \infty}} \\
&\leq C(\delta) \,.
\end{align}
By Ascoli-Arzela  the sequence $\{\theta_k \ast \phi_\delta \}_{k \geq 1}$ converges uniformly on $\R^2 \times [0, T]\,$ and by uniqueness of limits, we infer that this limit must coincide with $\theta \ast \phi_\delta\,.$  We can therefore choose $N\geq 1$ big enough such that for any $k, j \geq N$ we have $\norm{(\theta_{\epsilon_j}- \theta_{\epsilon_k}) \ast \phi_\delta}_{L^2(\R^2 \times [0,T])} \leq \frac{\eta}{3} \, .$ We conclude that for $k, j \geq N$ there holds
$\norm{\theta_{\epsilon_k}- \theta_{\epsilon_j}}_{L^2(\R^2 \times [0, T])} \leq \eta \, .$
Since $\eta$ was arbitrary, we conclude by uniqueness of limits that $\theta_{\epsilon_k} \rightarrow \theta$ strongly in $L^2(\R^2 \times [0, T]) \, .$ By the uniform boundedness in $L^{2(1+\alpha)}(\R^2 \times [0, \infty))$ and by \eqref{eq:linftyvv} we also deduce that $\theta_k \rightarrow \theta$ strongly in $L^r(\R^2 \times [0, T])$ for any $2 \leq r < 2(1+\alpha)$ and strongly in $L^r(\R^2 \times [\tau, T])$ for any $\tau>0$ and any $2 \leq r < \infty$. By Calderon-Zygmund, we infer that $u_\epsilon \to u:= \mathcal{R}^\perp \theta$ strongly in $L^2(\R^2\times[0, T])$ (and $L^r$ respectively). Passing to the limit $k \to \infty$ in the equation \eqref{eq:vanishingvis}, we infer that $\theta$ is a distributional solution to \eqref{eq:SQG}--\eqref{eq:u}. We are left to pass to the limit in the global and local energy (in-)equality. Consider first \eqref{eq:energyeq}. By Banach Anaoglu and uniqueness of limit $(-\Delta)^\frac{\alpha}{2} \theta_{\epsilon_k} \rightharpoonup (-\Delta)^\frac{\alpha}{2} \theta$ weakly in $L^2(\R^2 \times [0, T])$ and by weak lower semicontinuity 
\begin{equation}
\int_s^t \int \lvert (-\Delta)^\frac{\alpha}{2} \theta \rvert^2(x, \tau) \, dx \, d \tau \leq \liminf_{k\to \infty}\int_s^t \int \lvert (-\Delta)^\frac{\alpha}{2} \theta_{\epsilon_k} \rvert^2(x, \tau) \, dx \, d \tau
\end{equation}
for any $0 \leq s < t \,.$ For almost every $t \in [0,T]$ we can extract a further subsequence such that $\theta_{\epsilon_k}(\cdot, t) \rightarrow \theta(\cdot, t)$ strongly in $L^2(\R^2)$. By passing to the limit in \eqref{eq:energyeq}, we thereby obtain \eqref{e:g_energy_1} and \eqref{e:g_energy_2} 
for almost every $0<s<t\, .$ We obtain it for every $t>0$ (and almost every $0<s<t$) by observing that up to changing $\theta$ on a set of measure $0$, we may assume that $\theta$ is continuous with respect to the weak topology on $L^2(\R^2) \,.$ We are left to pass to the limit in the localized energy inequality for $f(x)= \frac{(x-M)^2}{L^2}$ and $f(x)= \lvert \frac{x-M}{L} \rvert^q$ for $q >2$ respectively, with some $L>0$ and $M \in \R$. Let us denote $\eta_k := (\theta_{\epsilon_k}-M)/L$, $\eta:= (\theta-M)/L\,$ and let us fix $\tau, R >0$ such that $\supp \varphi \subseteq B_R(0) \times [0, R] \times [\tau, T]\,.$ From the strong convergence established before, we infer that $\eta_k \to \eta$ strongly in $L^r_{loc}(\R^2 \times [ \tau, T])\, $ for $2 \leq r < \infty \, .$ Up to extracting a further subsequence and a diagonal argument, we obtain that $\eta_k(t) \rightarrow \eta(t)$ strongly in $L^r_{loc}(\R^2)$ for almost every $t>0$ and any $r \in \N_{\geq 2}$. By interpolation, the former statement holds in fact for every $2 \leq r < \infty\,.$  We deduce that for almost every $t>0\,,$ for $q=2$ and any $q\geq 4$
\begin{align}
\lim_{k \to \infty} &\int_{\R^2} \varphi (x,0,t) f( \theta_{\epsilon_k})^2(x,t) \, dx   =\lim_{k \to \infty} \int_{\R^2} \varphi (x,0,t) \lvert \eta_k \rvert^q(x,t) \, dx   =  \int_{\R^2} \varphi (x,0,t) \lvert \eta \rvert^q(x,t) \, dx \,,  \\
\lim_{k \to \infty} &C(\epsilon_k) = \int_0^t \int_{\R^2} \left[\lvert \eta \rvert^q \partial_t \varphi|_{y=0} + u \lvert \eta \rvert^q \cdot \nabla \varphi|_{y=0} \right] \, dx \, ds \,.
\end{align}
We recall that from the Poisson formula $
\theta_{\epsilon_k}^*(x,y,t) =( P(\cdot,y) \ast \theta_{\epsilon_k}(\cdot, t))(x)$
and Young's convolution inequality
\begin{equation}
\norm{\theta_{\epsilon_k}^*}_{L^q(\R^2 \times [0, R] \times [\tau, T],\, y^b)} \leq C(R) \norm{P(1, \cdot)}_{L^1} \norm{\theta_{\epsilon_k}}_{L^q(\R^2 \times [\tau, T])}= C(R, \alpha)\norm{\theta_{\epsilon_k}}_{L^q(\R^2 \times [\tau, T])}   \,,
\end{equation}
where we used that $\norm{P(\cdot, y)}_{L^1}= \norm{P(\cdot, 1)}_{L^1}$ for $y>0 \, .$
We deduce that $\theta_{\epsilon_k}^* \rightarrow \theta^*$ strongly in $L^q(\R^2\times [0, R] \times [\tau, T], y^b)\,.$ By linearity, $\eta_k^* = \frac{\theta_{\epsilon_k}^*-M}{L}$ so that $\eta_k^* \rightarrow \eta^*$ strongly in $L^q(B_R \times [0, R] \times [\tau, T], y^b )\,.$  Hence
\begin{equation}
\lim_{k \to \infty} D(\epsilon_k) = \int_0^t \int_{\R^3_+} y^b \lvert \eta^*\rvert^q  \overline{\Delta}_b \varphi \, dx \, dy \, ds \,.
\end{equation}
Moreover, we also deduce $\overline{\nabla} \eta_k^* \rightharpoonup \overline{\nabla} \eta^*$ and$\overline{\nabla} \lvert \eta_k^* \rvert^\frac{q}{2} \rightharpoonup \nabla \lvert \eta^* \rvert^\frac{q}{2} $ weakly in $L^2(B_R(0) \times [0, R] \times [\tau, T], y^b)\,$ and we infer  by weak lower semicontinuity and the positivity of $\varphi$ that
\begin{align}
\int_0^t \int_{\R^3_+} y^b \lvert \overline{\nabla} \eta^* \rvert^2 \varphi \, dx \, dy \, ds &\leq \liminf_{k \to \infty}  \int_0^t \int_{\R^3_+}y^b \lvert \overline{\nabla} \eta_k^* \rvert^2 \varphi \, dx \, dy \,ds \\
\int_0^t \int_{\R^3_+} y^b \lvert \overline{\nabla} \lvert \eta^* \rvert^\frac{q}{2} \rvert^2 \varphi \, dx \, dy \, ds &\leq \liminf_{k \to \infty}  \int_0^t \int_{\R^3_+}y^b \lvert \overline{\nabla} \lvert\eta_k^* \rvert^\frac{q}{2} \rvert^2 \varphi \, dx \, dy \,ds
\end{align}
for any $t >0 \,  .$ Passing to the limit $k \to \infty$, we obtain \eqref{eqn:suit-weak-tested2} and \eqref{eqn:suit-weak-tested} for almost every $t>0 \, .$ 
\end{proof}

\bigskip
\textbf{ Acknowledgements}. The authors thank the anonymous referee for pointing out a mistake in a previous version of the manuscript, which led to a substantial improvement in the paper. The authors have been supported by
the SNF Grant 182565 ``Regularity issues for the Navier-Stokes equations and for other variational problems".
MC has been supported by the NSF under Grant No. DMS-1638352. Both authors acknowledge gratefully the hospitality of the Institute for Advanced Studies, where part of this work was done. 

\bibliographystyle{plain}
\bibliography{SQG}
\end{document}